\documentclass[11pt]{article} 
\usepackage{authblk}
\usepackage[utf8]{inputenc} 

\usepackage{geometry} 
\usepackage{graphicx} 
\usepackage{color}

\usepackage{booktabs} 
\usepackage{array} 
\usepackage{paralist} 
\usepackage{verbatim} 
\usepackage{subfig} 

\usepackage{amsfonts}
\usepackage{amsmath}
\usepackage{amsthm}
\usepackage{amssymb}
\usepackage{mathtools}

\usepackage{fancyhdr} 
\usepackage{sectsty}
\allsectionsfont{\sffamily\mdseries\upshape} 

\usepackage[nottoc,notlof,notlot]{tocbibind} 
\usepackage[titles,subfigure]{tocloft} 

\usepackage{amsthm}

\usepackage{MnSymbol}

\numberwithin{equation}{section}

\newtheorem{theorem}{Theorem}[subsection]
\newtheorem{definition}{Definition}[subsection]
\newtheorem{remark}[theorem]{Remark}

\newtheorem{example*}{Example\textsuperscript{*}}
\newtheorem{proposition*}{Proposition\textsuperscript{*}}
\newtheorem{corollary}[theorem]{Corollary}
\newtheorem{corollary*}{Corollary\textsuperscript{*}}

\newtheorem{proposition}[theorem]{Proposition}
\newtheorem{lemma}[theorem]{Lemma}
\newtheorem{condition}[theorem]{Condition}

\usepackage{fancyhdr}
\fancypagestyle{plain}{}
\def\star{\textsuperscript{*} }

\def\Limes#1#2 {\lim\limits_{#1\rightarrow #2}}
\renewcommand{\Re}{\operatorname{Re}}

\def\eps{\epsilon}

\def\R{\mathbb{R}}

\def\T{\mathbb{T}}

\def\N{\mathbb{N}}

\def\XXint#1#2#3{{\setbox0=\hbox{$#1{#2#3}{\int}$ }
\vcenter{\hbox{$#2#3$ }}\kern-.59\wd0}}

\DeclareMathOperator{\dist}{dist}
\DeclareMathOperator{\curl}{curl}

\def\norm#1{\left\lVert #1 \right\rVert}
\def\scalar#1#2{\langle #1,#2 \rangle}
\def\de{\partial}
\renewcommand{\div}{\operatorname{div}}

\def\dx{\,\mathrm{d}x}
\def\dr{\mathrm{d}r}
\def\dz{\mathrm{d}z}

\def\dy{\mathrm{d}y}
\def\dt{\mathrm{d}t}
\def\ds{\mathrm{d}s}
\def\dr{\mathrm{d}r}

\newcommand{\mres}{\mathbin{\vrule height 1.6ex depth 0pt width
0.13ex\vrule height 0.13ex depth 0pt width 1.3ex}}

\def\eq{\,=\,}

\DeclareFontFamily{U}{mathx}{}
\DeclareFontShape{U}{mathx}{m}{n}{<-> mathx10}{}
\DeclareSymbolFont{mathx}{U}{mathx}{m}{n}
\DeclareMathAccent{\widehat}{0}{mathx}{"70}
\DeclareMathAccent{\widecheck}{0}{mathx}{"71}

\def\chek#1{\widecheck{#1}}

\usepackage[square]{natbib}
\setcitestyle{numbers}

\usepackage{hyperref}

\title{A model for the approximation of vortex rings by almost rigid bodies}
\author{David Meyer}
\affil{Institut f\"ur Analysis und Numerik, Universit\"at M\"unster, M\"unster, Germany

dmeyer2@uni-muenster.de}
\date{\vspace{-6ex}}

\begin{document}

\maketitle

\renewcommand{\thefootnote}{\fnsymbol{footnote}} 
\footnotetext{\today

}

\begin{abstract}
We consider a model that approximates vortex rings in the axisymmetric 3D Euler equation by the movement of almost rigid bodies described by Newtonian mechanics. We assume that the bodies have a circular cross-section and that the fluid is irrotational and interacts with the bodies through the pressure exerted at the boundary. We show that this kind of system can be described through an ODE in the positions of the bodies and that in the limit, where the bodies shrink to massless filaments, the system converges to an ODE system similar to the point vortex system. In particular, we can show that in a suitable set-up, the bodies perform a leapfrogging motion.
\end{abstract}

\section{Introduction and main results}
Consider the three-dimensional Euler equation \begin{align}
&\de_t u+(u\cdot\nabla)u+\nabla p\eq 0\\
&\div u\eq 0.
\end{align}

We are interested in axisymmetric solutions with no swirl, that is $u$ of the form \begin{align}
u\eq(u_r(r,z,t),u_z(r,z,t),0)\label{no swirl}
\end{align}
in standard axisymmetric coordinates $(r,z,\theta)$, defined through $x=(r\cos\theta,r\sin\theta,z)$ and $r>0$, $\theta\in [0,2\pi)$. The vorticity then takes the form \begin{align}
\omega\eq\curl u\eq (\de_zu_r-\de_ru_z)e_\theta
\end{align}

(note that the sign is different compared to the two-dimensional vorticity).

A vortex ring is an axisymmetric solution where the vorticity is concentrated on some torus. These occur in real life for instance as smoke or bubble rings.

In axisymmetric coordinates, the equation for the vorticity reads as \begin{align}
&\de_t\left(\frac{\omega}{r}\right)+u\cdot\nabla\left(\frac{\omega}{r}\right)\eq 0\\
&\div(ru)\eq 0\\
&\curl u\eq \omega.
\end{align}

The velocity can be recovered from the vorticity through the Biot-Savart law: \begin{align}
u(x)\eq \int \frac{y-x}{4\pi|y-x|^3}\times\omega(y)\dy.
\end{align}

Asymptotically, for rings of inner radius $\eps$, with center around $(R,Z)$ and total vorticity $\Gamma$, one has in axisymmetric coordinates \begin{align}
u(x)\,\approx\, -\frac{\Gamma\left(x-\binom{R}{Z}\right)^\perp}{2\pi\left|x-\binom{R}{Z}\right|^2}+\frac{\Gamma\left|\log \left|x-\binom{R}{Z}\right|\right|}{4\pi R}e_z+O(1)\label{asymptotic velocity}
\end{align}

(here ``$\vphantom{a}^\perp$'' is defined in axisymmetric coordinates as the linear map with $e_r^\perp=e_z$ and $e_z^\perp=-e_r$).

As the relative vorticity $\frac{\omega}{r}$ is transported, one would expect that asymptotically, only the second term changes the position of a ring, while the main interaction of two different rings is given through the first term. Hence in the asymptotic limit, one would expect that multiple rings of inner radius $\eps$ with centers at $q_i=\binom{R_i}{Z_i}$ and strengths $\Gamma_i$ are well approximated through the system \begin{align}\label{point vortex system}
\binom{R_i}{Z_i}'\eq\sum_{j\neq i}\frac{\Gamma_j(q_j-q_i)^\perp}{2\pi|q_i-q_j|^2}+\frac{\Gamma_i|\log\eps|}{4\pi R_i}e_Z.
\end{align}

Note that in this system, the second term goes to $\infty$, while the first one is bounded, unless the pairwise distances go to $0$.

The rigorous justification of this system in such a regime where the interaction between different rings and the self-interaction of each ring are of the same order is currently an open problem. We will derive this system in two regimes in a simpler model. 

Roughly speaking, our model consists of replacing the vortex rings with ``almost rigid'' toroidal bodies with a circular crosssection, immersed in an axisymmetric fluid with no swirl and no vorticity, which is described by the Euler equations. However, we do allow for non-zero circulations around the bodies. These circulations will take the role of the vorticity in the limit.
This model is simpler than the full Euler equations, as the crosssections are fixed and the possible instability of the rings is not an issue. The system does however still feature the critical interaction between the rings.

To get a nontrivial motion of the bodies, one needs to allow them to change their shape because otherwise, they can only move in one direction, which does not allow for any nontrivial dynamics. For this, we allow the bodies to change the ratio between their outer and inner radius. We assume that the bodies interact with the fluid only through the pressure exerted at the boundary. We will describe this model in more detail in Sections \ref{StateSpace} and \ref{Main results}.\\

In Section \ref{Section2} we will then be able to describe the fluid velocity through potentials and streamfunctions, which are completely determined by the positions and velocities of the bodies. As a result, we can reduce the system to a second order ODE, with coefficients determined by the potentials and streamfunctions, for the positions of the bodies. This lets us show that our system is well-posed.

The proof of the convergence of the system for shrinking bodies then requires analyzing the limit of this system. For this, we study the asymptotics of the potentials and the streamfunctions in Sections \ref{Section3} and \ref{Section4}. We will be able to show that these converge to the corresponding two-dimensional functions in the zero radius limit and that the interaction of the different parts of the stream functions produces the same terms as the Biot-Savart law in the limit.\\

In Section \ref{Section5}, we will study the convergence in the zero radius limit of the ODE, which is quite intricate, as the equation degenerates due to the vanishing mass. In order to still get estimates, we use a modulated version of the kinetic energy of the system, whose evolution only depends on the degenerating terms, which allows one to obtain uniform estimates on the velocity and to pass to the limit.\\

In a two-dimensional setting, the similar convergence of fluid-body systems with shrinking bodies to point vortex systems in a bounded domain has been studied in \cite{GlassMunnierSueur1} and \cite{GlassSueur3}, while further work on fluid-body systems has been done e.g.\ in \cite{GlassMunnierSueur2}, \cite{HouotMunnier4}, \cite{Munnier5} and \cite{glass2015uniqueness}, see also the references therein. The simpler problem of a stationary shrinking obstacle has been studied e.g.\ in \cite{iftimie2003two,he2019small,iftimie2009incompressible}.

Somewhat similar problems for filaments immersed in 3D Stokes flow have been considered e.g.\ in \cite{mori2020theoretical,MR4373666}.

The stability of vortex rings and their approximation by the System \eqref{point vortex system} has e.g.\ been studied in \cite{ButtaMarchioro,butta2022vanishing,benedetto2000motion}. Special solutions which behave like the System \eqref{point vortex system} have been constructed in \cite{Davila2022leapfrogging} and in \cite{JerrardSmets} for the similar Gross-Pitiaevskii equation. Non-axisymmetric vortex filaments have been considered e.g.\ in \cite{jerrard2017vortex}.

\subsection{The model}\label{StateSpace}

 We use axisymmetric coordinates $(r,z)$ and denote the right halfplane by $\mathbb{H}$. 
Fix some $k\in \N_{>0}$ and numbers $v_1,\dots, v_k>0$, which we interpret as the volumes of the bodies and which should remain constant along the evolution. We set \begin{align}
\rho_i\,:=\,\sqrt{\frac{v_i}{\pi R_i}}\qquad B_i(R_i,Z_i)\,:=\,B_{\rho_i}((R_i,Z_i))\subset\mathbb{H} \text{ for $i\eq1,\dots,k$.}\label{def rho}
\end{align}

With respect to the measure $r\dr\dz$, which corresponds to the three-dimensional volume, the $B_i$ then have the fixed volume $v_i=\pi\rho_i^2R_i$. See also the figure below.

Formally, we describe the configuration of the bodies through the manifold $M\subset (\R^2)^k$ of all $(R_1,Z_1)\dots (R_k,Z_k)$ such that the bodies $B_i$ all have positive distance from each other and from $\de\mathbb{H}$. 

For $q\in M$, we shall write $q_i$ for the $i$-th component and $q_{R_i}$ and $q_{Z_i}$ for the two components of $q_i$. We shall also write $B_i(q)$ for clarity instead of $B_i$ sometimes. Let $n$ denote the outer normal of $\bigcup_i\de B_i$.

\begin{figure}
  \centering
  \includegraphics[scale=0.3]{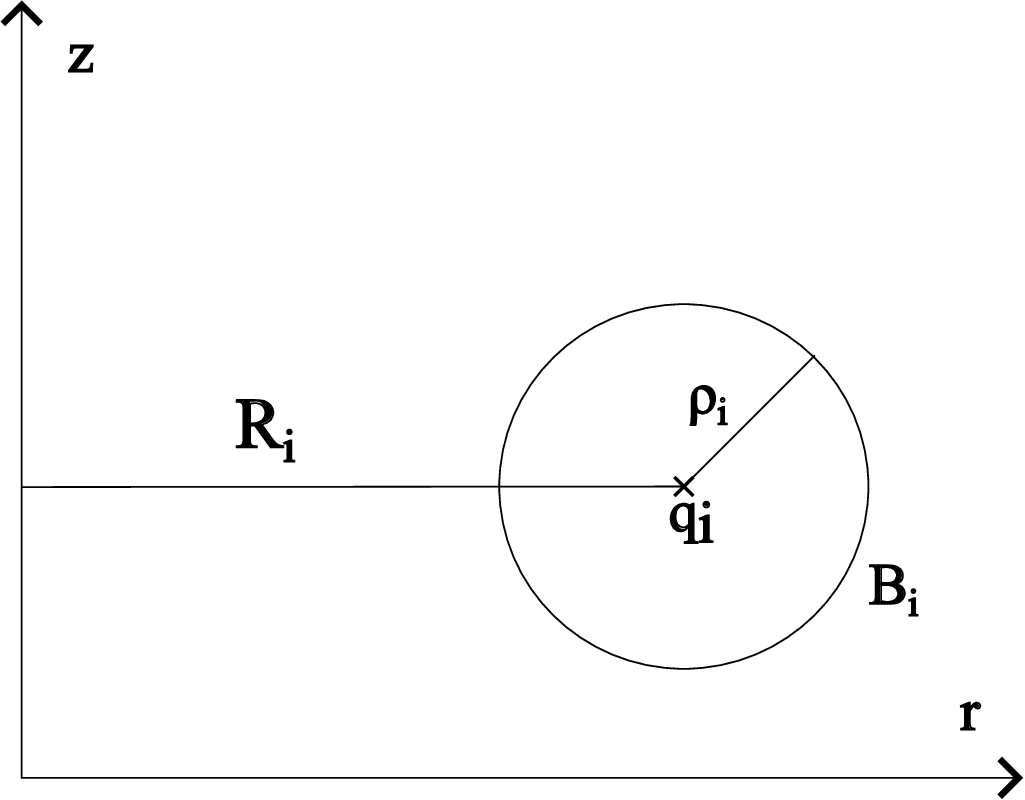}
\end{figure}

In this setup, there is a natural correspondence between tangent vectors and normal velocities on the $\de B_i$, as every $C^1$-curve in $M$ corresponds to a continuous movement of each $B_i$.

For a tangent vector $t^*$, let $t_{R_1}^*,t_{Z_1}^*,t_{R_2}^*,\dots$ denote its components. 
We say that a tangent vector is associated to $B_i$ if only its $R_i$- and $Z_i$-component are non-zero and write $T_{q_i}M$ for the subspace of those tangent vectors. 

Then for a $C^1$-curve $q$, the normal velocity is given by \begin{align}
u(\dot{q})\,:=\,\dot{q}_{Z_i}n\cdot e_Z+\dot{q}_{R_i}n\cdot e_R-\frac{\dot{q}_{R_i}\rho_i}{2R_i}\text{ on $\de B_i$.}\label{normal velo}
\end{align}

Here $e_R$ and $e_Z$ are the unit vectors and the last summand is due to the fact that if the outer radius $q_{R_i}$ changes, then the inner radius must also change because the volume is fixed.

Let $\mathcal{F}:=\mathbb{H}\backslash \bigcup_i B_i$ denote the (time-dependent) domain of the fluid. 

Let $L_R^2$ and $H_R^1$ denote the $L^2$ resp. $H^1$ spaces with respect to the measure $r\dr\dz$.

We assume that:

\begin{condition}In $\mathcal{F}(t)$ the fluid fulfills the axisymmetric Euler equations with zero vorticity and no swirl for $t\in [0,\infty)$, i.e.\ \begin{align}
\de_t u+(u\cdot\nabla)u+\nabla p\eq 0\label{Euler1}\\
\div(ru)\eq 0\label{Euler2}\\
\curl(u)\eq 0\label{Euler3}\\
u_r\eq 0 \text{ \normalfont{on} $\de \mathbb{H}$.}\label{Euler4}
\end{align} 

here the $\div$ and $\curl$ are taken with respect to the variables $(r,z)$ and $u$ is $\R^2$-valued.
\end{condition}

These equations are equivalent to the usual Euler equation for $u$ of the form \eqref{no swirl} with no vorticity after going back to three-dimensional coordinates.

We define \begin{align}
\gamma_i\,:=\,\int_{\de B_i}u\cdot\tau\dx,
\end{align}

where $\tau=n^\perp$. This is a conserved quantity by Kelvin's circulation law. For technical reasons, we will assume: \begin{condition}
None of the $\gamma_i$'s is $0$.
\end{condition}

\begin{condition}\phantomsection\label{strongsol}We assume that this solution is strong in the sense that $u,\nabla u,\de_t u\in L_R^2\cap C^1(\mathcal{F})$ and that $q$ is $C^2$ in time. In particular this should hold for the initial data.\label{strongsol}\end{condition}
We assume that the normal velocity of the boundary of each $B_i$ matches the velocity of $u$ in the corresponding direction: \begin{condition} For all $i$ we have \begin{align}
u\cdot n\eq u(\dot{q}) \text{ on $\de B_i$}\label{Euler5},
\end{align}

where we use the identification between tangent vectors and normal velocities mentioned above.\end{condition}

For simplicity we will assume that all bodies and the fluid have the constant density $1$, though all the arguments still work for different densities.

\begin{condition} We assume that in the $z$-direction, the momentum is (formally) preserved, which yields the condition \begin{align}
v_i\ddot{q}_{{Z}_i}\eq-\int_{\de B_i}rpn\cdot e_Z\dx\label{eq z}
\end{align}

for all $i$.\end{condition}

It remains to derive a condition for the interaction of the fluid and the solids in the $r$-component. As the solids can change their shape, we make the Ansatz of prescribing an interior velocity field and assuming that the kinetic energy of each $B_i$ only changes through the force exerted by the pressure at the boundary.

We associate to each normal velocity $t_i^*$ associated to $B_i$ such an interior velocity field $u_{i,int}(t_i^*)\in H^1(B_i)$ with \begin{align}
&\div ru_{i,int}\eq0 \text{ in $B_i$}\label{Int1}\\
&\curl u_{i,int}\eq0 \text{ in $B_i$}\label{Int2}\\
&u_{i,int}\cdot n\eq u(t_i^*)\text{ on $\de B_i$.}\label{Int3}
\end{align}

Existence and uniqueness can be obtained by standard elliptic theory \cite{evans2022partial}[Chapter 6], as $u_{i,int}(t_i^*)$ can be written as $\nabla \phi$ with $\div(r\nabla \phi)=0$ by the assumption of curl-freeness. Note that this is linear in $t_i^*$, and that \begin{align}\label{Ztrivial}u_{i,int}(t_{Z_i}^*)=t_{Z_i}^*e_Z.\end{align}

We can use this to associate a quadratic form on $T_{q_i}M$ by \begin{align}
(t_i^*)^TE_{q_i}t_i^+\,:=\,\int_{B_i}r\scalar{u_{i,int}(t_i^*)}{u_{i,int}(t_i^+)}\dx, \label{def E}
\end{align}

for tangent vectors $t_i^*,t_i^+$, associated to $B_i$. Clearly, this is symmetric and positive definite. Because of the explicit form \eqref{Ztrivial}, we know that \begin{align}
(e_Z)_i^TE_{q_i}(e_Z)_i\eq v_i,
\end{align}

where $(e_Z)_i\in (\R^2)^k$ is the vector with $e_Z$ in the $i$-th component. We set \begin{align}
(e_R)_i^TE_{q_i}(e_R)_i\,=:\, f_i(q_{R_i})\label{def f},
\end{align}

here the function $f_i$ depends on $v_i$ and $(e_R)_i$ is the vector with $e_R$ in the $i$-th component.

In summary, the quadratic form can be written in components as \begin{align}
(t_i^*)^T\begin{pmatrix} f_i(q_{R_i}) & 0\\ 0 & v_i\end{pmatrix}(t_i^+),\label{def f 2}
\end{align}

as one can see from the explicit formula \eqref{Ztrivial} and the fact that $u_{i,int}(t_{R_i}^*)$ must be asymmetric in $z$ in the second component and hence $t_{R_i}^*$ and $t_{Z_i}^*$ are orthogonal to each other with respect to $E_{q_i}$.

We define the kinetic energy of $B_i$ as \begin{align}\label{intEnergy}
\mathcal{E}_{B_i}\,:=\,\frac{1}{2}\dot{q}_i^TE_{q_i}\dot{q}_i.
\end{align}

We assume that the kinetic energy of each $B_i$ only changes through the force exerted at the boundary, that is \begin{align}
\mathcal{E}_{B_i}'\eq-\int_{\de B_i} rpu(\dot{q}_i)\cdot n\dx.\label{energyBalance}
\end{align}

After using the Decomposition \eqref{def f 2} and subtracting the Condition \eqref{eq z}, one obtains that \begin{align}
f_i(q_{R_i})\ddot{q}_{R_i}\dot{q}_{R_i}+\frac{1}{2}\de_{q_{R_i}}f_i(q_{R_i})(\dot{q}_{R_i})^3\eq -\int_{\de B_i}r p\dot{q}_{R_i}\left(n\cdot e_R -\frac{\rho_i}{2R_i}\right)\label{eq r}\dx.
\end{align}

We make the extra assumption that one can divide out $\dot{q}_{R_i}$, which is equivalent to saying that whenever $\dot{q}_{R_i}=0$ and the force at the boundary is nonzero, then $\ddot{q}_{R_i}$ is not zero, which rules out bodies with fixed $R_i$-coordinate. This gives the final equation: \begin{condition} 
For all $i$ we have
\begin{align}\label{eq R}
f_i(q_{R_i})\ddot{q}_{R_i}+\frac{1}{2}\de_{q_{R_i}}f_i(q_{R_i})(\dot{q}_{R_i})^2\eq -\int_{\de B_i}r p\left(n\cdot e_R -\frac{\rho_i}{2R_i}\right)\dx.\end{align}
\end{condition}

We can also write \eqref{eq z} and \eqref{eq R} as a single equation \begin{align}\label{2ndDeri}
(t_i^*)^TE_{q_i}\ddot{q}_i+\frac{1}{2}t_i^*(\de_{\dot{q}_i} E_{q_i}\cdot \dot{q}_i)\dot{q}_i\eq-\int_{\de B_i} rpu(t_i^*)\dx,
\end{align}

where $t_i^*$ is an arbitrary tangent vector associated to $B_i$.

\subsection{Main results}\label{Main results}

Our first main result is wellposedness: \begin{theorem}[Informal]
For every initial datum $q(0),$ $\dot{q}(0)$ the system detailed in the previous section has a unique solution up to some time $T>0$. If $T<\infty$, then $q$ blows up at $T$ in the sense that some of the bodies either collide with each other, the boundary or escape to $\infty$.

Furthermore, the system preserves energy.
\end{theorem}

The more precise statements can be found in Corollary \ref{Cor:Exist}, Lemma \ref{Lem:Energy} and Thm.\ \ref{Thm:exist}.

For the zero-radius limit, we shall first introduce some notation. We will use a rescaling parameter $\eps$ and denote the manifold of configurations associated to the bodies with the ``volumes'' $v_1\eps^2,\dots, v_k\eps^2$ by $\tilde{M}_\eps$ (recall that the ``volumes'' were defined in \eqref{def rho}). We still denote the inner radii with $\rho_1,\dots\rho_k$. We write $\tilde{\rho}_i$ for the unrescaled radii $\frac{\rho_i}{\eps}$. We fix some $\binom{R_0}{Z_0}$ with $R_0>0$. There are two different regimes that one can consider:

\subsubsection{The first regime}

The first one is also considered in \cite{JerrardSmets} and \cite{Davila2022leapfrogging}. We set all $\gamma_i$'s to be equal to one and set the centers to be \begin{align}\label{def regime1}
q_i\eq\left(R_0+\frac{\tilde{q}_{R_i}}{\sqrt{|\log\eps|}},Z_0+\frac{\tilde{q}_{Z_i}}{\sqrt{|\log\eps|}}\right),
\end{align}

where the rescaled positions $\tilde{q}(0):=(\tilde{q}_{R_1}(0),\tilde{q}_{Z_1}(0),\dots)$, should be independent of $\eps$.

We will further rescale time by a factor $|\log\eps|$ and work with the rescaled positions $\tilde{q}_i:=(\tilde{q}_{R_i},\tilde{q}_{Z_i})$.

Formally, the vorticity hidden in the rotation is $-1$ for each body (where the minus comes from opposite sign of the axisymmetric vorticity).

Hence, in this regime, the main part of the self-induced velocity (in rescaled time and space) of a vortex ring, that is $\frac{-1}{4\pi R_0}|\log\eps|^\frac{1}{2}e_Z$, is the same for all rings, hence we may neglect this part. The next order part is of the form $\frac{\tilde{q}_{R_i}}{4\pi R_0^2}e_Z$ in rescaled time and space by Taylor's theorem.

The velocity induced by the $i$-th ring on the $j$-th ring is of the form $\frac{1}{2\pi}\frac{(q_j-q_i)^\perp}{|q_i-q_j|^2}$.

We hence expect that in the limit $\eps\rightarrow 0$, the velocities $\tilde{q}_i$'s should solve the system \begin{align}\label{PointSystem1}
\tilde{q}_i'\eq\frac{1}{2\pi}\sum_{j\neq i}\frac{(q_i-q_j)^\perp}{|q_i-q_j|^2}+\frac{\tilde{q}_{R_i}}{4\pi R_0^2}e_Z\end{align}

up to the subtracted term $-\frac{1}{4\pi R_0}|\log\eps|^\frac{1}{2}e_Z$.

\begin{theorem}\phantomsection\label{Limit1}
Assume that the solution of \eqref{PointSystem1} exist until time $T$ (in the sense that no components of the solution go to $\infty$ and the distance between the different components stays positively bounded from below). 
Assume that the shifted initial velocity $\tilde{q}_i'+\frac{|\log\eps|^\frac{1}{2}}{4\pi R_0^2}e_Z$ is bounded uniformly in $\eps$. 

Then $\tilde{q}+t\frac{1}{4\pi R_0}|\log\eps|^\frac{1}{2}e_Z$ converges to the solution of \eqref{PointSystem1} weakly\star in $W_{loc}^{1,\infty}([0,T))$ in rescaled time.
\end{theorem}

This ODE system has been studied e.g.\ in \cite{JerrardSmets}, where the existence of periodic solutions for two rings has been shown. Such periodic solutions correspond to a ``leapfrogging'' motion of the rings, which was predicted already by Helmholtz in his famous work \cite{helmholtz1858integrale,helmholtz1867lxiii}.

\subsubsection{The second regime}

We can also consider the regime where the self-induced motion and the motion induced by other rings is of the same order. For this we set \begin{align}\label{def regime2}
q_i\eq \left(R_0+\frac{\tilde{q}_{R_i}}{|\log\eps|},Z_0+\frac{\tilde{q}_{Z_i}}{|\log\eps|}\right).
\end{align}

Here the rescaled initial position $\tilde{q}(0)$ should again be independent of $\eps$.

We will further rescale time by a factor $|\log\eps|^2$ and will again work with the rescaled positions $\tilde{q}$.

In this regime, all expected velocities are of order $1$ in rescaled time and space.

We hence expect that in the limit $\eps\rightarrow 0$ and in the rescaled time, the $\tilde{q}$ should solve the system \begin{align}\label{PointSystem2}
\tilde{q}_i'\eq\frac{1}{2\pi}\sum_{j\neq i}\gamma_j\frac{(\tilde{q}_i-\tilde{q}_j)^\perp}{|\tilde{q}_i-\tilde{q}_j|^2}-\frac{\gamma_i}{4\pi R_0}e_Z\end{align}

This system was studied in \cite{MarchioroNegrini2} and one can show that if all $\gamma_i$'s have the same sign, then solutions cannot blow up in finite time \cite{MarchioroNegrini2}[Thm.\ 2.1].

\begin{theorem}\phantomsection\label{Limit2}
Assume that the solution of \eqref{PointSystem2} exist until some time $T$ in the same sense as in the previous theorem. Further assume that the initial velocities $\tilde{q}_i'(0)$ (in the rescaled time) are bounded uniformly in $\eps$. 

Then $\tilde{q}$ converges to the solution of \eqref{PointSystem2} weakly\star in $W_{loc}^{1,\infty}([0,T))$ in rescaled time.
\end{theorem}

\textbf{General notation}: We will use the notation $\tilde{q}$ for the rescaled positions in both regimes, as most estimates work completely similarly for both regimes. We also denote the manifold of $\tilde{q}$'s for which the corresponding $q$ is in $\tilde{M}_\eps$ by $M_\eps$.

We write $A\lesssim B$ if there is a constant $C>0$ such that $A\leq CB$, where the constant $C$ is allowed to depend on the number $k$ and on $q$ resp.\ $\tilde{q}$ and on which of the regimes we are in, but not on any other quantities. 

Similarly, we write $\ell$ for irrelevant (finite) exponents, which are allowed to depend on which regime we are in, but not on any other quantities and are allowed to change their value from line to line.

If $\Omega\subset \mathbb{H}$ we write $L_0^2(\Omega)$ for the space of all functions $f\in L^2(\Omega)$ such that for every connected component $\Omega_i$ of $\Omega$ we have $\int_{\Omega_i}f\dx=0$. 

If $S\subset \mathbb{H}$, we write $S^{\R^3}$ for its figure of revolution in $\R^3$.

\section{Wellposedness of the system}\label{Section2}
We will follow the approach in \cite{GlassMunnierSueur2} to show that our system can be reduced to an ODE in $q$ only. The main additional difficulty is that we are in an unbounded domain and hence need decay estimates to justify partial integrations.

\subsection{Representation of $u$}\label{repU}
In this subsection, we show that $u$ can be recovered from a potential and a streamfunction. 

\begin{definition}
 For $t^*\in T_{q_i}M$, let $\phi_{i,t^*}=\phi_{i,t^*}(t)$ be defined as the unique solution of the Neumann problem\begin{align}
&\div(r\nabla\phi_{i,t^*})\eq0 \text{ \normalfont{in} $\mathcal{F}(t)$}\\
&\de_n\phi_{i,t^*}\eq u(t^*) \text{ \normalfont{on} $\de B_i$}\\
&\de_n\phi_{i,t^*}\eq 0 \text{ \normalfont{on} $\de B_j$  \normalfont{for} $j\neq i$  \normalfont{and on} $\de\mathbb{H}$}\\
&\phi_{i,t^*}\in \dot{H}_R^1.\\
&\phi_{i,t^*}\rightarrow 0\text{ at $\infty$}
\end{align}
\end{definition}

We first need to check that this is well-defined.

\begin{lemma}\phantomsection\label{ExistencePhi}
Let $b\in L^2(\bigcup_i \de B_i)$ be such that $\int_{\bigcup_i\de B_i}rb\dx=0$, then the system \begin{align}
&\div(r\nabla\phi)\eq0 \text{ \normalfont{in} $\mathcal{F}(t)$}\\
&\de_n\phi_{}\eq b \text{ \normalfont{on} $\bigcup_i \de B_i$}\\
&\de_n\phi_{}\eq 0 \text{ \normalfont{on} $\de\mathbb{H}$}\\
&\phi_{}\in \dot{H}_R^1.\\
&\phi_{}\rightarrow 0\text{ at $\infty$}
\end{align}

has a unique solution. Furthermore \begin{align}\label{decayphi}
|\nabla^m\phi(x)|\,\lesssim\,\frac{\norm{b}_{H^{m}}}{1+|x|^{2+m}} \text{ \normalfont{for all} $m\in\N_{\geq 0}$}\end{align}

where the implicit constant is bounded locally uniformly in $q$.

\end{lemma}

This implies the welldefinedness of $\phi_{i,t^*}$ and that the Estimate \eqref{decayphi} holds locally uniformly in $q$ for $\phi_{i,t^*}$, as only can directly check that $\int_{\bigcup_i\de B_i} ru(t^*)\dx=0$.

\begin{proof}
We go back to three-dimensional coordinates and set $\phi^{\R^3}(r,z,0)=\phi(r,z)$, where $\phi^{\R^3}$ is axisymmetric. Then $\phi$ solves the system above iff $\phi^{\R^3}$ solves the corresponding Neumann problem for $\Delta$ in $\R^3\backslash \bigcup_j B_j^{\R^3}$. By standard techniques (see e.g. \cite{Amrouche}), we obtain a unique solution $\phi^{\R^3}\in \dot{H}^{1}(\R^3\backslash \bigcup B_j^{\R^3})$.
We furthermore obtain from this that $\de_n\phi=0$ on $\de\mathbb{H}$. \\

The decay rate then directly follows from the Lemma below.
\end{proof}

\begin{lemma}\phantomsection\label{decay1}
\item[a)] Let $\zeta\in \dot{H}^1(\mathcal{F}^{\R^3})$ be axisymmetric and such that $\Delta\zeta=0$ in $\mathcal{F}^{\R^3}$  and $\zeta\rightarrow 0$ at $\infty$. Then it holds that \begin{align}
\left|\nabla^m\zeta(x)\right|\,\lesssim\,\frac{\norm{\de_n\zeta}_{H^{m}(\de\mathcal{F}^{\R^3})}}{1+|x|^{1+|m|}}\quad \forall m\in\N_{\geq 0}.
\end{align}

The implicit constant is bounded locally uniformly in $q$.

\item[b)] Let $\zeta$ be as in a). If additionally \begin{align}\int_{\bigcup_j \de B_j^{\R^3}}\de_n\zeta\dx\eq0,\end{align}
then \begin{align}
\left|\nabla^m\zeta(x)\right|\,\lesssim\,\frac{\norm{\de_n\zeta}_{H^{m}(\de\mathcal{F}^{\R^3})}}{1+|x|^{2+|m|}}\quad \forall m\in\N_{\geq 0},
\end{align}

where the implicit constant is controlled as in a).
\end{lemma}
\begin{proof}
If we extend $\zeta$ to $\R^3$ by solving the Dirichlet problem for $\zeta|_{\de B_j^{\R^3}}$ on each $B_j^{\R^3}$, then the (distributional) Laplacian of this extension is of the form $[\de_n\zeta]\mathcal{H}^2\mres{ \bigcup \de B_j^{\R^3}}$ where $[\cdot]$ denotes the jump across the boundary, as a direct calculation shows. By elliptic regularity (cf. \cite{grisvard2011elliptic}[Chapter 2]), this is a finite measure. 

We claim that we can recover this extension by convoluting the distributional Laplacian with the Newtonian potential, indeed for any $f\in C_c^\infty(\R^3)$, we have \begin{align}
\int_{\R^3}\int_{\bigcup_i\de B_i^{\R^3}}\frac{-[\de_n\zeta]}{4\pi|x-y|}\Delta f(x)\,\text{d}\mathcal{H}^2(y)\dx\eq\int_{\bigcup_i\de B_i^{\R^3}}[\de_n\zeta]f(y)\,\text{d}\mathcal{H}^2(y)\eq \int_{\R^3}\zeta \Delta f\dx,
\end{align}

where in the second step, we used that the Newtonian potential is the inverse Laplacian. Hence the difference $[\de_n\zeta]\mathcal{H}^2\mres(\bigcup_i\de B_i^{\R^3})*\frac{-1}{4\pi|x|}-\zeta$ is a harmonic tempered distribution, i.e.\ a polynomial.

Now \begin{align}
[\de_n\zeta]\mathcal{H}^2\mres(\bigcup_i\de B_i^{\R^3})*\frac{-1}{4\pi|x|}-\zeta\,\rightarrow\, 0,
\end{align} 

by the assumption on $\zeta$ and because $[\de_n\zeta]\mathcal{H}^2\mres(\bigcup_i\de B_i^{\R^3})$ is a finite measure, so this difference is zero, which shows the claim.

We hence obtain that $|\nabla^m\zeta(x)|\lesssim\frac{\norm{\de_n\zeta}_{L^2(\de\mathcal{F}^{\R^3})}}{1+|x|^{1+m}}$ for all $m\geq0$ and for $\dist(x,\bigcup B_i^{\R^3})\geq 1$. For $x$ close to the $B_j$'s the estimate follows from elliptic regularity theory. This shows the estimate in part a), part b) works exactly the same way, except that the integral of $[\de_n\zeta]$ vanishes by the assumption and partial integration, which gives one order more of decay.\\

To see that these bounds are locally uniform in $q$, we need uniform estimates on $\norm{[\de_n\zeta]}_{L^1(\bigcup_i\de B_i)}$, note that for this we only need up-to-the-boundary estimates for a neighborhood of each $B_i$, locally uniform in $q$ and a locally uniform $L^2$ estimate. 
By going back to axisymmetric coordinates, one obtains the former, as the geometries of these neighborhoods (in axisymmetric coordinates) only change by rescaling with a bounded factor. 
On the other hand, one can obtain from the energy equality (which is justified by the decay estimates we have already proven) in axisymmetric coordinates that \begin{align}\label{energy eq}
\int_{\mathcal{F}} r|\nabla\zeta|^2\dx\eq -\int_{\bigcup_i \de B_i}r\zeta\de_n \zeta\dx\gtrsim \norm{\zeta}_{H^1(\bigcup_i (B_i+B_1(0)\backslash B_i))}^2.
\end{align}

Here we used the (three-dimensional) Sobolev inequality to control the $L^2$-norm on the right-hand side. The constant of the Sobolev inequality is locally uniform in $q$, as one can e.g.\ see by using a diffeomorphism between the different instances of $\mathcal{F}$.

Now we can use the trace inequality from $H^1(B_i+B_1(0)\backslash B_i)$ to $L^2(\de B_i)$ in \eqref{energy eq}, whose constant is also locally uniform in $q$, as the geometry only changes through rescaling by a bounded factor. This implies the desired locally uniform estimate \begin{align}
\norm{\de_n\zeta}_{L^2(\bigcup_i\de B_i)}\,\gtrsim \,\norm{\zeta}_{L^2(\bigcup_i (B_i+B_1(0)\backslash B_i))}.
\end{align}

\end{proof}

\begingroup
\allowdisplaybreaks

The argument for the existence of the stream function is more complicated, as there is no easy algebraic relation between the three-dimensional stream function and the axisymmetric stream function.

\begin{definition}\phantomsection\label{def psi}
Let $\psi_{i}=\psi_i(t)$ be the solution of the elliptic equation \begin{align}
&\div\left(\frac{1}{r}\nabla\psi_{i}\right)\eq0 \text{ in $\mathcal{F}(t)$}\\
&\psi_{i}|_{\de B_j} \text{ is constant } \forall j\\
&\int_{\de B_j}\frac{1}{r}\de_{n}\psi_{i}\dx\eq\delta_{ij}\label{Flow}\\
&\psi_i|_{\{r=0\}}\eq0\\
&\lim_{r\rightarrow 0}\frac{1}{r}\de_z\psi_i\eq 0
\end{align}

We refer to the constant boundary values on $\de B_j$ as $C_{ij}$.
\end{definition}

\endgroup

\begin{lemma}\phantomsection\label{ExistencePsi}
Such a $\psi_i$ always exists and is unique under the constraint $\frac{1}{\sqrt{r}}\nabla\psi_i\in L^2$. Furthermore $\frac{1}{r}\nabla\psi_i$ is continuous at $r=0$ and we have the estimate \begin{align}
\left|\nabla^m\left(\frac{1}{r}\nabla\psi_i(x)\right)\right|\,\lesssim\,&\frac{1}{1+|x|^{2+m}}\quad \forall m\in\N_{\geq 0}.
\end{align}
The implicit constant in the estimate is locally uniformly bounded in $q$.
\end{lemma}
\begin{proof}
We first show that an auxiliary function $u_{2,i}$ can be constructed by going back to three-dimensional coordinates and later show that one can recover $\psi_i$ from it.\\

Step 1.  We define $u_{2,i}$ on $\mathcal{F}^{\R^3}$ as the axisymmetric vector field with no azimuthal component which solves (in three-dimensional variables) the system \begin{align} \label{DivCurl}
&\div\, u_{2,i}\eq0\\
& \curl\, u_{2,i}\eq0\\
& u_{2,i}\cdot n\eq0\text{ on $\de B_j^{\R^3}$ for all $j$}\\
& u_{2,i}\in L^2(\mathcal{F}^{\R^3}).\end{align}

We make the Ansatz $\curl \Psi_i=u_{2,i}$ for a purely azimuthal field $\Psi_i=\tilde{\Psi}_ie_\theta$, where $e_\theta$ is the unit vector in the $\theta$-direction.

This gives the equations \begin{align}
(u_{2,i})_r\eq-\de_z\tilde{\Psi}_i,\quad (u_{2,i})_z\eq\de_r\tilde{\Psi}_i+\frac{1}{r}\de_r \tilde{\Psi}_i,
\end{align}

using that the last term can be rewritten as $\frac{1}{r}\de_r(r\Psi_i)$, this  turns the System \eqref{DivCurl} into the system \begin{align}
&\Delta \Psi_i\eq0\text{ in $\mathcal{F}^{\R^3}$}\label{BigPhi1}\\
&r\tilde{\Psi}_i|_{\de B_j^{\R^3}}\text{ is constant for all $j$.}\label{BigPhi2}
\end{align}

For each fixed set of constant boundary values $(\tilde{C}_{ij})_{j=1,\dots,k}$ for $r\tilde{\Psi}_i$, this system has a unique solution $\Psi_i(\tilde{C}_{ij})\in H^1$ by standard techniques. Also, because $\Psi_i$ is purely azimuthal, it must vanish at $r=0$ and hence it holds that \begin{align}\label{u at 0}
(u_{2,i})_r\eq0\end{align}

 at $r=0$.\\

Step 2. To show uniqueness of $u_{2,i}$ for given $(\tilde{C}_{ij})$, we note that we can recover a $\Psi_i$ fulling the system \eqref{BigPhi1},\eqref{BigPhi2} from $u_{2,i}$.

 Indeed we may extend $u_{2,i}$ to the full space by zero as $\overline{u}_{2,i}$, which preserves the divergence-freeness, as the distributional divergence on the boundary equals $[\overline{u}_{2,i}\cdot n]=0$.
 
  It is well known that on the full space, every divergence-free field can be written as a curl, indeed we have $-\curl\Delta^{-1}\curl g=g$ for every divergence-free $g$, as a straightforward calculation shows. 
  
Since we also have $\curl \overline{u}_{2,i}\in H^{-1}$ and the fundamental solution of the Laplacian maps $H^{-1}$ to $H^{1}$, we see that the field obtained this way lies in $H^1$.
 
Hence two different solutions $u_{2,i}$ for the same boundary values would give rise to two different $\Psi_i$'s in $H^1$ with the same boundary values, which is impossible.\\

Step 3. Now we can view $u_{2,i}$ as a function in $(r,z)$ again, it fulfills $\div(ru_{2,i})=0$ (with the two-dimensional divergence). Then we can find a $\psi_i(\tilde{C}_{ij})$ such that \begin{align}
u_{2,i}\eq\frac{1}{r}\nabla^\perp\psi_i,
\end{align} because $ru_{2,i}^\perp$ is curl-free. This can be done with the usual path integral construction, it is easy to check that the condition $u_{2,i}\cdot n=0$ ensures that even paths which are not homotopy equivalent yield the same values. One can then check by direct calculation that $\div(\frac{1}{r}\nabla\psi_i(\tilde{C}_{ij}))=0$ holds and by the boundary condition for $u_{2,i}$, we see that $\psi_i(\tilde{C}_{ij})$ must be constant on each $\de B_j$. Furthermore, we have that $u_{2,i}$ is continuous at $r=0$ by elliptic regularity and hence $\frac{1}{r}\nabla\psi_i(\tilde{C}_{ij})$ is continuous at $r=0$ and $\frac{1}{r}\de_z\psi_i=0$ at $r=0$ by \eqref{u at 0}.

This $\psi_i(\tilde{C}_{ij})$ is unique up to an additive constant under the condition $\frac{1}{\sqrt{r}}\nabla\psi_i(\tilde{C}_{ij})\in L^2$ (here one gets an additional factor $r$ from the coordinate change), as one can recover the unique $u_{2,i}$ from it. \\

Next, we argue that we uniquely can pick the boundary values $(\tilde{C}_{ij})$ such that the Condition \eqref{Flow} holds. It suffices to show that the linear map that sends the boundary values $(\tilde{C}_{ij})$ to the integrals $\int_{\de B_j}\frac{1}{r}\de_{n}\psi_i(\tilde{C}_{ij})\dx$ is invertible for each fixed $i$. 

First note that the $\tilde{C}_{ij}$'s are also the boundary values of $\psi_i$ (up to an additive constant) because we have that $ru_{2,i}=\nabla^\perp r\tilde{\Psi}_i$ (in $(r,z)$-coordinates) and hence by applying the fundamental theorem of calculus, we see that $(r\tilde{\Psi})(x)-(r\tilde{\Psi})(y)=\psi_i(x)-\psi_i(y)$ (in axisymmetric coordinates).

Assume there is a nonzero vector of $\tilde{C}_{ij}$'s such that all integrals vanish. Without loss of generality, we may assume that $\tilde{C}_{i1}$ is the biggest one of the $\tilde{C}_{ij}$'s. Then the normal derivative of $\psi_i$ on $\de B_1$ must be nonpositive by the maximum principle and hence must be zero everywhere on $\de B_1$. Since the tangential derivative also vanishes, the constant extension of $\psi_i(\tilde{C}_{ij})$ to $B_1$ still fulfills $\div(\frac{1}{r}(\psi_i(\tilde{C}_{ij}))=0$. But this extension is then a locally, but not globally constant solution of an elliptic equation, which is a contradiction.\\

We have $\de_z\psi_i=0\cdot u_{2,i}=0$ at $r=0$ because $u_{2,i}$ is continuous by elliptic regularity. We may choose the additive factor that we have leftover such that $\psi_i=0$ at $r=0$ holds.

The uniqueness of $\psi_i$ follows from the uniqueness of the $\psi_i(\tilde{C}_{ij})$.\\

Step 4. By elliptic regularity, it is easy to see that the $C_{ij}$ are locally uniformly bounded in $q$, and hence $\norm{\de_n\Psi_i(C_{ij})}_{H^m}$ is locally uniformly bounded in $q$ by elliptic regularity theory.

By Lemma \ref{decay1} a), we see that \begin{align}|\nabla^m\Psi_i(x)|\,\lesssim\, \frac{1}{1+|x|^{1+m}},\end{align} for all $m\in \N_{\geq 0}$ which implies the decay estimate. 

\end{proof}

It is known that the system \begin{align}
&\div\left(\frac{1}{r}\nabla f\right)\eq g\text{ in $\mathbb{H}$}\\
&f\eq\de_nf\eq 0\text{ on $\de\mathbb{H}$}
\end{align}

has a fundamental solution $K$ such that \begin{align}\label{definition K}
f(y)\eq \int_{\mathbb{H}}g(x)K(x,y)\dx
\end{align}

is the unique solution under suitable decay assumptions on $f$ for e.g.\ $g\in C_c^\infty(\mathbb{H})$, see e.g. \cite{GallaySverak}[Section 2].

\begin{lemma}\phantomsection\label{rep psi}\begin{itemize}
\item[a)] Let $\div(\frac{1}{r}\nabla\zeta)=0$ in $\mathcal{F}$ with $\frac{1}{\sqrt{r}}\nabla\zeta\in L^2(\mathcal{F})$, assume that $\frac{1}{r}\nabla\zeta$ is continuous for $r\rightarrow 0$ and that $\zeta|_{r=0}=0$. Furthermore, assume that $\zeta|_{\de\mathcal{F}}$ is sufficiently smooth, then there is a constant $C$, depending on $\zeta$ such that \begin{align}
|\zeta(x)|\,\leq\, \frac{C}{1+|x|}.
\end{align}
In particular this holds for $\psi_i$.
\item[b)] $\psi_i$ can be represented as \begin{align}
\psi_i(y)\eq\int_{\bigcup_i\de B_i}\frac{1}{r}K(x,y)\de_n\psi_i\dx.
\end{align}
\end{itemize}
\end{lemma}
\begin{proof} 
a) and b) We claim that if we extend $\zeta$ to $\mathbb{H}$ by solving the Dirichlet problem for $\div(\frac{1}{r}\nabla\cdot)$ with boundary values $\zeta$ in each $B_i$, then it holds that \begin{align}
\zeta\eq \int_{\bigcup_i\de B_i}\frac{1}{r}K(x,y)[\de_n\zeta]\dx.
\end{align}

This directly shows b) because for $\psi_i$ the extension to each $B_i$ is constant. It also shows a) by the fact that $[\de_n\zeta]\mathcal{H}^1\mres \de B_i$ is a finite measure by elliptic regularity and  the fact that the fundamental solution $K(x,y)$ decays like $\frac{1}{1+|y|}$ at $\infty$ locally uniformly in $x$ (see \cite{GallaySverak}[Lemma 2.1 ii)]).

The same argument as in the proof of Lemma \ref{decay1} shows that \begin{align}
g(y)\,:=\,\int_{\de\mathcal{F}}\frac{1}{r}K(x,y)[\de_n\zeta]\dx\end{align} fulfills $\div(\frac{1}{r}\nabla (g-\zeta))=0$. Furthermore by the aforementioned decay estimates for $K$ we see that $|g(x)|\lesssim \frac{1}{|1+|x|}$, that $g=0$ at $r=0$ and that $\frac{1}{r}\nabla g$ is continuous at $0$. By elliptic regularity, it holds that $g-\zeta$ is smooth in the interior of $\mathbb{H}$.

Now let $h(r,z,\theta)=(\frac{1}{r}\nabla(\zeta-g)(r,z))^\perp$ (in axisymmetric coordinates). This function is divergence-free (with respect to three-dimensional variables) as a direct calculation shows for $r>0$, at $r=0$ it is also divergence-free by continuity.

Hence there is an $H$ with $\curl H=h$, where the gradient is taken with respect to three-dimensional variables. Then it holds that $\Delta H=0$ and $H$ is a tempered distribution and hence is a polynomial. 

Hence we know that $\zeta-g$ is a polynomial as well, however, we have that $\zeta-g=0$ at $r=0$, hence it is also a polynomial in $r$ only. However, $g\rightarrow 0$ for $r\rightarrow \infty$ and hence if $\zeta-g$ does not vanish, then $\zeta$ would have to grow at least linearly in $r$. However for all $a$ and large enough $R$ we would then have \begin{align}
1\,\lesssim\,\frac{1}{R}\left|\zeta(R,a)-\zeta\left(\frac{R}{2},a\right)\right|\,\lesssim\, \frac{1}{R}\int_\frac{R}{2}^R|\de_r\zeta(s,a)|\ds\,\lesssim\, \left(\int_\frac{R}{2}^R\frac{1}{s}|\de_r\zeta(s,a)|^2\ds\right)^\frac{1}{2}.
\end{align}

By taking a square root and using Fubini, we obtain that $\frac{1}{\sqrt{r}}\nabla\zeta\notin L^2$, which is a contradiction.

\end{proof}

\begin{lemma}\phantomsection\label{rot pre}
The Euler equation \eqref{Euler1} holds if \eqref{Euler2} and \eqref{Euler3} hold and the circulations $\gamma_i=\int_{\de B_i}\tau\cdot u\dx$ are conserved in time. 
\end{lemma}
In the two-dimensional setting this statement is well-known, see e.g.\ \cite{filho2007vortex}.

\begin{proof}
Indeed the vorticity equation always holds and therefore we have $\curl(\de_t u+(u\cdot \nabla)u)=0,$ however not every curl-free field has to be a gradient in $\mathcal{F}$, as the domain has holes. Using the usual path integral construction, it is easy to see that it is a gradient if \begin{align}
\int_\Gamma (\de_tu+(u\cdot \nabla)u)\cdot \tau_\Gamma\dx\eq0\end{align} along every closed, non-selfintersecting path $\Gamma\subset \mathcal{F}$ with normalized tangent $\tau_\Gamma$, which has winding number $1$ with respect to exactly one $B_i$ and $0$ with respect to all others. If $\Phi_t$ is the flow induced by $u$, then by direct calculation one sees that \begin{align}
\de_s\int_{\Phi_s(\Gamma)}u\cdot\tau_{\Phi_s(\Gamma)}\dx\big|_{s=0}\eq\int_\Gamma (\de_tu+(u\cdot \nabla)u))\cdot \tau_\Gamma\dx.
\end{align}

On the other hand we see that for such $\Gamma$ it holds \begin{align}
\int_\Gamma u\cdot\tau_\Gamma\dx\eq\int_{\de B_i} u\cdot\tau\dx,
\end{align}

which is constant by assumption, and hence we see that there is a $p$ such that $\de_t u+(u\cdot\nabla)u=-\nabla p$.
\end{proof}

\begin{proposition}
The function  \begin{align}\label{u1u2}
u(t)\eq\sum_{i=1}^k\left(\nabla\phi_{i,\dot{q}_i}+\gamma_i\frac{1}{r}\nabla^\perp\psi_i\right)\,=:\,u_1+u_2
\end{align}

is a solution to the axisymmetric Euler equations \eqref{Euler1}-\eqref{Euler5}. 
\end{proposition}
\begin{proof}
A direct calculation reveals that $u$ is $\curl$-free and fulfills $\div(ru)=0$ and hence $u$ fulfills \eqref{Euler2} and \eqref{Euler3} in $\mathcal{F}$. We further observe that \begin{align}
&n\cdot\sum_{i=1}^k\nabla\phi_{i,\dot{q}_i}\eq u(\dot{q}) \text{ on $\bigcup_j\de B_j$}\\
&n\cdot\frac{1}{r}\nabla^\perp\psi_i\eq 0 \text{ on $\bigcup_j\de B_j$}\\
&\int_{\de B_j}\nabla\phi_{i,\dot{q}_i}\cdot\tau\dx\eq \int_{\de B_j}\de_\tau\phi_{i,\dot{q}_i}\dx\eq0\\
&\int_{\de B_j}\frac{1}{r}\nabla^\perp\psi_{i}\cdot\tau\dx\eq \int_{\de B_j}\frac{1}{r}\de_n\psi_{i}\dx\eq\delta_{ij},\\
&\frac{1}{r}\de_z\psi_i\eq \de_r\phi_{i,\dot{q}_i}\eq 0\text{ on $\de \mathbb{H}$}
\end{align}

where $\tau=n^\perp$.

Hence this $u$ has the prescribed circulations and boundary velocities, which shows the statement by Lemma \ref{rot pre}. 
\end{proof}

We shall refer to both the $\phi_{i,t^*}$'s and their sum as potentials of $u$. We shall refer to both the $\psi_i$'s and their weighted sum as streamfunctions of $u$.

\begin{remark}\phantomsection\label{uunique}
This $u$ is uniquely determined (in $L_R^2$) by $q,\,\dot{q}$ and the $\gamma_i$'s. Indeed if there would be two such $u$'s, then their difference would give rise to a nonzero streamfunction with zero circulation (by the same argument as in the proof of Lemma \ref{ExistencePsi}, Step 3), which is impossible by the uniqueness of the $\psi_i$'s.
\end{remark}

\subsubsection{Representation of $\de_t u$}
We will need to show that the potential and stream function are differentiable in $q$ to be able to represent $\de_tu$. The differentiability of solutions to elliptic equations with respect to changes of the underlying domain is a classical topic and we refer the reader to \cite{sokolowski1992introduction} for further reading.

\begin{lemma}\phantomsection\label{SmoothnessInq}
\begin{itemize}
\item[a)] The function $\phi_{i,\bullet}$ is smooth as a map from the tangent bundle $TM$ to $H_R^1$ (here differentiability can be understood in both the $L_{loc}^2$-sense and the pointwise sense). 
\item[b)] The derivatives $\de_q \phi_{i,t^*},\de_q^2\phi_{i,t^*}$ lie in $H_R^1\cap C^\infty(\mathcal{F})$, furthermore their $H_{R}^1$-norm is bounded locally uniformly in $q$ and $t^*$.
\item[c)] $\div(r\nabla\de_q \phi_{i,t^*})=0$.
\item[d)] \begin{align}
|\nabla^m\de_q\phi_{i,t^*}(x)|\,\lesssim\,\frac{1}{1+|x|^{2+m}}\quad \forall m\in\N_{\geq 0}.
\end{align}
here the implicit constant is bounded locally uniformly in $q$ and $t^*$.

\end{itemize}
\end{lemma}

\begin{proof}
We can identify the tangent space at every point with $\R^2$ and $\phi_{i,\bullet}$ is linear in $t^*$, hence it suffices to show smoothness in $q$ for a fixed $t^*$.\\

Step 1. We first want to apply the implicit function theorem to obtain that a derivative of $\phi_{i,t^*}$ with respect to $q$ exists. Fix some $q^0$. We use the three-dimensional $\phi^{\R^3}=\phi_{i,t^*}(r,z)$ (in axisymmetrical coordinates) again. We set \begin{align}V\,:=\, (\dot{H}^1\cap L^6)(\R^3\backslash \bigcup_j B_j^{\R^3}(q^0)),\end{align} 

and equip this with the standard inner product of $\dot{H}^1$. By the Sobolev embedding this is a Hilbert space.

In order to fit different configurations of the bodies into one space, we introduce $C^\infty$ diffeomorphisms $\Xi:\R^3\backslash\bigcup_j B_j(q^0)\rightarrow\R^3\backslash \bigcup_j B_j(q^1)$ which map each $\de B_j(q^0)$  to $\de B_j(q^1)$. We can assume that the family $\Xi$ is smooth in the parameter $q^1$ since the $B_i$'s are. We may also assume that $\Xi$ is the identity outside a large ball depending on $q$, but bounded locally uniformly in $q$. 

Then $\phi^{\R^3}$ is harmonic on $\R^3\backslash\bigcup_j B_j(q^1)$ with Neumann boundary values $u(t^*,q^1)$ iff the function $\hat{\phi}:=\phi\circ \Xi$ fulfills \begin{align}\label{PullbackEq}
\int_{\mathcal{F}(q^0)^{\R^3}}\scalar{\nabla\hat{\phi}(\text{D}\Xi)^{-1}}{\nabla\eta(\text{D}\Xi)^{-1}}|\det \text{D}\Xi|\dx\eq-\int_{\bigcup \de B_j(q^0)^{\R^3}}\frac{\rho_j(q^1)}{\rho_j(q^0)}u(t^*,q^1)\eta\dx,
\end{align}
for all $\eta\in V$, where we have written the inner radius $\rho_j$ as a function of $q$. We may interpret the difference of the left- and right-hand side as a map $\mathcal{G}:M\times V\rightarrow V^*$. 

Since $\Xi$ is smooth in $q$ and compactly supported, we obtain that this map is Fr\'echet-smooth. Furthermore we have that \begin{align}D_V\mathcal{G}(q^0,\hat{\phi})\cdot\delta\phi\eq\int_{\mathcal{F}(q^0)^{\R^3}} \scalar{\nabla\delta\phi}{\nabla\cdot}\dx.\end{align} This is an isomorphism by the Riesz representation theorem. Hence we see that $\hat{\phi}$ is Fr\'echet-smooth by the implicit function theorem.

This implies that a function $\de_q\phi^{\R^3}$ exists in $V$ by the smoothness of $\Xi$ in $q$. Similarly, higher derivatives must exist. This shows a).\\

Step 2. Clearly, $\de_q \phi^{\R^3}$ must be harmonic and hence smooth away from the boundary. To see smoothness up to the boundary we differentiate \eqref{PullbackEq} with respect to $q$ at $q^0$ and obtain that \begin{align}\label{EquationDeri}
&\int_{\mathcal{F}(q^0)^{\R^3}}\scalar{\nabla\de_q\hat{\phi}(\text{D}\Xi)^{-1}}{\nabla\eta(\text{D}\Xi)^{-1}}|\det D\Xi|\dx+\notag\\
&\int_{\mathcal{F}(q^0)^{\R^3}}\scalar{\nabla\hat{\phi}\de_q((\text{D}\Xi)^{-1})}{\nabla\eta(\text{D}\Xi)^{-1}}|\det \text{D}\Xi|\dx+\\
&\int_{\mathcal{F}(q^0)^{\R^3}}\scalar{\nabla\hat{\phi}(\text{D}\Xi)^{-1}}{\de_q\left(\nabla\eta(\text{D}\Xi)^{-1}|\det \text{D}\Xi|\right)}\dx\eq-\int_{\bigcup \de B_j(q^0) }\scalar{\de_q\left(\frac{\rho_j(q)}{\rho_j(q^0)}u(t^*,q^0)\right)}{\eta}\dx\notag,
\end{align}

for all $\eta\in V$, the differentiation of this equation is justified by the differentiability of $\hat{\phi}$ and by $\Xi$ being compactly supported and smooth in $q$.

This is an elliptic equation for $\de_q\hat{\phi}$ with Neumann boundary conditions and a smooth and compactly supported term given by the second and third summand on the left hand side. Hence $\de_q\hat{\phi}$ is smooth up the boundary. The same argument can be used to show regularity of higher derivatives in $q$. Again this also shows that $\de_q\phi^{\R^3}$ (and hence also $\phi_{i,t^*}$) is smooth up to the boundary and the same is true for higher derivatives in $q$.

By the pointwise smoothness that follows from this, it is obvious that c) holds.\\

Step 3. To obtain the decay estimate for the derivative, we note that it is enough to show these estimates for $\de_q\phi^{\R^3}$. Clearly, it holds that $\Delta\de_q\phi^{\R^3}=0$.\\

We again employ Lemma \ref{decay1}, by e.g.\ going back to $\hat{\phi}$ and using Equation \eqref{EquationDeri}, it is easy to see that the boundary values are locally uniformly controlled by $q$ and $t^*$. To see that the integral over the Neumann boundary values of $\de_q\phi^{\R^3}$ is $0$, we introduce some compact $B_j'$ with smooth boundary, in which $B_j^{\R^3}$ is compactly contained and which intersects no other $B_{j'}^{\R^3}$. Then we rewrite them as \begin{align}
\int_{\de B_j^{\R^3}}\de_n\de_q\phi^{\R^3}\dx\eq -\int_{\de B_j'}\de_n\de_q\phi^{\R^3}\dx\eq-\de_q\int_{\de B_j'}\de_n\phi^{\R^3}\dx\eq0.
\end{align}

Here pulling out the derivative is justified by the regularity of $\de_q\phi^{\R^3}$.

In particular, the decay estimate also implies that the derivative is in $H^1$.

\end{proof}

\begin{lemma}
The functions $f_i$ are smooth in $R_i$, in particular, $E_{q_i}$ is smooth with respect to $q$. 
\end{lemma}
\begin{proof}This can be shown as in the previous Lemma by using a similar smooth family of diffeomorphisms.\end{proof}

\begin{lemma}\phantomsection\label{PsiDeri}
\begin{itemize}
\item[a)] The derivative of $\psi_i$ with respect to $q$ exists and is smooth up to the boundary (here the derivative can e.g.\ be taken as a classical pointwise derivative or in the $L_{loc}^2$ sense).
\item[b)] We have that $\frac{1}{r}\nabla\de_q\psi_i,\,\frac{1}{r}\nabla\de_q^2\psi_i\in L_R^2\cap C^\infty$. Furthermore the $L^2$-norm of these derivatives is bounded locally uniformly in $q$.
\item[c)] It holds that \begin{align}
\left|\de_q\psi_i(x)\right|\,\lesssim\,\frac{1}{1+|x|}
\end{align}

and \begin{align}
\left|\nabla^m\de_q\frac{1}{r}\nabla^\perp\psi_i(x)\right|\,\lesssim\,\frac{1}{1+|x|^{2+m}}\quad \forall m\in\N_{\geq0}.
\end{align}
In the second estimate the implicit constant is locally uniformly bounded in $q$.
\item[d)] The $C_{ij}$'s are differentiable with respect to $q$.
\end{itemize}
\end{lemma}
\begin{proof}
a) and d) The argument uses a similar technique as the existence proof for Lemma \ref{ExistencePsi}. First, we again consider the three-dimensional $\Psi_i$'s as in said proof, for fixed boundary values $(\tilde{C}_{ij})$, they have arbitrarily many derivatives in $q$, which are smooth up to the boundary by the same argument as in the previous proof and the derivatives are in $\dot{H}^1\cap L^6$. This also shows that for fixed $(\tilde{C}_{ij})$ there is a (smooth) derivative of $u_{2,i}$ and $\psi_{i}(\tilde{C}_{ij})$. 

It remains to argue that the $C_{ij}$'s are differentiable. To see this note that the linear map from the $(\tilde{C}_{ij})$ to the integrals $\int_{\de B_l}\frac{1}{r}\de_n\psi_i(\tilde{C}_{ij})\dx$, which was used to show existence of the $C_{ij}$'s, is differentiable in $q$ as well. Indeed we may again introduce some compact $B_l'$, which compactly contains $B_l$ and intersects no other $B_{j}$. Then we have \begin{align}
\int_{\de B_l'}\frac{1}{r}\de_n\de_q\psi_i(\tilde{C}_{ij})\dx\eq\de_q\int_{\de B_l'}\frac{1}{r}\de_n\psi_i(\tilde{C}_{ij})\dx\eq\de_q \int_{\de B_l}\frac{1}{r}\de_n\psi_i(\tilde{C}_{ij})\dx,
\end{align}

which shows differentiability. \\

As the $\dot{H}^1$-norm of $\Psi_i$ corresponds to the $L^2$-norm of $\frac{1}{r}\nabla^\perp\psi_i(\tilde{C}_{ij})$, we see the boundedness statement b).

 The decay of $\de_q\psi_i$ again follows from the fact that the derivative fulfills $\div(\frac{1}{r}\nabla\de_q\psi_i)=0$ and is smooth up to the boundary by using Lemma \ref{rep psi}.

The decay of $\de_q \frac{1}{r}\nabla^\perp\psi_i$ follows from the fact that the derivative of the three-dimensional stream function is harmonic and smooth up to the boundary as in the previous proof by Lemma \ref{decay1}, and can be controlled locally uniformly in $q$.
\end{proof}

\begin{remark}
Note that if $q$ is $C^2$ in time and $u$ is a solution of the Euler equations \eqref{Euler1}-\eqref{Euler5}, then we must have \begin{align}\label{eqDu}
\de_t u\eq \sum_{i=1}^k\de_q \phi_{i,\dot{q}}\cdot\dot{q}+\phi_{i,\ddot{q}_i}+\gamma_i\frac{1}{r}\nabla^\perp\de_q\psi_i\cdot \dot{q}.
\end{align}

Indeed this follows from the fact that $u$ is uniquely determined through $q,\dot{q}$ and the $\gamma_i$'s (Remark \ref{uunique}). In particular, $u$ has the regularity required in Condition \ref{strongsol}.

\end{remark}

\begin{lemma}\phantomsection\label{Decayp}
Assume that $u$ solves the Euler equations \eqref{Euler1}-\eqref{Euler5} with pressure $p$.
Then we have that \begin{align}|\nabla p(x)|\,\lesssim\,\frac{1}{1+|x|^2},\end{align} and there is a constant $C$ which may be choosen as $0$ such that \begin{align}|p(x)-C|\,\lesssim\, \frac{1}{1+|x|}.\end{align}
The implicit constants in these estimates are bounded locally uniformly in $q,\dot{q},\ddot{q}$.
\end{lemma}
\begin{proof}
By the construction of $u$ in \eqref{u1u2} and the decay estimates in the Lemmata \ref{ExistencePhi} and \ref{ExistencePsi} we have that \begin{align}
|(u\cdot\nabla)u(x)|\,\lesssim\, \frac{1}{1+|x|^5}.
\end{align} 

By Equation \eqref{eqDu} and the Lemmata \ref{ExistencePhi},\ref{SmoothnessInq} and \ref{PsiDeri}, we see that \begin{align}
|\de_tu(x)|\,\lesssim\,\frac{1}{1+|x|^2}.
\end{align}

Hence $|\nabla p(x)|\lesssim \frac{1}{1+|x|^2}$ and the estimate is locally uniform in $q,\dot{q},\ddot{q}$, because the estimates for $u$ and its derivatives are.

Since $\int_{\de B_R(0)\cap \mathbb{H}}|\nabla p|\dx\rightarrow 0$ for $R\rightarrow\infty$ we have \begin{align}
\max_{x\in \de B_R(0)} p(x)-\min_{x\in \de B_R(0)} p(x)\,\rightarrow\, 0.
\end{align}

Furthermore $\int_1^\infty |\nabla p(x,a)|\dx<\infty$ for all $a$ for which no $B_i$ intersects this line, hence we obtain that $p$ converges to some finite value at infinity, which then gives the decay statement for $p$ by the fundamental theorem of calculus.

\end{proof}

\subsection{Derivation of an ODE for the system}
We reduce the motion of the bodies $B_i$ to an ODE whose coefficients depend on the $\psi$'s and $\phi$'s. This will also yield existence and uniqueness of solutions to the system.
In two-dimensional bounded domains, a very similar calculation can be found e.g.\ in \cite{GlassMunnierSueur2}.

We first introduce some additional terminology.

We set $\phi(t^*)=\sum \phi_{i,t_i^*}$ if $t^*=t_1^*+\dots+ t_k^*$. Furthermore we set $\psi=\sum_i\gamma_i\psi_i$.

\begin{definition}\phantomsection\label{main def}
Let $t^*=t_1^*+\dots+ t_k^*$; $s^*=s_1^*+\dots+ s_k^*$ and $w^*=w_1^*+\dots+w_k^*$ for $t_i^*,s_i^*$ and $w_i^*$ associated to $B_i$. We define  \begin{align}
&G_i(q,\gamma)\cdot t_i^*\eq\int_{\de B_i} \frac{1}{2r}\left((\de_n\psi)^2\de_n\phi_{i,t_i^*}\right)\dx\\
&(\mathcal{M}_{ij}(q)t_i^*)\cdot s_j^*\eq\int_{\mathcal{F}}r\nabla\phi_{i,t_i^*}\nabla \phi_{j,s_j^*}\dx\\
&\langle \Gamma(q),t^*,s^*\rangle\cdot w^*\eq \frac{1}{2}\sum_{ij}\Bigl( \left(\left(\de_q\mathcal{M}_{ij}\cdot s^*\right)t^*\right)\cdot w^*+\left(\left(\de_q\mathcal{M}_{ij}\cdot t^*\right)s^*\right)\cdot w^*\nonumber\\
&-\left(\left(\de_q\mathcal{M}_{ij}\cdot w^*\right)s^*\right)\cdot t^*\Bigr)\\
&(A(q,\gamma)t^*)\cdot s^*\eq\sum_i\int_{\de B_i}\Bigl(-\de_\tau\phi(s^*)\de_n\phi(t^*)+\de_\tau\phi(t^*)\de_n\phi(s^*)\Bigr)\de_n\psi\dx,
\end{align}

where in the definition of $\Gamma$, the inner dot product refers to the derivative in that direction.

Furthermore, $\mathcal{M}$ shall be the matrix made up of the blocks $\mathcal{M}_{ij}$ and $E$ shall be the diagonal matrix made up of the blocks $E_{q_i}$. Let $G\in (\R^2)^k\simeq \R^{2k}$ be the vector with the entries $G_1,\dots G_k$.

\end{definition}

\begin{theorem}
The system detailed in Section \ref{StateSpace} is equivalent to the system of ODEs given by \begin{align}\label{MainODE}
&E(q)\ddot{q}+\frac{1}{2}\dot{q}(\de_{q}E(q)\cdot\dot{q})+\mathcal{M}(q)\ddot{q}+\langle\Gamma(q),\dot{q},\dot{q}\rangle \\
&\eq G(q,\gamma)+(A(q,\gamma)\dot{q})\notag.
\end{align}

\end{theorem}

\begin{remark}
The equation can be interpreted as the geodesic equation for the metric given by $\mathcal{M}+E$, with  extra terms due to the calculation on the right hand side. The matrix $\mathcal{M}$ describes the ``added inertia'', which encodes the fact that to accelerate one of the bodies, one also has to accelerate the surrounding fluid.

\end{remark}

\begin{proof}
We argued in Remark \ref{uunique} that $u$ is uniquely determined by $q,\dot{q}$, hence it suffices to show that the family of Equations in \eqref{2ndDeri} is equivalent to this system. Let $t_i^*$ be an arbitary tangent vector associated to $B_i$.
We set $u_i^*=\nabla(\phi_{i,t_i^*})$.

 Then by Equation \ref{2ndDeri} it holds that \begin{align}
(t_i^*)^TE_{q_i}\ddot{q}_i+\frac{1}{2}\dot{q}_i^T(\de_{q_i}E_{q_i}\cdot\dot{q}_i)t_i^*\eq -\int_{\de B_i} rpu_i^*\cdot n\dx.
\end{align}

By the equation for $p$, partial integration and the identity $u\nabla u=\frac{1}{2}\nabla|u|^2+u\cdot \curl u$ it follows that  \begin{align}
&(t_i^*)^TE_{q_i}\ddot{q}_i+\frac{1}{2}\dot{q}_i^T(\de_{q_i}E_{q_i}\cdot\dot{q}_i)t_i^*\eq\int_{\mathcal{F}}r\nabla p\cdot u_i^*\dx\\
&\eq-\int _{\mathcal{F}}r\left(\de_t(u_1+u_2)+\frac{1}{2}\nabla|u_1+u_2|^2\right)\cdot u_i^*\dx.\notag
\end{align}

It follows from the decay estimates in the Lemmata \ref{ExistencePhi}, \ref{ExistencePsi} and \ref{Decayp} that there are no boundary terms from $\infty$ in this partial integration.

We now split this into the different contributions and use the proposition below to obtain the equation in the theorem, tested against $t_i^*$. Since $t_i^*$ was arbitrary this implies the statement.\end{proof} \begin{proposition}\phantomsection\label{Coefficients}
\begin{itemize}
\item[a)] We have \begin{align}
-\frac{1}{2}\int_\mathcal{F} r\nabla|u_2|^2\cdot u_i^*\dx\eq G_i(q,\gamma)\cdot t_i^*.
\end{align}
\item[b)] It holds that \begin{align}
-\int_\mathcal{F} r\left(\de_tu_2+\nabla(u_1\cdot u_2)\right)\cdot u_i^*\dx\eq (A(q,\gamma)\dot{q})\cdot t_i^*.
\end{align}
\item[c)] We have that \begin{align}
\int_\mathcal{F} r\left(\de_t u_1+\frac{1}{2}\nabla|u_1|^2\right)\cdot u_i^*\dx\eq \ddot{q}^T\mathcal{M}(q)t_i^*+\scalar{\Gamma(q),\dot{q}}{\dot{q}}\cdot t_i^*.
\end{align}

\end{itemize}

\end{proposition}

\begin{proof}
a) Using that both $u_2$ and $u_i^*$ decay like $\frac{1}{|x|^2}$ by Lemma \ref{ExistencePhi} and Lemma \ref{ExistencePsi} we may partially integrate the left-hand side to obtain equality with \begin{align}
\int_{\de \mathcal{F}}\frac{1}{2}r|u_2|^2\de_n \phi_{i,t^*}\dx.
\end{align}

To see that this equals the definition of $G$ we note that $\de_n \phi_{i,t_i^*}$ vanishes on every boundary except $\de B_i$ and that $|u_2|=\frac{1}{r}|\nabla^\perp\sum_j \gamma_j\psi_j|=\frac{1}{r}|\de_n \sum_j \gamma_j\psi_j|$ since the tangential derivative of the $\psi$'s vanishes.\\

b) We have that \begin{align}
-\int_\mathcal{F} r\nabla(u_1\cdot u_2)u_i^*\dx\eq\int_{\de \mathcal{F}}r(u_1\cdot u_2)(u_i^*\cdot n)\dx,
\end{align}

(this partial integration is justified by the decay estimates from the Lemmata \ref{ExistencePhi} and \ref{ExistencePsi}) which by the construction of $u$ in \eqref{u1u2} equals \begin{align}
\sum_{l}\int_{\de B_i}r\left(\frac{1}{r}\de_n\sum_j\gamma_j\psi_j\right)(\de_\tau \phi_{l,\dot{q}_l})(\de_n \phi_{i,t_i^*})\dx,
\end{align}

because $u_2$ has no normal component on the boundary.\\

It holds that $\de_tu_2=\frac{1}{r}\nabla^\perp\de_t\psi$. We have that $\div(\frac{1}{r}\nabla\de_t\psi)=0$ and $\de_t\psi$ has the boundary values \begin{align}
\de_t\psi\eq\sum_j\gamma_j\de_{q} C_{jl}\cdot\dot{q}-(u_1\cdot n)\de_n\psi\quad\text{on $\de B_l$}
\end{align} as one can see by differentiating the identity $C_{jl}(q)=\psi_j(x_q)(q)$ where $x_q$ is some fixed point on $\de B_l$ whose derivative in $t$ equals $u_1\cdot n$. 

Then a partial integration, which can again be justified by the decay estimates in the Lemmata \ref{ExistencePhi} and \ref{PsiDeri}, reveals that \begin{align}
-\int_{\mathcal{F}}ru_i^*\cdot\de_tu_2\dx\eq\sum_l\sum_j\int_{\de B_l} \de_\tau\phi_{i,t_i^*}\left(\gamma_j\de_{q} C_{jl}\cdot\dot{q}-(u_1\cdot n)\de_n\psi\right)\dx
\end{align}

The first summand vanishes because $\de_q C_{jl}$ is a constant on each $\de B_l$ and this proves b).\\

c) We introduce an energy functional for the potential part of the fluid velocity:\begin{align}
\mathcal{E}_{u_1}\,:=\,\frac{1}{2}\int_{\mathcal{F}} r|u_1|^2\dx.
\end{align}

Following the approach in \cite{Munnier5}, we will show that \begin{align}\label{Claim1}
(\de_t\de_{\dot{q}}-\de_{q})\mathcal{E}_{u_1}\cdot t_i^*\,=\,\int_{\mathcal{F}}u_i^*\cdot(\de_tu_1+\frac{1}{2}\nabla|u_1|^2)\dx.
\end{align}

To prove this claim we shall use the following Lemma which can be proven exactly as in \cite{Munnier5}[Lemma 5.1]: \begin{lemma}\phantomsection\label{deriForm}
For $\eta\in \dot{H}_R^1$ let \begin{align}&\Lambda(\eta)\eq\int_{\mathcal{F}}r\scalar{\nabla\phi(\dot{q})}{\nabla \eta} \dx\end{align}

Then it holds that \begin{align}
(\de_t\de_{\dot{q}}-\de_q)(\Lambda)\eq0.
\end{align}\end{lemma}

We now note that \begin{align}
\mathcal{E}_{u_1}\eq\frac{1}{2}\Lambda(\phi(\dot{q}))
\end{align}

and that \begin{align}
\de_{\dot{q}}\mathcal{E}_{u_1}\cdot t_i^*\eq\frac{1}{2}\bigl((\de_{\dot{q}}\Lambda)(\phi(\dot{q}))\cdot t_i^*+\Lambda(\phi(t_i^*))\bigr).
\end{align}

Because $\phi(t_i^*)$ also equals $\de_{\dot{q}}\phi(\dot{q})\cdot t_i^*$, we see that \begin{align}\label{eq 5}
\de_{\dot{q}}\mathcal{E}_{u_1}\cdot t_i^*\eq(\de_{\dot{q}}\Lambda)(\phi(\dot{q}))\cdot t_i^*.
\end{align}

 Hence we obtain that \begin{align}
&(\de_t\de_{\dot{q}}-\de_{q})\mathcal{E}_{u_1}\cdot t_i^*\eq\\
&\left(\de_t\de_{\dot{q}}\cdot t_i^*-\de_q\cdot t_i^*\right)\Lambda(\phi)
+(\de_{\dot{q}}\Lambda)(\de_t\phi(\dot{q}))\cdot t_i^*
+\frac{1}{2}(\de_q\Lambda)(\phi(\dot{q}))\cdot t_i^*
-\frac{1}{2}\Lambda(\phi^\dagger),
\end{align}

where $\phi^\dagger=\de_q\phi(\dot{q})\cdot t_i^*$ and we made use of \eqref{eq 5}.

By Lemma \ref{deriForm}, the first term is $0$. By definition, the second term equals \begin{align}
\de_{\dot{q}}\Lambda(\de_t\phi(\dot{q}))\cdot t_i^*\eq\int_{\mathcal{F}}r\de_tu_1\nabla \phi(t_i^*)\dx.
\end{align}

By Reynolds transport theorem (whose usage is justified by Lemma \ref{SmoothnessInq} b)) we have that twice the third term equals \begin{align}
(\de_q\Lambda\cdot t_i^*)(\phi)\eq-\int_{\de B_i}r|u_1|^2u_i^*\cdot n\dx+\Lambda(\phi^\dagger).
\end{align}

 This yields the claim \eqref{Claim1} after another partial integration.\\

Now we use that $\mathcal{E}_{u_1}=\frac{1}{2}\mathcal{M}(q)\dot{q}\cdot\dot{q}$, which follows directly from the definition of $\mathcal{M}$. Then we may compute the Euler-Lagrange equation of this as \begin{align}
(\de_t\de_{\dot{q}}-\de_{q})\mathcal{E}_{u_1}\cdot t_i^*\eq \mathcal{M}(q)\ddot{q}\cdot t_i^*+((\de_q\mathcal{M}(q)\cdot \dot{q})\dot{q})\cdot t_i^*-\frac{1}{2}((\de_q\mathcal{M}(q)\cdot t_i^*)\dot{q})\cdot\dot{q}.
\end{align} 

The last two summands equal the Christoffel symbol $\Gamma$ as one can directly see by writing them out in components.

\end{proof}

\subsubsection{Uniqueness and Existence}
In this subsection, we show that the system is actually well-posed and that energy conservation will imply that solutions exist for all times if $q$ does not blow up.

\begin{lemma}
The coefficients $\mathcal{M},G,A, \Gamma$ are all continuously differentiable in $q$.
\end{lemma}
\begin{proof}
One can use the definition of all these terms and the Lemmata \ref{SmoothnessInq} and \ref{PsiDeri} to obtain that they are smooth in $q$. We leave the details to the reader.
\end{proof}

\begin{corollary}
\phantomsection\label{Cor:Exist} For every initial datum $q,\dot{q}$, there is some $T>0$ such that the System \eqref{MainODE} and hence also the system introduced in Section \ref{StateSpace} has a unique solution up to time $T$, which is $C^2$ in $q$. 
\end{corollary}
\begin{proof}
By the lemma above and Picard-Lindelöf, we have local existence and uniqueness if the matrix $\mathcal{M}+E$ is invertible, which follows from the fact that both $\mathcal{M}$ and $E$ are positive definite by definition.
\end{proof}

The total energy of the system is conserved: \begin{lemma}\phantomsection\label{Lem:Energy}
The kinetic energy \begin{align}\int_{\mathcal{F}}\frac{1}{2}r|u|^2\dx+\sum_i\mathcal{E}_{B_i}\end{align} 
($\mathcal{E}$ was defined in \eqref{intEnergy}) is constant in time.
\end{lemma}
\begin{proof}
By Reynolds we have that \begin{align}
&\frac{\text{d}}{\text{d}t}\int_{\mathcal{F}}\frac{1}{2}r|u|^2\dx\eq\int_{\mathcal{F}}ru\cdot\de_t u\dx-\int_{\de\mathcal{F}}\frac{1}{2}r|u|^2u\cdot n\dx.
\end{align}

Here differentiating under the integral sign is justified by the $L_R^2$-differentiability from the Lemmata \ref{SmoothnessInq} and \ref{PsiDeri}. The first integral also equals \begin{align}
\int_{\mathcal{F}}ru\cdot\de_t u\dx\eq\int_{\mathcal{F}}ru\cdot(-u\cdot\nabla u-\nabla p)\dx\eq\int_{\mathcal{F}}-\frac{1}{2}\div(ru|u|^2)-\div(rpu)\dx.
\end{align}

Applying Gauss to this and adding the second integral from the first equation we obtain the statement, as we have by Equation \eqref{energyBalance} \begin{align}
\int_{\de B_i}rpu(\dot{q}_i)\dx\eq-\frac{\text{d}}{\text{d}t}\mathcal{E}_{B_i}.
\end{align}\end{proof}

 \begin{lemma}\phantomsection\label{EnergySplitting}
The kinetic energy of the fluid decomposes into the energies $\int_{\mathcal{F}} \frac{1}{2}r(|u_1|^2+|u_2|^2)\dx$.
\end{lemma}
\begin{proof}
We have that \begin{align}
\int_{\mathcal{F}}ru_1\cdot u_2\dx\eq\sum_i\gamma_i\int_{\mathcal{F}}r\nabla\phi\frac{1}{r}\nabla^\perp\psi_i\dx\eq\sum_{i,j}\gamma_i\int_{\de B_j}\de_\tau\phi C_{ij}\dx\eq 0
\end{align}

where we abbreviated the potential of $u_1$ with $\phi$.
\end{proof}

\begin{theorem}\phantomsection\label{Thm:exist}
Solutions of the system exist until $q$ leaves any compact set, i.e.\ until either some of the bodies collide with each other or the boundary or escape to infinity.
\end{theorem}
\begin{proof}
By the energy conservation, we see that $\int r|u|^2\dx$ is bounded uniformly in time and hence by Lemma \ref{EnergySplitting} we also have that $\int r|u_1|^2\dx$ is bounded uniformly in time. Now note that there is no $q$ such that for some $t^*\neq 0$ it holds that $\nabla\phi(t^*)=0$. Hence by the continuity of the coefficients we have on compact sets \begin{align}\int r|u_1|^2\dx\,\gtrsim\, |\dot{q}|^2.\end{align} This implies that the only way the solution can blow up is if $q$ leaves any compact set.
\end{proof}

\section{Convergence of the potential part of the velocity}\label{Section3}

In this section, we will consider the limit of the potential velocity and of the interior field in order to compute the limit of the coefficients of the equation. 

We will show that all relevant main quantities converge to the corresponding two-dimensional quantities for a single body, which can be explicitly written down, and that the error is an order $\eps|\log\eps|^\ell$ smaller. Furthermore, we will show that quantities that only exist for multiple bodies are even smaller. We will also show that derivatives with respect to $q$ are an order $\eps|\log\eps|^\ell$ smaller as well. In Subsections \ref{subsec3.3} and \ref{subsec4.2} we will see that $\mathcal{M}$ and $A$ converge to the corresponding two-dimensional quantities for a single body and that $\Gamma$ and $\de_q A$ are neglectable.\\

We omit the indices of $B,C,q, R, Z, u_{int},$ etc.\ when dealing with only a single body. We identify the tangent space of $\mathcal{M}$ with $(\R^2)^k$ via the map $t^*\rightarrow (t_{R_1}^*,t_{Z_1}^*,\dots)$.

\subsection{The interior field}

For the kinetic energy of each body, we only need to consider a single body as the definition of $E_{q_i}$ (see \eqref{def E}) only depends on $B_i$. Therefore we drop the indices in this subsection.

We write $f_\eps$ for the function $f_i$, defined with the rescaling parameter $\eps$.

\begin{lemma}\phantomsection\label{InteriorField}
Consider the energy function $f_\eps$ defined in \eqref{def f}:\begin{itemize}
\item[a)]We have\begin{align}
\left|f_\eps(R)-\pi R\tilde{\rho}^2\eps^2\right|\,\lesssim\, \eps^3,
\end{align}

where the implicit constant depends locally uniformly on $R$.

\item[b)] $f_\eps(R)$ is lipschitz in $R$ with constant $\lesssim \eps^3$, locally uniformly in $R$.

\end{itemize}
\end{lemma}

In particular this implies that we have $|E_{q_i}|\approx \eps^2$ and  $|\nabla_q E_{q_i}|\lesssim \eps^3$.

\begin{proof}
a) We compare the potential of $u_{int}$ with the one of the constant speed movement.

Set $\phi_1(x)=r$, which solves the Neumann problem $\Delta \phi_1=0$, $\de_n\phi_1=e_R\cdot n$. 

Similarly, $u_{int}(e_R)$ can by definition (see \eqref{Int1}-\eqref{Int3}) be written as $\nabla \phi_2$, where $\div(r\nabla\phi_2)=0$ and $\de_n\phi_2=u(e_R)$. Testing these equations with $\phi_1-\phi_2$ we obtain that \begin{align}
&\int_{B} \scalar{\nabla \phi_1}{\nabla(\phi_1-\phi_2)}\dx\eq\int_{\de B} e_R\cdot n(\phi_1-\phi_2)\dx\\
&\int_{B}r\scalar{\nabla \phi_2}{\nabla(\phi_1-\phi_2)}\dx\eq\int_{\de B}r\left(e_R\cdot n-\frac{\rho}{2R}\right)(\phi_1-\phi_2)\dx.
\end{align}

We multiply the first equation with $R$ and subtract the second  from it, this yields that \begin{align}
&\int_{B} r\scalar{\nabla(\phi_1-\phi_2)}{\nabla(\phi_1-\phi_2)}+(R-r)\scalar{\nabla \phi_1}{\nabla(\phi_1-\phi_2)}\dx\eq\\
&\int_{\de B} \left(Re_R\cdot n-r\left(e_R\cdot n-\frac{\rho}{2R}\right)\right)(\phi_1-\phi_2)\dx.\notag
\end{align}

Note that we may add a constant to $\phi_1-\phi_2$ in the last integral because the other factor is mean-free.

Applying the Cauchy-Schwarz inequality we obtain that \begin{align}\label{L2bound1}
&\int_{B}r\left|\nabla(\phi_1-\phi_2)\right|^2\dx\leq \rho\norm{\nabla\phi_1}_{L^2(B)}\norm{\nabla(\phi_1-\phi_2)}_{L^2(B)}\\
&+\norm{Re_R\cdot n-r\left(e_R\cdot n-\frac{\rho}{2R}\right)}_{L^2(\de B)}\norm{\phi_1-\phi_2}_{L^2(\de B)/constants}.\notag
\end{align}

The last factor can be estimated by $c_{trace}\norm{\nabla(\phi_1-\phi_2)}_{L^2(B)}$, where $c_{trace}$ is the operator norm of the trace from $\dot{H}^1$ to $L^2(\de B)/constants$. By scaling one can see that this constant is $\lesssim \eps^\frac12$.\\

This gives us an upper bound on the right-hand side of \eqref{L2bound1} of  \begin{align}\label{L2bound2}
\rho\norm{\nabla(\phi_1-\phi_2)}_{L^2}\norm{\nabla\phi_1}_{L^2}+c_{trace}\norm{\nabla(\phi_1-\phi_2)}_{L^2(B)}\norm{\rho+\frac{\rho}{2R}}_{L^2(\de B)}\,\lesssim\, \eps^2\norm{\nabla(\phi_1-\phi_2)}_{L^2}.
\end{align}

Together with the observation that \begin{align}
\int_{B}r|\nabla(\phi_1-\phi_2)|^2\dx\,\lesssim\, \norm{\nabla(\phi_1-\phi_2)}_{L^2}^2
\end{align}

we obtain from \eqref{L2bound1} and \eqref{L2bound2} that \begin{align}
\norm{\nabla(\phi_1-\phi_2)}_{L^2}\,\lesssim\, \eps^2.
\end{align}

Now by definition \begin{align}
|f_\eps-\pi R\tilde{\rho}^2\eps^2|\eq \left|\int_B r\left(|\nabla\phi_2|^2-|\nabla\phi_1|^2\right)\dx\right|\,\lesssim\, \norm{\nabla(\phi_1-\phi_2)}_{L^2}\left(\norm{\nabla\phi_1}_{L^2}+\norm{\nabla\phi_2}_{L^2}\right)\,\lesssim\, \eps^3.
\end{align}
\vphantom{a}

b) We first estimate the derivative of the potential of $u_{int}$ with respect to $R$ and then use this to estimate the Lipschitz constant. We fix some $q^0=(Z^0,R^0)$ with inner radius $\rho^0$ and use the family of maps \begin{align}\Xi_q(x)\,:=\,\frac{\rho^0}{\rho}\left(x-\binom{R}{Z}\right)+\binom{R^0}{Z^0},\end{align}

which map $B(q)$ to $B(q^0)$.

Let $\phi^{q}$ be defined by $\nabla \phi^q=u_{int}^{q}(e_R)$, where the $q$ in the superscript denotes the $q$-dependence and we use the identification between the tangent space and $\R^2$ mentioned above. Let \begin{align}
\hat{\phi}\,:=\,\frac{\rho^0}{\rho}\phi^{q}\circ \Xi_q^{-1}.
\end{align}

Then a direct calculation shows that this fulfills the system \begin{align}
&\div\left(\frac{R^0}{R}\left(R+\frac{\rho}{\rho^0}\left(r-R^0\right)\right)\nabla\hat{\phi}\right)\eq 0 \text{ in $B(q^0)$}\\
&\de_n \hat{\phi}\eq e_R\cdot n-\frac{\rho}{2R} \text{ on $\de B(q^0)$}\label{Compeq2}.
\end{align}

Using e.g.\ the implicit function as in the proof of Lemma \ref{SmoothnessInq}, one can easily see that one can differentiate the solution of this equation in $R$ (w.r.t.\ the $H^1$-norm) and that the derivative fulfills the system \begin{align}
&\div\left(\frac{R^0}{R}\left(R+\frac{\rho}{\rho^0}\left(r-R^0\right)\right)\nabla\de_{R_1}\hat{\phi}\right)+\div\left(\de_{R_1}\left(\frac{R^0}{R}\left(R+\frac{\rho}{\rho^0}\left(r-R^0\right)\right)\right)\nabla\hat{\phi}\right)\eq0\text{ in $B(q^0)$}\label{Compeq1}\\
&\de_n \de_{R_1}\hat{\phi}\eq\de_{R_1}\left(e_R\cdot n-\frac{\rho}{2R}\right)\text{ on $\de B(q^0)$}\label{Compeq2},
\end{align}

here we write  $\de_{R_1}\hat{\phi}$ for the derivative with respect to the parameter $R=R_1$ in order to prevent confusion with the spatial derivative in the $R$-direction.

Now one can easily check that \begin{align}
&\left|\de_{R_1}\left(\frac{R^0}{R}\left(R+\frac{\rho}{\rho^0}\left(r-R^0\right)\right)\right)\right|\,\lesssim\, \eps\label{Deriest1}\\
&\norm{\nabla\hat{\phi}}_{L^2}\,\lesssim\, \eps\label{Deriest2}\\
&\left|\de_{R_1}\left(e_R\cdot n-\frac{\rho}{2R}\right)\right|\,\lesssim\, \eps.\label{Deriest3}
\end{align}

We can now test the equations \eqref{Compeq1},\eqref{Compeq2} with $\de_{R_1}\hat{\phi}$ and obtain after using the Cauchy-Schwarz inequality similarly as in part a) and the bounds \eqref{Deriest1}-\eqref{Deriest3} that \begin{align}
\int_{B(q^0)}\frac{R^0}{R}\left(R+\frac{\rho}{\rho^0}(r-R^0)\right)|\nabla\de_{R_1}\hat{\phi}|^2\dx\,\lesssim \,\eps^2\norm{\nabla \hat{\phi}}_{L^2}^2+\eps^{\frac{3}{2}}\norm{\de_{R_1}\hat{\phi}}_{L^2(\de B(q^0))/constants}\label{Compeq3},
\end{align} 

where we again used the mean-freeness of the boundary values to take the $L^2$-norm modulo constants.

Clearly, the prefactor in the integral on the left-hand side is $\simeq 1$. Again the operator norm of the trace operator from $\dot{H}^1$ to $L^2(\de B(q^0))/constants$ is $\simeq \eps^{\frac{1}{2}}$ by scaling, hence we obtain from \eqref{Compeq3} that \begin{align}
\norm{\nabla \de_{R_1}\hat{\phi}}_{L^2}\,\lesssim\, \eps^2
\end{align}

and this bound is locally uniform in $R$.\\

By definition it holds that \begin{align}
f_\eps(R)\eq\int_{B(q)} r|\nabla \phi^q|^2\dx\eq \int_{B(q^0)}\left(\frac{\rho}{\rho^0}\right)^2\left(R+\frac{\rho}{\rho_0}(r-R^0)\right) |\nabla\hat{\phi}|^2\dx.
\end{align}

The prefactor in the second integral is differentiable in $R$ with a derivative $\lesssim \eps$. Now we can differentiate the right-hand side under the integral by the $H^1$-differentiability of $\hat{\phi}$ and obtain from the product rule that \begin{align}
|\de_Rf_\eps(R)|\,\lesssim\, \eps^2\norm{\nabla\hat{\phi}}_{L^2}+\eps\norm{\nabla\hat{\phi}}_{L^2}^2\,\lesssim\, \eps^3.
\end{align}

This is locally uniform in $R$ as all the used estimates are.

\end{proof}

\subsection{The potential part of the velocity}
We show that the boundary values of the potential converge to the boundary values of the corresponding ``two-dimensional'' potential.

\subsubsection{The case of a single body}
\begin{definition}\phantomsection\label{def flat func}
Let $t^*,q\in \R^2$. Let $\rho>0$. Let $n$ denote the outer normal vector of  $\de B_\rho(q)$. We define a  ``two-dimensional'' potential \begin{align}
\chek{\phi}_{t^*}\eq\chek{\phi}_{t^*}(q,\rho)\,:=\,-\rho^2\frac{t^*\cdot e_1(x-q)+t^*\cdot e_2(y-q)}{(x-q)^2+(y-q)^2}.
\end{align}

One can check that this is the solution of \begin{align}
&\Delta \chek{\phi}_{t^*}\eq0\text{ in $\mathbb{R}^2\backslash B_\rho(q)$}\label{Pot1}\\
&\de_n\chek{\phi}_{t^*}\eq t^*\cdot n \text{ on $\de B_\rho(q)$}\label{Pot2}\\
&\chek{\phi}_{t^*}\rightarrow 0 \text{ at $\infty$}\label{Pot3}
\end{align}

(uniqueness of this can be found e.g.\ in \cite{Amrouche}[Thm.\ 3.1]).
\end{definition}

In order to estimate the potentials for $\eps\rightarrow 0$, we first use only a single body and again drop the indices.

Fix some $q$ and $t^*$, where we again make use of the identification of the tangent space with $\R^2$ as in the previous subsection. Furthermore, fix some $\tilde{\rho}>0$ such that $\eps\tilde{\rho}=\rho$.

We will prove a more general statement for arbitrary normal velocities, which will be useful later to estimate derivatives with respect to $q$.

It will simplify the argument to rescale everything by a factor of $\eps$. Therefore we let $\grave{B}:=B_{\tilde{\rho}}(\frac{1}{\eps}q)$ and first prove our estimates around this rescaled body.

\begin{proposition}\phantomsection\label{PotentialGeneral}
Let $b_1,b_2$ be smooth functions on $\de \grave{B}$. Further, assume that \begin{align}
\int_{\de \grave{B}}rb_1\dx\eq\int_{\de \grave{B}}b_2\dx\eq0.
\end{align}

Let $\chek{\phi}\in \dot{H}^1$ and $\grave{\phi}\in \dot{H}_R^1$ be the solutions of \begin{align}
&\Delta\chek{\phi}\eq0\text{ in $\R^2\backslash \grave{B}$}\label{def chek 1}\\
&\de_n\chek{\phi}\eq b_2 \text{ on $\de \grave{B}$}\label{def chek 2}\\
&\chek{\phi}\,\rightarrow\, 0 \text{ at $\infty$}
\end{align}

and \begin{align}
&\div(r\nabla\grave{\phi})\eq0\text{ in $\mathbb{H}\backslash \grave{B}$}\\
&\de_n\grave{\phi}\eq b_1 \text{ on $\de \grave{B}$}\\
&\de_n\grave{\phi}\eq0 \text{ on $\de\mathbb{H}$}\\
&\grave{\phi}\,\rightarrow\, 0 \text{ at $\infty$}.
\end{align}

Then for all $m\in \N_{>0}$ it holds that \begin{equation}\begin{aligned}
&\norm{\nabla^m(\grave{\phi}-\chek{\phi})}_{L^2(\de \grave{B})}\\
&\lesssim_m\, \sqrt{|\log\eps|}\left(\eps\left(\norm{b_2}_{H^{m-1}(\de \grave{B})}+\norm{b_1}_{H^{m-1}(\de\grave{B})}\right)+\norm{b_1-b_2}_{H^{m-1}(\de\grave{B})}\right).
\end{aligned}\end{equation}
The implicit constant in these estimates is bounded locally uniform in $q$.
\end{proposition}

\begin{remark}\begin{itemize}
\item[a)]Existence and uniqueness of $\grave{\phi}$ follows by Lemma \ref{ExistencePhi} and existence and uniqueness of $\chek{\phi}$ are shown in \cite{Amrouche}[Thm.\ 3.1].
\item[b)] The author strongly believes that the factor $\sqrt{|\log\eps|}$ is an artifact of the proof and can be removed by estimating $\norm{\Delta(\chek{\phi}-\grave{\phi})}_{L^2}$ instead of $\norm{\sqrt{r}\nabla(\chek{\phi}-\grave{\phi})}_{L^2}$ in the proof, which would require more effort.
\end{itemize}
\end{remark}

\begin{corollary}\phantomsection\label{Potential1}
For all $m\in \N_{>0}$ it holds that \begin{align}
\norm{\nabla^m(\phi_{t^*}-\chek{\phi}_{t^*})}_{L^2(\de B)}\,\lesssim_m\, \eps^{\frac{5}{2}-m}\sqrt{|\log\eps|}|t^*|.
\end{align}
These implicit constant is bounded locally uniformly in $q$.
\end{corollary}
\begin{proof}[Proof of the Corollary]
Observe that if we set $b_1=\eps u(t^*)(\eps\cdot)$ and $b_2=\eps t^*\cdot n$, then it holds that \begin{align}
\phi_{t^*}(\eps\cdot)\eq\grave{\phi}\text{ and } \chek{\phi}_{t^*}(\eps\cdot)\eq \chek{\phi},
\end{align}

because these fulfill the same elliptic equation. One easily sees that \begin{align}
\norm{b_1}_{H^m(\de B)}\,\lesssim \, \eps|t^*|,\quad \norm{b_2}_{H^m(\de B)}\,\lesssim \, \eps|t^*|,\quad \norm{b_1-b_2}_{H^m(\de B)}\,\lesssim \, \eps^2|t^*|.
\end{align}

The statement then follows from applying the proposition and rescaling.\end{proof}

Our strategy to prove the proposition is to again apply $L^2$ estimates as in the previous section, as the coefficients are similar close to $\grave{B}$, together with decay estimates for the far away behavior.

\begin{lemma}\phantomsection\label{basic est}
Let $b_2$ and $\chek{\phi}$ be as in the Proposition, then for all $m\in \N_{\geq 0}$ it holds that \begin{itemize}
\item[a)] \begin{align}
\norm{\nabla^m\chek{\phi}}_{L^2(\de \grave{B})}\,\lesssim_m\,\norm{b_2}_{{H}^{\max(0,m-1)}}.
\end{align}

\item[b)]
\begin{align}
|\nabla^m\chek{\phi}|(x)\,\lesssim_m\, \frac{\norm{b_2}_{L^2}}{\dist(x,\grave{B})^{1+m}}.
\end{align}

\item[c)] \begin{align}
\norm{\chek{\phi}}_{\dot{H}^m(\grave{B}+B_1(0)\backslash \grave{B})}\,\lesssim_m \, \norm{b_2}_{H^{\max(0,m-1)}}.\label{L2 est flat}
\end{align}
\end{itemize}
Here all the implicit constants are bounded locally uniformly in $q$.
\end{lemma}
\begin{proof} All three statements are well-known, we sketch the proof here for the convenience of the reader.

One can first repeat the argument in the proof of Lemma \ref{decay1} to show that $\phi$ and $\nabla\phi$ must decay like $|x|^{-1}$ resp.\ $|x|^{-2}$. 

Then by testing the PDE \eqref{def chek 1},\eqref{def chek 2} with $\chek{\phi}$ and partially integrating, we see that \begin{align}
\int_{\de \grave{B}}b_2\chek{\phi}\dx\eq -\norm{\chek{\phi}}_{\dot{H}^1(\R^2\backslash \grave{B})}^2,
\end{align}

where the partial integration is justified by the decay of $\phi$.
 Now $b_2$ is mean-free, so by using the trace in $\grave{B}+B_1(0)\backslash \grave{B}$, we see that \begin{align}
\norm{\chek{\phi}}_{\dot{H}^1(\grave{B}+B_1(0)\backslash \grave{B})}\,\lesssim\, \norm{b_2}_{L^2(\de \grave{B})}.
\end{align}

By using elliptic regularity estimates in $\grave{B}+B_1(0)\backslash \grave{B}$, we see that for $m>0$ we have \begin{align}
\norm{\nabla^m\chek{\phi}}_{L^2(\de \grave{B})}\,\lesssim\, \norm{b_2}_{H^{m-1}(\de \grave{B})}.
\end{align}

This lets us control $\norm{\nabla \chek{\phi}}_{L^2(\de\grave{B})}$ and we can repeat the argument in the proof of Lemma \ref{decay1} to show b), where we get one order of decay less from using the two-dimensional Newtonian potential.

c) follows for $m>0$ similarly by using elliptic regularity. The estimates for $m=0$ in a) and c) follows by using the estimate b) for $\dist(x,\de\grave{B})\geq 1$ and combining the estimate on the gradient in c) with e.g.\ the Poincare inequality.

\end{proof}

\begin{remark}
In particular, by rescaling, we see that we have \begin{align}
 \norm{\de_\tau\chek{\phi}_{t^*}}_{L^2(\de B)}\simeq |t^*|\eps^{\frac{1}{2}}
\end{align}

and the same estimate holds for $\phi_{i,t^*}$ by the Proposition.
\end{remark}

\begin{lemma}\phantomsection\label{Lem:Decay1}
We have the following estimates:\begin{itemize}
\item[a)] \begin{align}
\norm{\nabla^m\grave{\phi}}_{L^2(\de\grave{B})}\,\lesssim_m\, \norm{b_1}_{H^{\max(0,m-1)}}
\end{align}
for all $m\geq 0$.
\item[b)] It holds that \begin{align}
|\nabla^m\grave{\phi}|(x)\,\lesssim_m\, \min\left(\frac{\norm{b_1}_{H^{m}}}{1+\dist(x,\grave{B})^{1+m}},\,\frac{\norm{b_1}_{H^{m}}}{\eps(1+\dist(x,\grave{B})^{2+m})}\right),\label{ decay est b}
\end{align}
for all $m\in \N_{\geq 0}$.

\item[c)] For $\dist(x,\grave{B})\geq 1$ and $m\in\N_{\geq 0}$ it holds that \begin{align}
|\nabla^m\grave{\phi}|(x)\,\lesssim_m\, \min\left(\frac{\norm{b_1}_{L^2}}{\dist(x,\grave{B})^{1+m}},\,\frac{\norm{b_1}_{L^2}}{\eps\dist(x,\grave{B})^{2+m}}\right).
\end{align}
\end{itemize}
The implicit constant in these estimates is bounded locally uniformly in $R$.
\end{lemma}
\begin{proof}
a) follows from the same argument as the previous Lemma, where we again use the decay estimate from b) or c) for the case $m=0$.

b) and c) Our strategy is to quantify the argument of Lemma \ref{decay1}.

If we extend $\grave{\phi}$ to $\grave{B}$ by solving the Dirichlet problem with boundary data $\grave{\phi}$, then by a) and elliptic regularity \begin{align}
\norm{[\de_n\grave{\phi}]}_{\mathcal{M}(\de \grave{B})}\,\lesssim\, \norm{b_1}_{L^2},\label{est jump 1}
\end{align}

where $[\cdot]$ denotes the jump across the boundary. We can then proceed as in the proof of Lemma \ref{decay1} and set $\phi^{\R^3}(R,\theta,Z)=\grave{\phi}(R,Z)$, where $(R,\theta,Z)$ are axisymmetric coordinates in $\R^3$, then as argued there it holds that \begin{align}
\phi^{\R^3}\eq\frac{-1}{4\pi|\cdot|}*[\de_n\grave{\phi}]\mathcal{H}^2\mres\de\grave{B}^{\R^3}.\label{repSol}
\end{align}

This gives the second estimate in c) for $m=0$, since $[\de_n \grave{\phi}]$ is mean-free, where the factor $\frac{1}{\eps}$ comes from the fact that $\de\grave{B}$ is of size $\approx\frac{1}{\eps}$.

 For the first estimate in c) for $m=0$ we may restrict ourselves to the case $1\leq \dist(x,\grave{B})\leq\frac{1}{\eps}$, as otherwise it follows from the second one. We use three-dimensional axisymmetric coordinates $(R,\theta,Z)$ again and fix some $x$. Then we split $\grave{B}^{\R^3}$ into parts $S,T_{-n},\dots T_n$ where we take $T_{-n}\dots T_n$ as the sets \begin{align}
T_i\,:=\,\grave{B}^{\R^3}\cap\left\{(R',\theta',Z')|\theta'-\theta_x\in \left[\frac{\pi}{6n}\left(i-\frac{1}{2}\right),\frac{\pi}{6n}\left(i+\frac{1}{2}\right)\right)\right\},
\end{align}

where $\theta_x$ denotes the azimuthal angle of $x$ and the difference is taken modulo $2\pi$.
If we set $n=\lfloor\frac{1}{\eps}\rfloor$, then each such piece has diameter $\lesssim 1$.
 We set $S:=\grave{B}^{\R^3}\backslash\bigcup_i T_i$. By the estimate \eqref{est jump 1}, the mean-freeness and the fact that the distance between $x$ and $S$ is $\gtrsim \frac{1}{\eps}+\dist(x,\grave{B})$ we have  \begin{align}
 \left|\left(\frac{1}{4\pi|\cdot|}*[\de_n\grave{\phi}]\mathcal{H}^2\mres\de S\right)(x)\right|\,\lesssim\ \frac{\norm{b_1}_{L^2}}{\eps\left(\frac{1}{\eps^2}+\dist(x,\grave{B})^2\right)}.\end{align}

For every $i$ we have \begin{align}
\left|\left(\frac{1}{4\pi|\cdot|}*[\de_n\grave{\phi}]\mathcal{H}^2\mres\de T_i\right)(x)\right|\,\lesssim\,\frac{\norm{b_1}_{L^2}}{1+|i|^2+\dist(x,\hat{B})^2},
\end{align}

where we exploited the facts that due to the rotational symmetry, the integral of the boundary values over each $T_i$ is $0$ and that $\dist(x,T_i)\gtrsim |i|+\dist(x,\grave{B})$. Summing up and using \eqref{repSol} gives \begin{align}
|\grave{\phi}(x)|\,\lesssim\, \frac{\norm{b_1}_{L^2}}{\eps\left(\frac{1}{\eps^2}+\dist(x,\grave{B})^2\right)}+\sum_{|i|\leq\frac{1}{\eps}}\frac{\norm{b_1}_{L^2}}{1+|i|^2+\dist(x,\grave{B})^2}\,\lesssim\, \frac{\norm{b_1}_{L^2}}{\dist(x,\grave{B})},
\end{align}

where we estimated the sum with the integral and used the assumption $\dist(x,\grave{B})\leq \frac{1}{\eps}$. 

This shows c) if $m=0$.

For $m>0$ the estimates in c) follow from making the same argument with the derivatives of the fundamental solution. 

We set \begin{align}\label{DefinitionD}
D\,:=\,(\grave{B}+B_1(0))\backslash \grave{B}.
\end{align}

In $D$ the estimates in b) for $m>0$ follow from using elliptic regularity and a) to obtain that \begin{align}
\norm{\nabla^m\grave{\phi}}_{L^\infty(D)}\,\lesssim\, \norm{\nabla^m\grave{\phi}}_{H^\frac{4}{3}(D)}\,\lesssim\, \norm{b_1}_{H^{m}}.
\end{align}

The case $\dist(x,\grave{B})<1$ and $m=0$ follows from this by using the estimate for $\dist(x,\grave{B})=1$ and the fundamental theorem of calculus.

\end{proof}

\begin{proof}[Proof of Proposition \ref{PotentialGeneral}]

By subtracting both PDEs and rearranging we obtain that \begin{align}\label{Elliptic1} &\int_{\mathbb{H}\backslash\grave{B}}  r\scalar{\nabla\grave{\phi}-\nabla\chek{\phi}}{\nabla \xi}+\left(r-\frac{R}{\eps}\right)\scalar{\nabla\chek{\phi}}{ \nabla\xi}\dx\eq-\int_{\de\grave{B}}  \xi\left(rb_1-\frac{R}{\eps}b_2\right)\dx,
\end{align}

for $\xi\in H^1$ compactly supported in $\mathbb{H}\backslash\grave{B}$.\\

In order to be able to use both $\chek{\phi}$ and $\grave{\phi}$ as test functions for each others equation even though one is defined on the half-space and one on the full space, we introduce smooth cutoff functions $\eta_l$, supported in $(\mathbb{H}+\frac{R}{2\eps}e_R)\cap B_{l+1}(0)$, which equal $1$ in \linebreak$(\mathbb{H}+(\frac{R}{2\eps}+1)e_R)\cap B_l(0)$ and whose derivatives have absolute value $\leq 2$ everywhere.

By testing with $\eta_l(\grave{\phi}-\chek{\phi})$ we obtain that for large enough $l$ it holds that \begin{equation}\begin{aligned} &\int_{\mathbb{H}\backslash\grave{B}} \eta_l r|\nabla\grave{\phi}-\nabla\chek{\phi}|^2+\eta_l \left(\frac{R}{\eps}-r\right)\nabla\chek{\phi}\cdot \nabla\left(\grave{\phi}-\chek{\phi}\right)+\nabla\eta_l\cdot\left(r\nabla\grave{\phi}-\frac{R}{\eps}\nabla\chek{\phi}\right)\left(\grave{\phi}-\chek{\phi}\right) \dx\\
&\eq-\int_{\de\grave{B}}  \left(\grave{\phi}-\chek{\phi}\right)\left(rb_1-\frac{R}{\eps}b_2\right)\dx.
\end{aligned}\end{equation}

By rearranging and using the Cauchy-Schwarz inequality, we obtain the inequality \begin{equation}\begin{aligned}
&\norm{\sqrt{r\eta_l}\nabla\left(\grave{\phi}-\chek{\phi}\right)}_{L^2}^2\,\leq\, \norm{\sqrt{\eta_l}\frac{|r-\frac{R}{\eps}|}{\sqrt{r}}\nabla\chek{\phi}}_{L^2}\norm{\sqrt{r\eta_l}\nabla\left(\grave{\phi}-\chek{\phi}\right)}_{L^2}+\\
&2\int_{\left([\frac{R}{2\eps},\frac{R}{2\eps}+1]\times \R\right)\cup B_{l+1}\left(0\right)\backslash B_l\left(0\right)}\left(\left|\grave{\phi}\right|+\left|\chek{\phi}\right|\right)\left(\left|r\nabla\grave{\phi}\right|+\left|\frac{R}{\eps}\nabla\chek{\phi}\right|\right)\dx\,+\\
&2R^{-\frac{1}{2}}\eps^\frac{1}{2}c_{trace}\norm{\sqrt{r\eta_l}\nabla\left(\grave{\phi}-\chek{\phi}\right)}_{L^2}\norm{rb_1-\frac{R}{\eps}b_2}_{L^2\left(\de \grave{B}\right)}\\
&=:\, (I+III)\cdot \norm{\sqrt{r\eta_l}\nabla\left(\grave{\phi}-\chek{\phi}\right)}_{L^2}+II
\end{aligned}\end{equation}

where $c_{trace}$ denotes the operator norm of the trace from $\dot{H}^{1}(\grave{B}+B_1(0)\backslash \grave{B})$ to $L^2(\de \grave{B})/constants$, which is $\lesssim 1$ and we have estimated $r^{-\frac{1}{2}}$ with $2R^{-\frac{1}{2}}\eps^{\frac{1}{2}}$ in the last summand. Here $I,II$ and $III$ stand for the factors in the first, second, and third lines of the right-hand side.\\

Using Lemma \ref{basic est}, we estimate the first term as  \begin{align}\label{term I}
&I\,\lesssim\, \left(\int_{\mathbb{H}\backslash\grave{B}}\eta_l\frac{1}{1+\dist(x,\grave{B})^4}\frac{(r-\frac{R}{\eps})^2}{r}\dx\right)^\frac{1}{2}\norm{b_2}_{L^2}\end{align}

We split into the regions $ r\in [\frac{R}{2\eps},\frac{R}{\eps}-1]\cup [\frac{R}{\eps}+1,\frac{3R}{2\eps}],r\in [ \frac{R}{\eps}-1,\frac{R}{\eps}+1]$ and $r\geq \frac{3R}{2\eps}$, for other $r$ the integrand is $0$. This gives that \eqref{term I} is

\begin{align}
&\lesssim\,\left(\int_{[1,\frac{R}{2\eps}]\times\R}\frac{1}{|x|^4}\frac{|x_1|^2}{\frac{R}{\eps}}\dx+\int_{1}^\infty\frac{1}{|x|^4}{\frac{\eps}{R}}\dx+\int_{[\frac{R}{\eps},\infty)\times \R}\frac{1}{|x|^4}x_1\dx\right)^\frac{1}{2}\norm{b_2}_{L^2}\\
&\lesssim\, \eps^\frac{1}{2}\sqrt{|\log \eps|}\norm{b_2}_{L^2}.
\end{align}

We use Lemma \ref{Lem:Decay1} and Lemma \ref{basic est} to obtain that for $l>>\frac{1}{\eps^2}$, we have \begin{align}
&II\,\lesssim\, \frac{1}{\eps}\int_{\left([\frac{R}{2\eps},\frac{R}{2\eps}+1]\times \R\right)\cup B_{l+1}(0)\backslash B_l(0)}\frac{1}{\dist(x,\grave{B})}\frac{1}{\dist(x,\grave{B})^2}\dx\norm{b_1}_{L^2}\norm{b_2}_{L^2}\\
&\lesssim\, \frac{1}{\eps}\left(\int_{B_{l+1}(0)\backslash B_l(0)}\frac{1}{|x|^3}\dx+\int_\R\frac{1}{\left(|x|+\frac{1}{\eps}\right)^3}\dx\right)\norm{b_1}_{L^2}\norm{b_2}_{L^2}\\
&\lesssim\, \left(\frac{1}{l\eps}+\eps\right)\norm{b_1}_{L^2}\norm{b_2}_{L^2}\,\lesssim\,\eps\norm{b_1}_{L^2}\norm{b_2}_{L^2}.
\end{align}

The third term can be estimated as \begin{align}
III\,\lesssim\, \eps^\frac{1}{2}\norm{rb_1-\frac{R}{\eps}b_2}_{L^2(\de \grave{B})}\,\lesssim\, \eps^\frac12\left(\norm{b_1}_{L^2(\de\grave{B})}+\frac{1}{\eps}\norm{b_1-b_2}_{L^2(\de\grave{B})}\right).
\end{align}

Hence we obtain that \begin{equation}\begin{aligned}
\norm{\sqrt{r\eta_l}\nabla(\grave{\phi}-\chek{\phi})}_{L^2}^2\,\lesssim\,& \eps^\frac{1}{2}\sqrt{|\log\eps|}\norm{\sqrt{r\eta_l}\nabla(\grave{\phi}-\chek{\phi})}_{L^2}\Bigl(\norm{b_2}_{L^2(\de\grave{B})}+\norm{b_1}_{L^2(\de\grave{B})}+\\
&\frac{1}{\eps}\norm{b_1-b_2}_{L^2(\de\grave{B})}\Bigr)
+\eps\norm{b_1}_{L^2(\de \grave{B})}\norm{b_2}_{L^2(\de\grave{B})}.
\end{aligned}\end{equation}

This implies that \begin{equation}\begin{aligned}
\norm{\sqrt{r\eta_l}\nabla(\grave{\phi}-\chek{\phi})}_{L^2}\,\lesssim\,  &\eps^\frac{1}{2}\sqrt{|\log\eps|}\left(\norm{b_2}_{L^2(\de\grave{B})}+\norm{b_1}_{L^2(\de\grave{B})}+\frac{1}{\eps}\norm{b_1-b_2}_{L^2(\de\grave{B})}\right).\label{L2 est}
\end{aligned}\end{equation}

Now we can apply elliptic regularity estimates around $\de \grave{B}$ to the elliptic equation \eqref{Elliptic1} to obtain that for $m>0$ it holds that \begin{align}
&\norm{\nabla^m(\grave{\phi}-\chek{\phi})}_{L^2(\de\grave{B})}\,\lesssim\, \norm{\frac{(r-\frac{R}{\eps})}{\frac{R}{\eps}}\nabla\chek{\phi}}_{H^{m}(D)}+\eps\norm{rb_1-\frac{R}{\eps}b_2}_{H^{m-1}(\de \grave{B})}+\norm{\nabla(\grave{\phi}-\chek{\phi})}_{L^2(D)}\\
&\lesssim\, \eps\sqrt{|\log\eps|}\left(\norm{b_2}_{H^{m-1}}+\norm{b_1}_{H^{m-1}}+\frac{1}{\eps}\norm{b_1-b_2}_{H^{m-1}}\right),
\end{align}

where the neighborhood $D$ was defined in \eqref{DefinitionD}, for the first term we used that $\norm{\chek{\phi}}_{H^m(D)}\lesssim \norm{b_2}_{H^{m-1}}$ , which follows from \eqref{L2 est flat} and in the last step we used the Estimate \eqref{L2 est} and the fact that $r\approx \eps^{-1}$. \end{proof}

\subsubsection{Multiple bodies}

We remind the reader of the convention of writing $\ell$ for irrelevant exponents.

We will work in the rescaled setting and keep $\tilde{q}$ (defined in \eqref{def regime1} and \eqref{def regime2}) fixed independently of $\eps$. We set $\grave{B}_i=\frac{1}{\eps}B_i$.

We will use the method of reflections to construct the potentials for multiple bodies from the potential of a single body.

Fix some sufficiently smooth normal velocity $b_i$ on $\de \grave{B}_i$ with $\int_{\de \grave{B}_i} rb_i\dx=0$ and let $\grave{\phi}_i^1\in H_R^1(\mathbb{H}\backslash \grave{B}_i)$ solve \begin{align}
&\div(r\nabla\grave{\phi}_i^1)\eq 0\text { in $\mathbb{H}\backslash \grave{B}_i$}\\
&\de_n \grave{\phi}_i^1\eq b_i\text{ on $\de \grave{B}_i$}\\
&\de_n\grave{\phi}_i^1\eq0 \text{  on $\de\mathbb{H}$}
\end{align}

and let $\grave{\phi}_i\in H_R^1(\mathbb{H}\backslash\bigcup_j\grave{B}_j)$ solve \begin{align}
&\div(r\nabla\grave{\phi}_i)\eq 0\text { in $\mathbb{H}\backslash \bigcup_j \grave{B}_j$}\\
&\de_n \grave{\phi}_i\eq b_i\text{ on $\de \grave{B}_i$}\\
&\de_n \grave{\phi}_i\eq0 \text{ on $\de \grave{B}_j$ for $j\neq i$}\\
&\de_n \grave{\phi}_i\eq0 \text{ on $\de\mathbb{H}$}.
\end{align}

We then add corrector functions $\grave{\phi}_{1}^2,\grave{\phi}_{2}^2\dots$ to $\grave{\phi}_i^1$, which for $j\neq i$ fulfill the equations \begin{align}
&\div(r\nabla\grave{\phi}_{j}^2)\eq0 \text{ in $\mathbb{H}\backslash \grave{B}_j$}\\
&\de_n \grave{\phi}_{j}^2\eq-\de_n \grave{\phi}_i^1 \text{ on $\de \grave{B}_j$}\\
&\de_n\grave{\phi}_j^2\eq0 \text{  on $\de\mathbb{H}$}\\
& \grave{\phi}_j^2\in H_R^1.
\end{align}

Existence and uniqueness of these follows from Lemma \ref{ExistencePhi}. We set $\grave{\phi}_i^2=0$.

These correction terms then change the normal trace at all other $\de B_l$. For this new error we can again construct corrector functions $\grave{\phi}_{1}^3,\grave{\phi}_{2}^3\dots$ with normal boundary values $-\sum_{l\neq j} \de_n \grave{\phi}_l^2$ and so on. If the sums of the errors and the corrector functions converge, the limit will be $\grave{\phi}_i$, since it is unique.

\begin{proposition}\phantomsection\label{MultiplePotentials}
This iteration scheme converges for small enough $\eps$ to the solution $\grave{\phi}_i$ in both the regimes \eqref{def regime1} and \eqref{def regime2}, the convergence is in $H_R^1(\R^3)$. Furthermore, for all $m\in \N_{>0}$ we have the following estimates, if $\eps$ is small enough: \begin{itemize}
\item[a)] For $i\neq j$ it holds that \begin{align}
\norm{\nabla^m\grave{\phi}_i}_{L^2(\de \grave{B}_j)}\,\lesssim_m\, \eps^{2}|\log\eps|^{\ell(1+m)}\norm{b_i}_{L^2}
\end{align}

\item[b)] For all $i$ it holds that \begin{align}
\norm{\nabla^m(\grave{\phi}_i^1-\grave{\phi}_i)}_{L^2(\de \grave{B}_i)}\,\lesssim_m\,\eps^{4}|\log\eps|^{\ell(1+m)} \norm{b_i}_{L^2}.
\end{align}

\end{itemize}

The implicit constant in these estimates and the requirements on $\eps$ are locally uniform in $\tilde{q}$ in both regimes \eqref{def regime1} and \eqref{def regime2}. In particular the estimates for the single body case in Proposition \ref{PotentialGeneral} and Lemma \ref{Lem:Decay1} a) are still true.
\end{proposition}

\begin{proof}

The rescaled bodies have pairwise distances $\gtrsim\frac{1}{\eps|\log\eps|}$ resp.\  $\gtrsim\frac{1}{\eps\sqrt{|\log\eps|}}$, hence by Lemma \ref{Lem:Decay1} c), for $i\neq j$ we have \begin{align}
\norm{\nabla^m\grave{\phi}_i^1}_{L^{\infty}(\de \grave{B}_j)}\,\lesssim_m\, \eps^{1+m}|\log\eps|^{\ell(1+m)} \norm{b_i}_{L^2},
\end{align}

for every $m\geq 0$.

This lets us estimate the decay of the correctors $\grave{\phi}_j^2$ by the same Lemma, which can again be used to estimate the second-order correctors and so on. Iteratively, we obtain from Lemma \ref{Lem:Decay1} that \begin{align}\label{Normal corrector}
\norm{\nabla^m\grave{\phi}_j^l}_{L^2(\de \grave{B}_j)}\,\lesssim_m\, C^l\eps^{2l-2}|\log\eps|^{\ell (l+m+1)} \norm{b_i}_{L^2}
\end{align}

here $C$ is a numerical factor, depending on $k$, but not on $l$ or $m$, coming from the implicit constant in Lemma \ref{Lem:Decay1} and the fact that we have to sum over $k$ correctors in each step. By integrating over the decay estimate in Lemma \ref{Lem:Decay1} we obtain that \begin{align}\label{L2 corrector}
\norm{\grave{\phi}_j^l}_{H_R^1(\mathbb{H}\backslash\bigcup_j \grave{B}_j)}\,\lesssim\, \frac{|\log\eps|}{\eps}\norm{\de_n\grave{\phi}_j^l}_{H^1(\de \grave{B}_j)}.
\end{align}

Therefore the scheme converges in $H_R^1$ if $\eps$ is small enough.

By Lemma \ref{Lem:Decay1}, we obtain from \eqref{Normal corrector} that\begin{align}
|\nabla^m \grave{\phi}_j^l|\, \lesssim_m\, C^l\eps^{2l+m-1}|\log\eps|^{\ell (m+l+1)}\norm{b_i}_{L^2}\text{ on $\grave{B}_n$ with $n\neq j$}
\end{align}

and for $l\geq2$  we have \begin{align}
\norm{\nabla^m \grave{\phi}_j^l}_{L^2(\de \grave{B}_j)}\,\lesssim_m\, \norm{\de_n \grave{\phi}_j^l}_{H^{m-1}(\de \grave{B}_j)}\,\lesssim_m\, C^l\eps^{2l-2}|\log\eps|^{\ell (m+l+1)}\norm{b_i}_{L^2},
\end{align}

where we made use of Lemma \ref{Lem:Decay1} a).

Hence by summing up, we see that for $j\neq i$ and small enough $\eps$ we have  \begin{align}
&\norm{\nabla^m\grave{\phi}_i}_{L^2(\de \grave{B}_j)}\,\lesssim_m\, \eps^{2}|\log\eps|^{\ell(1+m)}\norm{b_i}_{L^2} \\
&\norm{\nabla^m(\grave{\phi}_i-\phi_i^1)}_{L^2(\de \grave{B}_i)}\,\lesssim_m\, \eps^{4}|\log\eps|^{\ell(1+m)}\norm{b_i}_{L^2}. 
\end{align}

After rescaling back to the original balls, we obtain the statement. 
\end{proof}

\begin{corollary}\phantomsection\label{MultPotSpec}
Fix some $t^*\in T_{q_i}M$ and let $\phi_{i,t^*}^1$ be the potential for $t^*$ if there is only a single body, let $\phi_{i,t^*}$ be the potential for $k$ bodies, in one of the two regimes \eqref{def regime1} or \eqref{def regime2}. Then we have the following bounds for all $m\in \N_{>0}$, with implicit constant bounded locally uniformly in $\tilde{q}$ in both regimes:\begin{itemize}
\item[a)] For $i\neq j$ it holds that \begin{align}
\norm{\nabla^m\phi_{i,t^*}}_{L^2(\de B_j)}\,\lesssim\, \eps^{\frac{7}{2}-m}|\log\eps|^{\ell(1+m)}|t^*|.
\end{align}

\item[b)] For all $i$ it holds that \begin{align}
\norm{\nabla^m(\phi_{i,t^*}^1-\phi_{i,t^*})}_{L^2(\de B_{i})}\,\lesssim\, \eps^{\frac{11}{2}-m}|\log\eps|^{\ell(1+m)}|t^*|.
\end{align}

\end{itemize}
All these estimates hold locally uniformly in $\tilde{q}$. In particular, this implies that the estimates from the single body case in Corollary \ref{MultPotSpec}  still hold for multiple bodies.
\end{corollary}
\begin{proof}
This follows directly by rescaling the potentials as in the proof of Corollary \ref{Potential1} and using Proposition \ref{MultiplePotentials}.

\end{proof}

\subsubsection{The derivative with respect to $q$}\label{sec q deri}

In the following, we want to obtain bounds on the derivative of the potential with respect to $q$. 

As we have already shown that these are smooth in $q$ in Lemma \ref{SmoothnessInq}, it is enough to estimate partial derivatives with respect to the $R_i$'s and $Z_i$'s. In order to compare boundary data at different instances of one $B_j$, we fix some $q$ and use the affine maps \begin{align}\label{def c}
c_0(x,s)\eq\frac{\rho_j(R_j+s)}{\rho_j(R_j)}(x-q_j)+q_j+se_R
\end{align}

and \begin{align}\label{def d}
d_0(x,s)\eq x+se_Z
\end{align}

where $\rho_j(\cdot)$ denotes $\rho_j$ as a function $R_j$. As we do not want to move the other bodies, we define $c$ and $d$ as smooth maps which equal $c_0$ and $d_0$ in a neighborhood of $B_j$ and which equal the identity in a neighborhood of the other bodies. For technical reasons, will assume that the neighborhood in which $c$ and $d$ equal the terms in \eqref{def c} resp.\ \eqref{def d} has a size $>>|\log\eps|^{-2}$, which is not restrictive.

We also use the convention of writing $\phi_{i,t^*,R_j}$ resp.\ $\phi_{i,t^*,R_j+s}$ for $\phi_{i,t^*}$, defined for the position $(R_j,Z_j)$ resp.\ $(R_j+s,Z_j)$ and also use the same notation with $Z_j$ for the $Z_j$-derivative.

\begin{proposition}\phantomsection\label{prop deri}
Fix some $t^*\in T_{q_i}M$, then for all $j,l$ we have \begin{align} 
\norm{\de_s \left(\left(\de_{\tau}\phi_{i,t^*,R_j+s}\right)\circ c(\cdot,s)\right)}_{L^2(\de B_l(q))}\,\lesssim\, \eps^{\frac{3}{2}}|\log\eps|^\ell
\end{align} and \begin{align}
\norm{\de_s \left(\left(\de_{\tau}\phi_{i,t^*,Z_j+s}\right)\circ d(\cdot,s)\right)}_{L^2(\de B_l(q))}\,\lesssim\, \eps^{\frac{3}{2}}|\log\eps|^\ell.\end{align}

Here $\tau=n^\perp$ refers to the tangent both on $B_l(q)$ and on the translated body. The implicit constant is bounded locally uniformly in $\tilde{q}$ in both the regimes \eqref{def regime1} and \eqref{def regime2} in a neighborhood of $s=0$.
\end{proposition}
\begin{proof}
We assume $|t^*|=1$ and omit the indices $i$ and $t^*$. We first show the statement for the derivative in the $R$-direction and explain in the end why the $Z$-derivative works with the same argument. Without loss of generality, we may also assume that $Z_j=0$, as the system is invariant under translation.

We rescale by a factor $\frac{1}{\eps}$ again and use the same notation as in the previous proofs. Clearly, we have $\de_{R_j}\grave{\phi}=\de_{R_j}\phi(\eps\cdot)$, in particular, the derivative exists and it holds that \begin{align}
&\div(r\nabla\de_{R_j}\grave{\phi})\eq0.\label{eq deri}\\
&\de_n\de_{R_j}\grave{\phi}\eq0 \text{ on $\de\grave{B}_l$ for $l\neq j$ and on $\de\mathbb{H}$}.\label{boundary deri 1}
\end{align}

We first want to compute the remaining normal derivative. For this we introduce \begin{align*}\grave{c}(x,s)\eq\frac{1}{\eps}c(\eps x,s).\end{align*}

For $x\in \de\grave{B}_j$ we have \begin{equation}\begin{aligned}\label{s deri 1}
&\de_s\left(\de_n\grave{\phi}_{R_j+s}(\grave{c}(x,s))\right)\big|_{s=0}\eq\eps\delta_{ij}\de_s\left(t^*\cdot n+\frac{\eps}{2(R_j+s)}\sqrt{\frac{\tilde{\rho}_j^2R_j}{R_j+s}}t^*\cdot e_R\right)\bigg|_{s=0}\eq \\
&-\eps^2\delta_{ij}\frac{3\tilde{\rho}_j}{4R_j^2}t^*\cdot e_R.
\end{aligned}\end{equation}

Here the expressions for the normal boundary values follow from Definition \eqref{normal velo} and Assumption \eqref{def rho} after rearranging. This is $\lesssim \eps^2$ in any $H^m$-norm. 

On the other hand, by the product rule, on $\de\grave{B}_j$ we also have \begin{equation}\begin{aligned}\label{s deri 2}
&\de_s\left(\de_n\grave{\phi}_{R_j+s}(\grave{c}(x,s))\right)\big|_{s=0}\\
&\eq\de_s n(\grave{c}(x,s))\big|_{s=0}\cdot\nabla\grave{\phi}_{R_j}(x)+n\cdot\de_{s}\nabla\grave{\phi}_{R_j+s}(x)\big|_{s=0}+n\cdot \nabla^2\grave{\phi}_{R_j}(x)\de_s\grave{c}(x,s)\big|_{s=0}.
\end{aligned}\end{equation}

Here $n(\grave{c}(x,s))$ stands for the normal at $\grave{c}(x,s)$ of the rescaled body centered at $\frac{1}{\eps}(R_j+s)$, which is in fact constant, and hence the first summand drops out. 

We have by definition \begin{align}
\de_s\grave{c}(x,s)\big|_{s=0}\eq\frac{1}{\eps}e_R-\frac{1}{2R_j}(x-q_j).
\end{align}

Combining \eqref{s deri 1} and \eqref{s deri 2}, we see that on $\de\grave{B}_j$ we have \begin{align}
\de_n\de_{s}\grave{\phi}_{R_j+s}\big|_{s=0}\eq - \frac{1}{\eps}\de_n\de_r\grave{\phi}_{R_j}+\frac{\tilde{\rho}_j}{2R_j}\de_n^2\grave{\phi}_{R_j}-\eps^2\delta_{ij}\frac{3\tilde{\rho}_j}{4R_j^2}t^*\cdot e_R.\label{boundary deri 2}
\end{align}

We now estimate the derivatives of the boundary values on  $\de\grave{B}_j$ and on $\de\grave{B}_l$ for $l\neq j$ differently.\\

1. Case: $l\neq j$. Here $\grave{c}$ is the identity on $\de\grave{B}_l$ and hence we only need to estimate $\nabla\de_s\grave{\phi}_{R_j+s}$.

By rescaling we know from Corollary \ref{MultPotSpec} and Lemma \ref{Lem:Decay1} a) that \begin{align}
\norm{\nabla^m\grave{\phi}}_{L^2(\de\grave{B}_j)}\,\lesssim_m\, \eps.
\end{align}

By \eqref{boundary deri 2} it follows that \begin{align}
\norm{\de_n\de_{s}\grave{\phi}_{R_j+s}\big|_{s=0}}_{L^2(\de\grave{B}_j)}\,\lesssim\, 1.\label{bd boundary deri}
\end{align}

Hence we conclude from \eqref{eq deri},\eqref{boundary deri 1},\eqref{bd boundary deri} and by Proposition \ref{MultiplePotentials} that for $m\in \N_{>0}$ it holds\begin{align}
\norm{\nabla^m\de_s\grave{\phi}_{R_j+s}\big|_{s=0}}_{L^2(\de\grave{B}_l)}\,\lesssim_m\, \eps^{2}|\log\eps|^{\ell(m+1)}.
\end{align}

By rescaling back to the original $B$'s, one obtains the statement.\\

2. Case $l=j$.\\

We build an auxiliary function that equals the desired derivative.

Consider $\de_r\grave{\phi}$. It fulfills the PDE \begin{align}
\div(r\nabla\de_r\grave{\phi})\eq\div\left(-\nabla\grave{\phi}+\frac{r\eps}{R_j}\nabla\grave{\phi}\right) \label{r eq}
\end{align}

and has the Neumann boundary values $\de_n\de_r\grave{\phi}$ on $\de\grave{B}_j$.

Consider the function $x\cdot\nabla\grave{\phi}$, a direct calculation shows that  \begin{align}
&\div(r\nabla(x\cdot\nabla\grave{\phi}))\eq0\label{n eq}\\
&\de_n(x\cdot\nabla\grave{\phi})\eq\de_n\grave{\phi}+\frac{R_j}{\eps}\de_n\de_r\grave{\phi}+\tilde{\rho}_j\de_n^2\grave{\phi} \text{ on $\de\grave{B}_j$}.
\end{align}

Here we used the assumption that $Z_j=0$.

Now consider the function \begin{align}\label{def tilde}
\tilde{\phi}\,:=\,\de_{s}\grave{\phi}_{R_j+s}\big|_{s=0}-\frac{1}{2R_j}x\cdot\nabla\grave{\phi}+\frac{1}{2R_j}\grave{\phi}+\frac{3}{2\eps}\de_r\grave{\phi}.
\end{align}

and set \begin{align}
\hat{\phi}_s(x)\eq\sqrt{\frac{R_j+s}{R_j}}\grave{\phi}_{R_j+s}(\grave{c}(s,x)).
\end{align}

Then in a neighborhood of $\grave{B}_j$ it holds that $\de_s\hat{\phi}_s|_{s=0}=\tilde{\phi}$ and hence \begin{align}
\de_s\left((\nabla\grave{\phi}_{R_j+s})(\grave{c}(x,s)\right)\big|_{s=0}\eq \de_s\nabla\hat{\phi}_s(x)\big|_{s=0} \eq\nabla\tilde{\phi}(x).\label{rewritting phi}
\end{align}

Hence to prove the proposition it suffices to estimate $\nabla\tilde{\phi}$.\\

Let us further reduce this to the case in which $k=1$. First consider the case where $i\neq j=l$. Then by \eqref{boundary deri 2} and by rescaling the estimates on $\phi$ in Corollary \ref{MultPotSpec} we have  \begin{align}
 \norm{\de_n\de_s\grave{\phi}_{R_j+s}\big|_{s=0}}_{H^m(\de\grave{B}_j)}\,\lesssim\, \frac{1}{\eps}\norm{\nabla^2\grave{\phi}}_{H^{m}(\de\grave{B}_j)}\,\lesssim_m\,\eps^{2}|\log\eps|^\ell.
 \end{align}
 
Hence we conclude from Proposition \ref{MultiplePotentials} and \eqref{eq deri} and \eqref{boundary deri 1} that \begin{align}
\norm{\de_s\grave{\phi}_{R_j+s}\big|_{s=0}}_{H^m(\de\grave{B}_j)}\,\lesssim_m\, \eps^{2}|\log\eps|^\ell.
\end{align}

 Hence by the definition of $\tilde{\phi}$ (see \eqref{def tilde}) and Corollary \ref{MultPotSpec} we have \begin{align}
 \norm{\nabla\tilde{\phi}}_{H^m(\de\grave{B}_j)}\,\lesssim_m\, \norm{\de_s\grave{\phi}_{R_j+s}\big|_{s=0}}_{H^m(\de\grave{B}_j)}+\frac{1}{\eps}\norm{\nabla^2\grave{\phi}}_{H^m(\de\grave{B}_j)}+\norm{\nabla\grave{\phi}}_{H^m(\de\grave{B}_j)}\,\lesssim_m\,\eps^{2}|\log\eps|^\ell.
\end{align}
 
After rescaling back to the original $B$'s, this shows the statement in the case $i\neq j=l$.\\

Next, consider the case $i=j=l$ and $k>1$, and let $\grave{\phi}^1$ be the rescaled potential for a single body and let $\tilde{\phi}^1$ be the version of $\tilde{\phi}$ from the single body case. As the formulas \eqref{boundary deri 1} and \eqref{boundary deri 2} hold for both $\grave{\phi}$ and $\grave{\phi}^1$, we note that we have \begin{align}
\norm{\de_n\de_{s}(\grave{\phi}_{R_j+s}-\grave{\phi}_{R_j+s}^1)\big|_{s=0}}_{H^m(\de\grave{B}_j)}\,\lesssim\, \frac{1}{\eps}\norm{\nabla^2(\grave{\phi}-\grave{\phi}^1)}_{H^m(\de\grave{B}_j)}\,\lesssim_m \, \eps^{4}|\log\eps|^\ell,
\end{align}

here we used Proposition \ref{MultiplePotentials} b). Hence we conclude that \begin{align}
\norm{\nabla^m\de_{s}(\grave{\phi}_{R_j+s}-\grave{\phi}_{R_j+s}^1)\big|_{s=0}}_{L^2(\de\grave{B}_j)}\,\lesssim_m\,\eps^{4}|\log\eps|^\ell.
\end{align}

By rescaling Corollary \ref{MultPotSpec} b) and the definition of $\tilde{\phi}$  we conclude that \begin{align}
&\norm{\nabla(\tilde{\phi}-\tilde{\phi}^1)}_{H^m(\de\grave{B}_j)}\\
&\lesssim_m\, \norm{\nabla^m\de_{s}(\grave{\phi}_{R_j+s}-\grave{\phi}_{R_j+s}^1)\big|_{s=0}}_{L^2(\de\grave{B}_j)}+ \frac{1}{\eps}\norm{\nabla^2(\grave{\phi}-\grave{\phi}^1)}_{H^m(\de\grave{B}_j)}+\norm{\nabla(\grave{\phi}-\grave{\phi}^1)}_{H^m(\de\grave{B}_j)}\\
&\lesssim_m\, \eps^{4}|\log\eps|^\ell.
\end{align}

This shows the desired estimate once we have proven the case $k=1$ after rescaling back to the original bodies. \\

It remains to deal with the case $k=1$. We drop the indices on the $\grave{B}$'s.

By \eqref{boundary deri 2},\eqref{r eq}-\eqref{def tilde} we have  \begin{align}
&\div(r\tilde{\phi})\eq-\frac{3}{2}\div\left(\left(\frac{1}{\eps}-\frac{r}{R_j}\right)\nabla\grave{\phi}\right)\text{ in a neighborhood of $\grave{B}$}\label{tilde eq 1}\\
&\de_n\tilde{\phi}\eq-\eps^2\frac{3\tilde{\rho}_j}{4R_j^2}t^*\cdot e_R\label{tilde eq 2}.
\end{align}

Set $\tilde{B}:=B_{\frac{1}{\eps|\log\eps|^2}}(0)+\grave{B}\backslash\grave{B}$. By assumption the cutoff in the definition of $\grave{c}$ happens outside of $\tilde{B}$ and hence \eqref{tilde eq 1} holds in all of $\tilde{B}$.

 The crucial observation is that $\de_{R_j}\grave{\phi}$ will decay an order faster than expected.

\begin{lemma}\phantomsection\label{Decay deri}
For $x$ with $\dist(x,\grave{B})\geq 1$ it holds that \begin{align}
|\de_s\phi_{R_j+s}(x)|\,\lesssim\, \frac{\eps+\norm{\nabla\tilde{\phi}}_{L^2(\tilde{B})}}{\eps\dist(x,\grave{B})^2}+\frac{1}{\eps\dist(x,\grave{B})^3}.
\end{align}

and \begin{align}
|\nabla\de_s\tilde{\phi}_{R_j+s}(x)|\,\lesssim\, \frac{\eps+\norm{\nabla\tilde{\phi}}_{L^2(\tilde{B})}}{\eps\dist(x,\grave{B})^3}+\frac{1}{\eps\dist(x,\grave{B})^4}.
\end{align}
The implicit constant is bounded locally uniformly in $\tilde{q}$.

\end{lemma}
\begin{proof}
We only show the first estimate, the second works completely similarly, using the derivative of the fundamental solution.

The proof builds on the idea of the proof of Lemma \ref{decay1}. Recall that in the proof there, we extended the solution to the full space $\R^3$ and used the fundamental solution. We first need to show an estimate on the Neumann boundary values of the solution to the interior problem. As in the proof of Lemma \ref{decay1}, we extend $\grave{\phi}$ to $\mathbb{H}$ by solving the Dirichlet problem for $\div(r\nabla\cdot)$ in $\grave{B}$. We also set \begin{align}
\hat{\phi}_s(x)\,:=\,\sqrt{\frac{R_j+s}{R_j}}\grave{\phi}_{R_j+s}(\grave{c}(x,s))
\end{align}

inside of $\grave{B}$.

We claim that \begin{align}
\norm{\de_s((\de_{n_{int}}\grave{\phi}_{R_j+s})\circ c)}_{L^2(\de\grave{B})}\,\lesssim\, \norm{\tilde{\phi}}_{L^2(\tilde{B})}+\eps^2,\label{inner bd}
\end{align}

here $\de_{n_{int}}$ denotes the normal derivative from the inside.\\

By applying elliptic regularity estimates to the equations \eqref{tilde eq 1}-\eqref{tilde eq 2} and using the identity $\de_s\hat{\phi}_s=\tilde{\phi}$, we see that \begin{equation}\begin{aligned}
&\norm{\nabla\de_s\hat{\phi}_s\big|_{s=0}}_{L^2(\de\grave{B})}\eq\norm{\nabla\tilde{\phi}}_{L^2(\de\grave{B})}\\
&\lesssim\, \norm{\nabla\tilde{\phi}}_{L^2(\tilde{B})}+\norm{\de_n\tilde{\phi}}_{L^2(\de\grave{B})}+\eps\norm{\left(\frac{1}{\eps}-\frac{r}{R_j}\right)\nabla\grave{\phi}}_{H^1(\grave{B}+B_1(0)\backslash\grave{B})}\,\lesssim\,\eps^2+\norm{\nabla\tilde{\phi}}_{L^2(\tilde{B})}.\label{deri dir bd}
\end{aligned}\end{equation}

In $\grave{B}$ we have \begin{align}
\div\left(\frac{\grave{c}(x,s)_R}{R_j+s}\nabla\hat{\phi}_s\right)\eq0,
\end{align}

as one can see from a direct computation, where $\grave{c}(x,s)_R$ denotes the $R$-component of $\grave{c}$. Using for instance the implicit function theorem as in the proof of Lemma \ref{SmoothnessInq}, one can justify that this equation can be differentiated in $s$.

Hence we obtain that $\de_s\hat{\phi}_s(x)$ fulfills the equation \begin{align}
&\div\left(r\nabla\de_s\hat{\phi}_s\right)\eq-\div\left(\de_s\frac{\grave{c}(x,s)_R}{R_j+s}\nabla\grave{\phi}\right).
\end{align}

Hence by elliptic regularity for this equation and \eqref{deri dir bd} we obtain that \begin{equation}\begin{aligned}
&\norm{\de_{n_{int}}\de_s\hat{\phi}_s}_{L^2(\de\grave{B})}\,\lesssim\, \norm{\nabla\de_s\hat{\phi}_s}_{L^2(\de\grave{B})}+\eps\norm{\de_s\frac{\grave{c}(x,s)_R}{R_j+s}\nabla\grave{\phi}}_{H^1(\grave{B})}\\
&\lesssim\, \eps^2+\norm{\nabla\tilde{\phi}}_{L^2(\tilde{B})}+\eps\norm{\de_n\grave{\phi}}_{H^1(\de\grave{B})}\,\lesssim\,\eps^2+\norm{\nabla\tilde{\phi}}_{L^2(\tilde{B})}.
\end{aligned}\end{equation}

The Claim \eqref{inner bd} now follows from the fact that $\nabla\hat{\phi}_s=\left(\nabla\grave{\phi}_{R_j+s}\right)\circ\grave{c}(\cdot,s)$.

Now fix some $y\in \tilde{B}$. Then as argued in the proof of Lemma \ref{decay1} we have \begin{align}\label{deri conv}
\de_{s}\grave{\phi}_{R_j+s}(y)\eq \de_s\left([\de_n\grave{\phi}_{R_j+s}]\mathcal{H}^2\mres\de\grave{B}(R_j+s)^{\R^3}*\frac{-1}{4\pi|\cdot|}\right)(y),
\end{align}

where the convolution is taken with respect to three-dimensional coordinates and $\grave{B}(R_j+s)^{\R^3}$ is the axisymmetric torus corresponding to $\grave{B}$ defined with respect to $R_j+s$ and the factor $r$ in the jump disappears due to the coordinate change. We shall also view the curves $\grave{c}(x,s)$ as curves in $\R^3$, by setting (in axisymmetric coordinates) \begin{align}
\grave{c}(x,s)\eq(\grave{c}(x,s)_R,x_\theta,\grave{c}(x,s)_Z),
\end{align}

where $x_\theta$ is the azimuthal angle of $x$.

Using these curves, the convolution in \eqref{deri conv} can be written as \begin{align}
-\de_s\int_{\de\grave{B}^{\R^3}}[\de_n\grave{\phi}_{R_j+s}](\grave{c}(x,s))\frac{\grave{c}(x,s)_R}{r}\sqrt{\frac{R_j}{R_j+s}}\frac{1}{4\pi|y-\grave{c}(x,s)|}\dx,
\end{align}

here the factor in the middle is the determinant due to the coordinate change. Now we can differentiate under the integral, as all derivatives are smooth by Lemma \ref{SmoothnessInq} and use the product rule. We estimate the derivative of the first three factors and of the last factor separately.

First, we deal with \begin{align}
-\int_{\de\grave{B}^{\R^3}}\de_s\left([\de_n\grave{\phi}_{R_j+s}](\grave{c}(x,s))\frac{\grave{c}(x,s)_R}{r}\sqrt{\frac{R_j}{R_j+s}}\right)\bigg|_{s=0}\frac{1}{4\pi|x-y|}\dx.
\end{align}

Observe that \begin{equation}\begin{aligned}
&\int_{\de\grave{B}^{\R^3}}\de_s\left([\de_n\grave{\phi}_{R_j+s}](\grave{c}(x,s))\frac{\grave{c}(x,s)_R}{r}\sqrt{\frac{R_j}{R_j+s}}\right)\bigg|_{s=0}\dx\\
&=\, \de_s \int_{\de\grave{B}(R_j+s)^{\R^3}}[\de_n\grave{\phi}_{R_j+s}]\dx\eq 0.
\end{aligned}\end{equation}

Furthermore, we have by \eqref{inner bd} and \eqref{s deri 1} \begin{align}
\norm{\de_s\left([\de_n\grave{\phi}_{R_j+s}](\grave{c}(x,s))\right)\big|_{s=0}}_{L^2(\de\grave{B})}\,\lesssim\, \eps^2+\norm{\nabla\tilde{\phi}}_{L^2(\tilde{B})}
\end{align}

and \begin{align}
\left|\de_s\frac{\grave{c}(x,s)_R}{r}\sqrt{\frac{R_j}{R_j+s}}\right|\,\lesssim\, 1.
\end{align}

Since the boundary values of $\grave{\phi}$ are $\lesssim \eps$ in any $H^m$-norm we have \begin{align}
\norm{[\de_n\grave{\phi}]}_{L^2(\de\grave{B})}\,\lesssim\, \norm{\de_n\grave{\phi}}_{L^2(\de\grave{B})}+\norm{\grave{\phi}}_{H^1(\de\grave{B})}\,\lesssim\, \eps.\label{est jump}
\end{align}

Hence we have \begin{align}
\norm{\de_s\left([\de_n\grave{\phi}_{R_j+s}](\grave{c}(x,s))\frac{\grave{c}(x,s)_R}{r}\sqrt{\frac{R_j}{R_j+s}}\right)\bigg|_{s=0}}_{L^2(\de\grave{B})}\,\lesssim\, \eps+\norm{\nabla\tilde{\phi}}_{L^2(\de\grave{B})}.
\end{align}

Now we combine this with the mean-freeness to estimate \begin{align}
\left|\int_{\de\grave{B}^{\R^3}}\de_s\left([\de_n\grave{\phi}_{R_j+s}](\grave{c}(x,s))\frac{\grave{c}(x,s)_R}{r}\sqrt{\frac{R_j}{R_j+s}}\right)\bigg|_{s=0}\frac{1}{4\pi|x-y|}\dx\right|\,\lesssim\, \frac{\eps+\norm{\nabla\tilde{\phi}}_{L^2(\de\grave{B})}}{\eps\dist(y,\grave{B})^2}.
\end{align}

It remains to estimate the second summand, which is  \begin{align}
\int_{\de\grave{B}^{\R^3}}[\de_n\grave{\phi}]\nabla_x\frac{1}{4\pi|y-x|}\cdot \de_s\grave{c}(x,0)\dx.
\end{align}

Note that $[\de_n\grave{\phi}]$ is mean-free and one easily sees that $|\de_s\grave{c}(x,0)|\lesssim \frac{1}{\eps}$ and $|\nabla\de_s\grave{c}(x,0)|\lesssim 1$. Hence by exploiting the mean-freeness of $[\de_n\grave{\phi}]$ and using \eqref{est jump}, we see that \begin{align}
&\left|\int_{\de\grave{B}^{\R^3}}[\de_n\grave{\phi}]\nabla_x\frac{1}{4\pi|y-x|}\cdot \de_s\grave{c}(x,0)\dx\right|\,\lesssim\, \frac{\norm{\de_s\grave{c}(0)}_{\sup}}{\dist(y,\grave{B})^3}+\frac{\norm{\nabla_x\de_s\grave{c}(0)}_{\sup}}{\dist(y,\grave{B})^2}\\
&\lesssim\, \frac{1}{\eps\dist(y,\grave{B})^3}+\frac{1}{\dist(y,\grave{B})^2}.
\end{align}

The statement follows because $\dist(y,\grave{B})<<\eps^{-1}$.

\end{proof}

We continue with the proof of Proposition \ref{prop deri}. By testing the elliptic equation \eqref{tilde eq 1},\eqref{tilde eq 2} with $\tilde{\phi}$ in $\tilde{B}$ and using partial integration and the Cauchy-Schwarz inequality, we obtain that \begin{align}\label{est tilde}
&\norm{\sqrt{r}\nabla\tilde{\phi}}_{L^2(\tilde{B})}^2\,\leq\, \frac{3}{2}\norm{\frac{\left(\frac{1}{\eps}-\frac{r}{R_j}\right)}{\sqrt{r}}\nabla\grave{\phi}}_{L^2(\tilde{B})}\norm{\sqrt{r}\nabla\tilde{\phi}}_{L^2(\tilde{B})}\\
&-\int_{\de \grave{B}}\left(\frac{3}{2}\left(\frac{1}{\eps}-\frac{r}{R_j}\right)\de_n\grave{\phi}+r\de_n\tilde{\phi}\right)\tilde{\phi}\dx+\int_{\de \tilde{B}\backslash\de\grave{B}}\left(\frac{3}{2}\left(\frac{1}{\eps}-\frac{r}{R_j}\right)\de_n\grave{\phi}+r\de_n\tilde{\phi}\right)\tilde{\phi}\dx\notag\\
&:=\,I+II+III,
\end{align}

where the normal $n$ on $\de \tilde{B}\backslash\de\grave{B}$ is taken as the outer normal and $I$, $II$ and $III$ are defined in the obvious way.

Using the decay estimate from Lemma \ref{Lem:Decay1}, we can estimate \begin{align}
&I\,\lesssim\, \eps\left(\int_{B_\frac{1}{\eps|\log\eps|^2}\backslash B_1(0)}\frac{|x_1|^2}{(|x_1|+\frac{1}{\eps})|x|^4}\dx\right)^\frac{1}{2}\norm{\sqrt{r}\nabla\tilde{\phi}}_{L^2(\tilde{B})}\\
&\lesssim\, \eps\left(\int_{1}^{\frac{1}{\eps|\log\eps|^2}}\frac{1}{x(x+\frac{1}{\eps})}\dx\right)^\frac{1}{2}\norm{\sqrt{r}\nabla\tilde{\phi}}_{L^2(\tilde{B})}\,\lesssim\, \eps^\frac{3}{2}|\log\eps|^\ell\norm{\sqrt{r}\nabla\tilde{\phi}}_{L^2(\tilde{B})}.
\end{align}

To estimate the second term, we note that by partial integration we have \begin{equation}\begin{aligned}
&\int_{\de\grave{B}}\left(\frac{3}{2}\left(\frac{1}{\eps}-\frac{r}{R_j}\right)\de_n\grave{\phi}+r\de_n\tilde{\phi}\right)\dx\\
&\eq \lim_{R\rightarrow \infty}\int_{\de \left(\mathbb{H}\cap B_R(0)\right)}\left(\frac{3}{2}\left(\frac{1}{\eps}-\frac{r}{R_j}\right)\de_n\grave{\phi}+r\de_n\tilde{\phi}\right)\dx\eq0.
\end{aligned}\end{equation}

Here the last equality follows from the fact that the integrand is $0$ on $\de\mathbb{H}$ and by using the decay estimates from Lemma \ref{Lem:Decay1} and Lemma \ref{Decay deri} together with the definition of $\tilde{\phi}$.

We can then use this mean-freeness and the explicit form of the normal derivatives to estimate

 \begin{align}
II\,\lesssim\, \norm{\frac{3}{2}\left(\frac{1}{\eps}-\frac{r}{R_j}\right)\de_n\grave{\phi}+r\de_n\tilde{\phi}}_{L^2(\de\grave{B})}\norm{\nabla\tilde{\phi}}_{L^2(\tilde{B})}\,\lesssim\, \eps^{\frac{3}{2}}\norm{\sqrt{r}\nabla\tilde{\phi}}_{L^2(\tilde{B})}.
\end{align}

To estimate the third term, we use Lemma \ref{Lem:Decay1} and Lemma \ref{Decay deri} and obtain that 

 \begin{align}
&III\,\lesssim\,\frac{1}{\eps|\log\eps|^2}\sup_{x\in \de\tilde{B}\backslash\de\grave{B}}|\tilde{\phi}(x)|\left(\frac{1}{\eps}|\nabla\tilde{\phi}(x)|+\frac{1}{\eps}|\nabla\grave{\phi}(x)|\right)\\
&\lesssim\,\frac{1}{\eps} \left(\eps^2+\eps\norm{\nabla\tilde{\phi}}_{L^2(\tilde{B})}\right)\left(\eps^2+\eps\norm{\nabla\tilde{\phi}}_{L^2(\tilde{B})}\right)|\log\eps|^\ell\\
&\lesssim\, \eps^3|\log\eps|^\ell+\eps|\log\eps|^\ell\norm{\nabla\tilde{\phi}}_{L^2(\tilde{B})}.
\end{align}

Putting these back into the estimate \eqref{est tilde}, we obtain \begin{align}
\norm{\sqrt{r}\nabla\tilde{\phi}}_{L^2(\tilde{B})}\,\lesssim\, \eps^{\frac{3}{2}}|\log\eps|^\ell.
\end{align}

Now we can apply elliptic regularity estimates to \eqref{tilde eq 1},\eqref{tilde eq 2} and obtain together with the previous estimates on $\grave{\phi}$ that \begin{align}
\norm{\nabla\tilde{\phi}}_{H^m(\de\grave{B})}\,\lesssim_m\, \eps^2|\log\eps|^\ell.
\end{align}

This finishes the proof of the estimate for the $R_j$-derivative. The $Z$-derivative proceeds in the same way, except that one uses \begin{align}
\tilde{\phi}\eq\de_s\grave{\phi}_{Z_j+s}+\frac{1}{\eps}\de_Z\grave{\phi}; \quad \hat{\phi}_s\eq\grave{\phi}_{Z_j+s}(\grave{d}(s,x)).
\end{align}

\end{proof}

\subsection{The Christoffel Symbol and the added Inertia}\label{subsec3.3}

\begin{lemma}\phantomsection\label{conv Inertia}
If we identify the tangent space $T_qM$ with $\R^{2k}$, then it holds that \begin{align}
\left|\mathcal{M}-\pi\begin{pmatrix}R_1\eps^2\tilde{\rho}_1^2 &0 &\dots\\0& R_1\eps^2\tilde{\rho}_1^2& \dots\\0& 0 & R_2\eps^2\tilde{\rho}_2 &0\\
\dots\end{pmatrix}\right|\,\lesssim\, \eps^3|\log\eps|^\ell,
\end{align}

where $\mathcal{M}$ was defined in \ref{main def}.

Here the implicit constant is bounded locally uniformly in $\tilde{q}$.
\end{lemma}
\begin{proof}
By partial integration we have for $t^*\in T_{q_i}M$ and $s^*\in T_{q_j}M$ that \begin{align}
(\mathcal{M}t^*)\cdot s^*\eq-\int_{\de B_j}r \de_n\phi_{i,t^*}\phi_{j,s^*}\dx.
\end{align}

If $i\neq j$ this is $\lesssim \eps^3|\log\eps|^\ell|s^*||t^*|$ by Corollary \ref{MultPotSpec} a). If $i=j$, then it holds that \begin{equation}\begin{aligned}
&\left|\int_{\de B_i} r\de_n\phi_{i,t^*}\phi_{j,s^*}\dx-R_i\int_{\de B_i}t^*\cdot n\chek{\phi}_{s^*}\dx\right|\\
&\lesssim\, \norm{ru(t^*)}_{L^2}\norm{\phi_{i,s^*}-\chek{\phi}_{s^*}}_{L^2(\de B_i)/constants}+\norm{ru(t^*)-R_it^*\cdot n}_{L^2}\norm{\chek{\phi}_{s^*}}_{L^2(\de B_i)},
\end{aligned}\end{equation}

where we used the fact that $ru(t^*)$ is mean-free and $\chek{\phi}$ was defined in \ref{def flat func}. Now by the Corollaries \ref{Potential1} and \ref{MultPotSpec} and the definition of $u(t^*)$ this is $\lesssim \eps^3|\log\eps|^\ell|s^*||t^*|$.

Now observe that if $t^*=e_1$ and $s^*=e_2$ then $t^*\cdot n$ is orthogonal to $\chek{\phi}_{s^*}$ as the former is symmetric with respect to the $e_2$-direction while the latter is antisymmetric in that direction.

If $t^*=s^*$ then it follows from the explicit form of $\chek{\phi}_{s^*}$ that \begin{align}
\int_{\de B_i}t^*\cdot n\chek{\phi}_{s^*}\dx\eq -\pi|t^*|^2\rho_i^2.
\end{align}

\end{proof}

\begin{lemma}\phantomsection\label{conv Gamma}
It holds that \begin{align}
|\Gamma|\,\lesssim\,\eps^3|\log\eps|^\ell,
\end{align}

where the Christoffel symbol $\Gamma$ was defined in \ref{main def} and the implicit constant is bounded locally uniformly in $\tilde{q}$.
\end{lemma}
\begin{proof} 
By the definition of $\Gamma$, it suffices to estimate the derivative of $\mathcal{M}$. If we are differentiating $\mathcal{M}$ with respect to $q_l$ for $l\neq i$, then it holds that \begin{align}
\left|-\de_{q_l}\int_{\de B_i}r\de_n\phi_{i,t^*}\phi_{j,t^*}\dx\right|\,\lesssim\, \eps\norm{ru(t^*)}_{L^2}\norm{\de_{q_l}\de_\tau\phi_{j,s^*}}_{L^2}\,\lesssim\,\eps^3|\log\eps|^\ell|t^*||s^*|,
\end{align}

where we used that $u(t^*)$ does not depend on $q_l$, and used the mean-freeness of $ru(t^*)$ to estimate $\de_{q_l}\phi_{j,s^*}$ with its derivative, and in the last step we used Proposition \ref{prop deri}. If $i=l$ we can switch the roles of $i$ and $j$ unless $i=j=l$. For simplicity we only consider the derivative with respect to $R_i$ in this case, the other derivative is easier. We again use the diffeomorphism $c$, defined at the beginning of Subsection \ref{sec q deri}.

Setting $\mathcal{M}_{R_i+s}$ for $\mathcal{M}$ defined with respect to $R_i+s$ we have \begin{align}\label{inertia fixed}
(\mathcal{M}_{R_i+s}t^*)\cdot s^*\eq - \int_{\de B_i}c(x,s)_R\frac{\rho_i(R_i+s)}{\rho_i(R_i)}u(t^*)\circ c\phi_{i,s^*,R_i+s}\circ c\dx,
\end{align}

where $\rho_i(R_i)$ refers to $\rho_i$ as a function of $R_i$ and $c(x,s)_R$ is the $R$-component of $c$. Using the fundamental theorem of calculus and Proposition \ref{prop deri}, we see that \begin{align}
\norm{\de_s \left(\frac{\rho_i(R_i)}{\rho_i(R_i+s)}\phi_{i,s^*,R_i+s}\circ c\right)}_{L^2/constants}\,\lesssim\, \eps^\frac{5}{2}|\log\eps|^\ell|s^*|.\label{M deri 2}
\end{align}

Furthermore, we have \begin{align}
\norm{\de_su(t^*)\circ c}_{L^2}\,\lesssim\, \eps^\frac{3}{2},\label{M deri 1}
\end{align}

as one sees by rescaling \eqref{s deri 1}.

It is easy to see that \begin{align}
\left|\de_s\left( c(x,s)_R\frac{\rho_i(R_i+s)^2}{\rho_i(R_i)^2}\right)\right|\,\lesssim\, \eps,\label{M deri 3}
\end{align}

uniformly in $x$.

We can now use this to estimate the derivative of \eqref{inertia fixed} by the product rule: \begin{equation}\begin{aligned}
&\left|\de_s \int_{\de B_i}c(x,s)_R\frac{\rho_i(R_i+s)}{\rho_i(R_i)}u(t^*)\circ c\phi_{i,s^*,R_i+s}\circ c\dx\bigg|_{s=0}\right|\\
&\leq\, \norm{ru(t^*)}_{L^2}\norm{\de_s \left(\frac{\rho_i(R_i)}{\rho_i(R_i+s)}\phi_{i,s^*,R_i+s}\circ c\right)\bigg|_{s=0}}_{L^2/constants}\\
&\:\:\:+\norm{\de_s \left(c(x,s)_R\frac{\rho_i(R_i+s)^2}{\rho_i(R_i)^2}\right)}_{L^\infty}\norm{u(t^*)}_{L^2}\norm{\phi_{i,s^*}}_{L^2(\de B_i)}\\
&\:\:\:+\norm{r\phi_{i,s^*}}_{L^2(\de B_i)}\norm{\de_s(u(t^*)\circ c)\big|_{s=0}}_{L^2}\,\lesssim\, \eps^3|\log\eps|^\ell.
\end{aligned}\end{equation}

Here we used \eqref{M deri 1},\eqref{M deri 2},\eqref{M deri 3} and used Lemma \ref{Lem:Decay1} a) (after rescaling) to estimate $\norm{\phi_{i,s^*}}_{L^2(\de B_i)}\leq \eps^{\frac{3}{2}}$.

\end{proof}

\section{The stream function}\label{Section4}
In the section, we want to compute the asymptotic of $G$ and $A$ and their derivative with respect to $q$, which requires us to compute the asymptotic of the streamfunction. The streamfunction will, up to lower order terms, resemble the asymptotic of the Biot-Savart law \eqref{asymptotic velocity}. Plugged in the definition of $G$, the leading order term of the stream function gives $0$, so for a direct computation, one would need a higher order expansion of the stream function. 

As we need to compute a derivative of the streamfunction anyway we take the alternative approach of expressing $G$ as the derivative of the energy of the streamfunction, which gives the asymptotic of $G$ just from the highest order term of the stream function at the expense of requiring an estimate on the second derivative, which is not much more complicated than just computing the first derivative.

Unlike for the potential function, the interaction between the different bodies will matter and we will obtain an interaction term in $G$, which can be computed from the highest order parts alone.

The computation of $A$ on the other hand will be more straightforward.

Another difficulty is that the limiting object $\ln|x|$ does not lie in $H^1$, so we can not expect $L^2$-based estimates to work for the streamfunction. As we are only interested in the boundary data anyway, we will directly characterize it in terms of an integral equation on the boundary.\\

In this section, we will make massive use of the fundamental solution $K$ of the operator $\div(\frac{1}{r}\nabla\cdot)$, as introduced in \eqref{definition K}. By an abuse of notation, we will also denote the linear operator \begin{align}f\rightarrow \int_{\bigcup_i\de B_i}K(x,y)f(x)\dx\end{align} by $K$. Similarly, we will write \begin{align}\overline{K}_R(x,y)=\frac{R}{2\pi}(\log(|x-y|)-\log(8)+2-\log(R))\end{align} and also write $\overline{K}_R$ for the associated integral operator.

Recall that in Lemma \ref{rep psi}, we showed that:

The function $\frac{1}{r}\de_n \psi_i$ is a solution of the system \begin{align}
&\int_{\bigcup_j\de B_j}K(x,y)\mu(x)\dx \text{ is constant on each $B_j$}\\
&\int_{\de B_j}\mu(x)\dx\eq\delta_{ij}.
\end{align}

Our goal is to show that for a single body $\frac{1}{r}\de_n\psi$ converges to a constant by showing that the kernel $K$ converges to $\overline{K}_R$ (for which the solution of the analogous system is constant).

For multiple bodies we will show that the ``cross-terms'' in $K$ are an order lower and that the corresponding lower order terms in $\frac{1}{r}\de_n\psi_i$ are essentially given through the derivatives of $K$ itself.

Recall that the kernel $K$ can be written as \begin{align}\label{exp K}
K(x,y)\eq\frac{-1}{2\pi}\sqrt{x_Ry_R}F\left(\frac{|x-y|^2}{x_Ry_R}\right),
\end{align}

where $x_R$ and $y_R$ stand for the $R$-component and \begin{align}
F(s)\eq\int_0^\pi \frac{\cos(t)}{\sqrt{2(1-\cos(t))+s}}\dt
\end{align}

(see \cite{GallaySverak}[Section 2]). This integral cannot be elementarily evaluated, however it has a series expansion at $0$, which we will make use of:

\begin{lemma}\phantomsection\label{Series F}
For small enough $s>0$ there is an expansion \begin{align}
F(s)=-\frac{1}{2}\log(s)+\log(8)-2+\sum_{j\geq 1} a_js^j+b_js^j\log(s).\label{exp F}
\end{align}

This series has a positive radius of convergence, in particular, we also have the corresponding asymptotic for the derivatives of $F$.
\end{lemma}
\begin{proof}
The statement is known, see e.g.\ \cite{Sverak}[Footnote 101] and the computation of the explicit terms can be found there. We provide the proof of the convergence of the expansion here as we were not able to find it in the literature.

By elementary manipulations, one sees that \begin{align}
F(s)\eq(1+\frac{s}{2})\int_0^\frac{\pi}{2}\frac{1}{\sqrt{1+\frac{s}{4}-\sin^2(t)}}\dt-2\int_0^\frac{\pi}{2}\sqrt{1+\frac{s}{4}-\sin^2(t)}\dt.
\end{align}

These integrals can be rewritten by using complete elliptic integrals of the first and second kind \cite{byrd2013handbook}. These are defined as \begin{align}
&K_{elliptic}(m)=\int_0^\frac{\pi}{2}(1-m^2\sin^2(t))^{-\frac{1}{2}}\dt,\\
&E_{elliptic}(m)=\int_0^\frac{\pi}{2}(1-m^2\sin^2(t))^{\frac{1}{2}}\dt.
\end{align}

With this it holds that \begin{align}
F(s)\eq(1+\frac{s}{2})\frac{1}{\sqrt{1+\frac{s}{4}}}K_{elliptic}(\sqrt{1-\frac{s}{4+s}})-2\sqrt{1+\frac{s}{4}}E_{elliptic}(\sqrt{1-\frac{s}{4+s}}).
\end{align}

It suffices to show that the functions $K_{elliptic}(\sqrt{1-t^2})$ and $E_{elliptic}(\sqrt{1-t^2})$ have an expansion of the type \begin{align}
\sum_{j\geq 0}c_jt^{2j}+d_jt^{2j}\log (t)\label{log analytic}
\end{align}

for small enough $t$, because close to $s=0$ the functions $\frac{1}{\sqrt{1+\frac{s}{4}}}$; $\sqrt{1+\frac{s}{4}}$ and $\sqrt{\frac{s}{4+s}}$ are analytic and one easily sees that the class of functions with an expansion of the type \eqref{log analytic} with a positive radius of convergence are stable under composition and multiplication with analytic functions.

It is known \cite{schwarz1893formeln}[Page 53] that there is an expansion \begin{align}
K_{elliptic}(\sqrt{1-t^2})\eq\log\left(\frac{4}{t}\right)-2\left(\sum_{j\geq 1}\frac{1}{2j(2j-1)}\sum_{l=j}^\infty \frac{(2l)!}{2^{2l}(l!)^2}t^{2l}\right)\label{expansion K}
\end{align}

By the facts that $\sum_j\frac{1}{2j(2j-1)}<\infty$ and $\frac{(2l)!}{2^{2l}(l!)^2}<1$, we see that this converges for $|t|<1$.

By using the definition and elementary calculations, one can see that \begin{align}
E_{elliptic}(m)\eq m(1-m^2)K_{elliptic}'(m)+mK_{elliptic}(m).
\end{align}

This implies \begin{align}
E_{elliptic}(\sqrt{1-t^2})\eq t^2\sqrt{1-t^2}(K_{elliptic}(\sqrt{1-t^2}))'\frac{\sqrt{1-t^2}}{t}+\sqrt{1-t^2}K_{elliptic}(\sqrt{1-t^2})\label{identity E}
\end{align}

One then obtains the desired expansion by combining \eqref{expansion K}, \eqref{identity E} and using a binomial series for the prefactor $\sqrt{1-t^2}$.\end{proof}

We set \begin{align}
&h(s)\,:=\,F(s)+\frac{1}{2}\log(s)-\log 8+2\\
&g(x,y)\,:=\,- h\left(\frac{|x-y|^2}{x_Ry_R}\right)
\end{align}

for the remainder. For all $n$ and small enough $|x-y|$ it holds that \begin{align}\label{est g}
|\nabla^ng(x,y)|\,\lesssim_n\, |x-y|^{2-n}|\log|x-y||,
\end{align}

locally uniformly in $x_R$ and $y_R$ by Lemma \ref{Series F} above.

\subsubsection{The case of a single body}
In this subsection we drop the index $i$.

\begin{lemma}\phantomsection\label{K bd}
\begin{itemize}
\item[a)]The linear map $K$ is invertible from $L_0^2(\de B)$ to $\dot{H}^1(\de B)$ with operator norm $\lesssim 1$  for small enough $\eps$.
\item[b)] We have \begin{align}
\norm{K}_{L^2(\de B)\rightarrow L^2(\de B)}\,\lesssim\, \eps|\log\eps|.
\end{align}
\item[c)] We have that \begin{align}
\norm{K-\overline{K}_R}_{L^2(\de B)\rightarrow H^1(\de B)}\,\lesssim\, \eps|\log\eps|.
\end{align} 
and
\begin{align}
\norm{K-\overline{K}_R}_{L^2(\de B)\rightarrow L^2(\de B)}\,\lesssim\, \eps^2|\log\eps|.
\end{align}
\end{itemize}

All these estimates are locally uniform in $q$.
\end{lemma}
\begin{proof}
a) and b) Observe that for a) it is enough to show c) and to show that $\overline{K}_R$ is invertible with operator norm $\lesssim 1$, as one sees e.g.\ by using a geometric series. Similarly, for b) it is enough to show that \begin{align}
\norm{\overline{K}_R}_{L^2(\de B)\rightarrow L^2(\de B)}\,\lesssim\, \eps|\log\eps|.
\end{align}

Let $\theta\in\T:= [0,1)$ parametrize $\de B$, then we claim that the kernel  $\overline{K}_R$ acts as \begin{equation}\begin{aligned}\label{action K flat}
e^{2\pi\theta in}&\,\rightarrow\, -\frac{R}{2|n|}\eps\tilde{\rho} e^{2\pi\theta in} \text{ for $n\neq 0$}\\
1&\,\rightarrow\, -R\eps\tilde{\rho}(-\log(\eps\tilde{\rho})+\log(8)-2+\log(R)),
\end{aligned}\end{equation}

which clearly has the desired operator norm and is invertible with norm $\lesssim 1$ from $L^2(\de B)$ to $H^1(\de B)$, as one loses the factor $\eps$ again due to the derivative. 

To show the Claim \eqref{action K flat}, we first observe that the constants $\frac{1}{2\pi}(\log(8)-2+\log(R))$ act as multiplication with $\eps\tilde{\rho}(\log(8)-2+\log(R))$ on constants and maps other frequencies to $0$. The Claim \eqref{action K flat} then follows from the Lemma below.

\begin{lemma}\phantomsection\label{action log}
The kernel $\log|x-y|$ acts as the Fourier multiplier \begin{align}
&1\rightarrow 2\pi\eps\tilde{\rho}\log(\eps\tilde{\rho})\\
&e^{2\pi in\theta}\rightarrow -\frac{\eps\tilde{\rho}\pi}{|n|}e^{2\pi in\theta},
\end{align}

here $\theta\in \T=[0,1)$ is a constant speed parametrization of the boundary $\de B$.
\end{lemma}
\begin{proof}
Note that when parametrizing the boundary with $\theta\in\T$, the action of the kernel corresponds to convolution with \begin{align}
2\pi\eps\tilde{\rho}\left(\log(\eps\tilde{\rho})+\log\left|1-e^{2\pi i\theta}\right|\right).
\end{align}

We have that \begin{align}
\log|1-e^{2\pi i\theta}|\eq \Re \log1-e^{2\pi i\theta}
\end{align}

and this can be approximated in $L^2$ by $\Re\log(1-(1-\delta)e^{2\pi i\theta})$ for $\delta\searrow 0$ by e.g.\ dominated convergence. Now it holds that \begin{align}
\Re\log(1-(1-\delta)e^{2\pi i\theta})\eq -\Re\sum_{j=1}^\infty\frac{(1-\delta)^j}{j}e^{2\pi ij\theta}\eq-\sum_{j=1}^\infty\frac{(1-\delta)^j}{j}\cos\left(2\pi j\theta\right),
\end{align}

where we used the Taylor series of the logarithm around $1$.

By Plancherel, we can take the limit $\delta\searrow 0$ in this Fourier series and obtain the statement by the wellknown formula $\widehat{f_1*f_2}=\widehat{f_2}\widehat{f}_2$.

\end{proof}

Proof of part c) of Lemma \ref{K bd}: It suffices to show that the kernels $K-\overline{K}_R$ and $\de_y(K-\overline{K}_R)$ are bounded on $L^2(\de B)$. We can write \begin{equation}\begin{aligned}\label{Exp diff}
&2\pi(K-\overline{K}_R)(x,y)\eq -(\sqrt{x_Ry_R}-R)(-\log(|x-y|)+\log(8)-2+\log(R)) \\
&-\frac{1}{2}\sqrt{x_Ry_R}\log\left(\frac{x_Ry_R}{R^2}\right)+\sqrt{x_Ry_R}g(x,y),
\end{aligned}\end{equation}  

where we used the Expansion \eqref{exp F} and the definition of $\overline{K}_R$. It is easy to see that \begin{align}
\left|\sqrt{x_Ry_R}-R\right|\,\lesssim\, \eps\quad\text{and}\quad \left|\log\frac{x_Ry_R}{R^2}\right|\,\lesssim\, \eps,
\end{align}

hence one obtains $L^2$-boundedness from Schur's Lemma (cf.\ \cite{Grafakos}[Apprendix A.1]).

 Taking a $y$-derivative in \eqref{Exp diff}, we get \begin{equation}\begin{aligned}
&2\pi\de_y (K-\overline{K}_R)(x,y)\eq-\de_y (\sqrt{x_Ry_R}-R) \left(-\log(|x-y|)+\log(8)-2+\log(R)\right)\\
&+(\sqrt{x_Ry_R}-R)\de_y\log(|x-y|)-\frac{1}{2}\de_y\left(\sqrt{x_Ry_R}\log \frac{x_Ry_R}{R^2}\right)+O(|x-y||\log|x-y|).
\end{aligned}\end{equation}

Clearly, the $O$-term is bounded from $L^2$ to $L^2$ and of the desired order. It is easy to check that \begin{align}
\left|\de_y (\sqrt{x_Ry_R}-R)\right|\,\lesssim\, 1\quad\text{and}\quad\left|\de_y\sqrt{x_Ry_R}\log \frac{x_Ry_R}{R^2}\right|\,\lesssim\, 1.
\end{align}

This shows boundedness of all terms by Schur's Lemma except for \begin{align*}(\sqrt{x_Ry_R}-R)\de_y\log|x-y|\end{align*} by direct estimates.  For this we use the Lemma below to conclude.

\end{proof}

\begin{lemma}\phantomsection\label{perturbed log}
Let $j\in C^1(\de B\times \de B)$, then for all $f\in L^2(\de B)$ it holds that \begin{align}
\norm{\int_{\de B} j(x,y)f(x)\de_y\log|x-y|\dx}_{L_y^2(\de B)}\,\lesssim\, \left(\norm{j}_{\sup}+\eps\norm{j}_{C^1}\right)\norm{f}_{L^2}.
\end{align}

This estimate holds locally uniformly in $q$. 
\end{lemma}
\begin{proof}
We can write $j(x,y)=p_1(x)+p_2(x,y)|x-y|$ with $\norm{p_1}_{\sup}\lesssim \norm{j}_{\sup}$ and $\norm{p_2}_{\sup}\lesssim \norm{j}_{C^1}$. Then the kernel $p_1(x)\de_y\log|x-y|$ has operator norm $\lesssim \norm{p_1}_{\sup}$  by applying the Lemma \ref{action log} to the function $p_1f$. The other part has operator norm $\lesssim \eps\norm{p_2}_{\sup}$, as one can easily check that \begin{align}
|(x-y)\de_y\log|x-y||\,\lesssim\, 1.
\end{align}
\end{proof}

\begin{proposition}\phantomsection\label{est single body}
It holds that \begin{align}
\norm{\frac{1}{r}\de_n\psi-\frac{1}{2\pi \eps\tilde{\rho}}}_{L^2(\de B)}\,\lesssim\, \eps^\frac{1}{2}|\log\eps|,
\end{align}
where the implicit constant is bounded locally uniformly in $q$.
\end{proposition}
\begin{proof}
We have that \begin{align}
K\frac{1}{r}\de_n\psi\quad\text{and}\quad\overline K_R\frac{1}{2\pi\eps\tilde{\rho}}
\end{align} are constant on $\de B$ and $\frac{1}{r}\de_n\psi-\frac{1}{2\pi \eps\tilde{\rho}}$ is mean-free by the definition of $\psi$, hence we may subtract these two identities from each other and obtain that \begin{align}
\norm{\frac{1}{r}\de_n\psi-\frac{1}{2\pi \eps\tilde{\rho}}}_{L_0^2(\de B)}\,\lesssim\, \norm{(K-\overline{K}_R)\frac{1}{2\pi \eps\tilde{\rho}}}_{\dot{H}^1(\de B)}\,\lesssim\, \eps^\frac{1}{2}|\log\eps|,
\end{align}

here we made use of Lemma \ref{K bd} a) in the first estimate and of c) in the second.

\end{proof}

\subsubsection{Multiple bodies}
Next we consider multiple bodies again. Recall that we defined \begin{align}
L_0^2(\bigcup_i \de B_i)\,:=\,\left\{f\in L^2(\bigcup_i\de B_i)\,|\,\int_{\de B_i} f\dx\eq 0\:\:\:\forall i\right\}.
 \end{align}
 
We denote the space $H^1(\bigcup_i \de B_i)$ modulo \emph{locally} constant functions with $\dot{H}^1(\bigcup_i\de B_i)$ with the norm $\norm{\de_\tau\cdot}_{L^2(\bigcup_i\de B_i)}$ where $\tau=n^\perp$.

We set \begin{align}
\tilde{K}(x,y)\eq K(x,y)I_{\{\exists i \text{ with } x,y\in \de B_i\}}.
\end{align}

and also denote the associated linear operator with $\tilde{K}$.

\begin{lemma}\phantomsection\label{comp tilde K}
We have \begin{align}
\norm{K-\tilde{K}}_{L^2(\de B_i)\rightarrow \dot{H}^1(\de B_j)}\,\lesssim\, \eps|q_i-q_j|^{-1}
\end{align}

locally uniformly in $\tilde{q}$ (in both regimes \eqref{def regime1} and \eqref{def regime2}).
\end{lemma}
\begin{proof}
The statement is nontrivial only for $i\neq j$.
By the Expansions \eqref{exp K} and \eqref{exp F}, we have \begin{align}
|\de_yK(x,y)|\eq\left|\de_y\sqrt{x_Ry_R}F\left(\frac{|x-y|^2}{x_Ry_R}\right)+\sqrt{x_Ry_R}F'\left(\frac{|x-y|^2}{x_Ry_R}\right)\de_y\frac{|x-y|^2}{x_Ry_R}\right|\,\lesssim\, \frac{1}{|x-y|}.
\end{align} 

The statement immediately follows since the bodies have pairwise distance $\approx |q_i-q_j|$.
\end{proof}

\begin{corollary}\phantomsection\label{K inv mult}
The operator $K$ is invertible from $L_0^2(\bigcup_i \de B_i)$ to $\dot{H}^1(\bigcup_i \de B_i)$ for small enough $\eps$ with operator norm $\lesssim 1$, where the implicit constant and the smallness requirement for $\eps$ are locally uniform in  $\tilde{q}$. Furthermore, for $i\neq j$, it holds that \begin{align}
\norm{K^{-1}}_{\dot{H}^1(\de B_j)\rightarrow L_0^2(\de B_i)}\,\lesssim\, \eps|q_i-q_j|^{-1}.\label{double bd}
\end{align}
\end{corollary}
\begin{proof}
By Lemma \ref{K bd} a), invertibility holds for the operator $\tilde{K}$. By using e.g.\ a geometric series, this implies invertibility and by Lemma  \ref{comp tilde K}, we have that \begin{align}
\norm{K^{-1}-\tilde{K}^{-1}}_{\dot{H}^1(\bigcup_i\de B_i)\rightarrow L_0^2(\bigcup_i\de B_i)}\,\lesssim\, \eps|q_i-q_j|^{-1}.
\end{align}

This shows the statement, since $\norm{\tilde{K}^{-1}}_{\dot{H}^1(\de B_j)\rightarrow L_0^2(\de B_i)}=0$ by definition. 
\end{proof}

Let $\psi_i^1$ denote $\psi_i$ in case $B_i$ is the only body present.

\begin{proposition}\phantomsection\label{comp mult}
For all $i$ we have \begin{align}
\norm{\frac{1}{r}\de_n(\psi_i^1-\psi_i)}_{L^2(\de B_i)}\,\lesssim\,  \eps^\frac{3}{2}|\log\eps|^\ell,
\end{align}

where the implicit constant is bounded locally uniformly in $\tilde{q}$.
In particular the estimate from the single body case in Proposition \ref{est single body} still holds.
\end{proposition}
\begin{proof}
We have that \begin{align}
K\frac{1}{r}\de_n\psi_i\eq const \text{ on all $\de B_j$}
\end{align}

and \begin{align}
&\tilde{K}\left(\frac{1}{r}\de_n\psi_i^1\big|_{\de B_i}\right)\eq const \text{ on all $\de B_j$}.
\end{align}

By subtracting the two equations, we see that \begin{align}
K\left(\frac{1}{r}\left(\de_n\psi_i-\de_n\psi_i^1\big|_{\de B_i}\right)\right)+\left(K-\tilde{K}\right)\left(\frac{1}{r}\de_n\psi_i^1\big|_{\de B_i}\right)\eq const.
\end{align}

Now we can use Corollary \ref{K inv mult} and that $\frac{1}{r}\left(\de_n\psi-\de_n\psi_i^1\big|_{\de B_i}\right)$ is mean-free

and \eqref{double bd} and that \begin{align}(K-\tilde{K}) \left(\frac{1}{r}\de_n\psi_i^1\big|_{\de B_i}\right)\eq 0\end{align} on $\de B_i$ by definition to obtain that \begin{align}
\norm{\frac{1}{r}\left(\de_n\psi_i-\de_n\psi_i^1\big|_{\de B_i}\right)}_{L_0^2(\de B_i)}\,\lesssim\,\eps|\log\eps|^\ell\norm{(K-\tilde{K}) \left(\frac{1}{r}\de_n\psi_i^1\big|_{\de B_i}\right)}_{\dot{H}^1(\bigcup_l\de B_l)}\,\lesssim\, \eps^\frac{3}{2}|\log\eps|^\ell
\end{align}
\end{proof}

\begin{proposition}\phantomsection\label{1st order exp}
For $i\neq j$ we have that \begin{align}
\norm{\frac{1}{r}\de_n\psi_i-\frac{2}{r}n\cdot\nabla_yK(q_i,q_j)}_{L^2(\de B_j)}\,\lesssim\,\eps^\frac{3}{2}|\log\eps|^\ell,
\end{align}

where $''\nabla_y''$ refers to the gradient in the second variable and the implicit constant is bounded locally uniformly in $\tilde{q}$.

In particular, it follows from the Asymptotics \eqref{exp K} and \eqref{exp F} that $\norm{\de_n\psi_i}_{L^2(\de B_j)}\,\lesssim\, \eps^\frac{1}{2}|q_i-q_j|^{-1}$ for $i\neq j$, locally uniformly in $\tilde{q}$.
\end{proposition}
\begin{proof}
Let $\psi_i^1$ be the potential in case there is only a single body $B_i$.

Let $\mu\in L_0^2(\de B_j)$ be such that \begin{align}
K\mu+K\left(\frac{1}{r}\de_n\psi_i^1|_{\de B_i}\right)\eq const \text{ on $\de B_j$}.
\end{align} 

This is well-defined by the Lemmata \ref{K bd} and \ref{comp tilde K} and it holds that $\norm{\mu}_{L^2(\de B_j)}\lesssim \eps^\frac{1}{2}|q_i-q_j|^{-1}$.

Then for $m\neq j$ by Lemma \ref{comp tilde K}, it holds that \begin{align}
\norm{K\mu}_{\dot{H}^1(\de B_m)}\,\lesssim\, \eps^\frac{3}{2}|\log\eps|^\ell
\end{align}

and hence \begin{align}
\norm{K\mu+K\left(\frac{1}{r}\de_n\psi_i^1|_{\de B_i}\right)}_{\dot{H}^1(\de B_m)}\,\lesssim\, \eps^\frac{1}{2}|q_i-q_j|^{-1}.
\end{align}

Hence by \eqref{double bd} we conclude that \begin{align}\label{comp mu}
\norm{\mu-\frac{1}{r}\de_n\psi_i}_{L_0^2(\de B_j)}\,\lesssim\, \eps^\frac{3}{2}|\log\eps|^\ell,
\end{align}

and hence it suffices to compute $\mu$.

We first estimate $\nabla\psi_i^1$. We know from the definition of $\psi_i^1$ and the maximum principle that $\frac{1}{r}\de_n\psi_i^1\geq 0$ on $\de B_i$. Hence we can use the mean value theorem to estimate \begin{align}\label{est psi 1}
\norm{\nabla K\left(\frac{1}{r}\de_n\psi_i^1|_{\de B_i}\right)-\nabla_y K(q_i,\cdot)}_{C^0(\de B_j)}\,\lesssim\,\eps \sup_{z\in B_i,y\in \de B_j}\nabla_{z,y}^2 K(z,y)\,\lesssim\, \eps|\log\eps|^\ell
\end{align}

where we used the Asymptotics \eqref{exp K} and \eqref{exp F}. Similarly we have \begin{align}
\norm{\nabla K(q_i,q_j)-\nabla_yK\left(q_i,\cdot\right)}_{C^0(\de B_j)}\,\lesssim\, \eps \sup_{y\in B_j}|\nabla_{y}^2K(q_i,y)|\,\lesssim\,\eps|\log\eps|^\ell.
\end{align}

Hence we have that \begin{align}\label{inv psi 1}
\norm{\psi_i^1-x\cdot\nabla_y K(q_i,q_j)}_{\dot{H}^1(\de B_j)}\,\lesssim\, \eps^{\frac{3}{2}}|\log\eps|^\ell.
\end{align}

We have \begin{align}\label{inv x}
\overline{K}_{R_j}^{(-1)}x\eq-\frac{2}{R_j}n+const\end{align} by \eqref{action K flat} (where $n$ denotes the normal as usual).

By Lemma \ref{K bd} and the definition of $\mu$ we have
\begin{align}\label{mu est}
\norm{\mu+\overline{K}_{R_j}^{(-1)}\left(\psi_i^1\right)}_{L_0^2(\de B_j)}\eq \norm{\left(\overline{K}_{R_j}^{(-1)}-K^{-1}\right)\left(\psi_i^1\right)}\,\lesssim\,\eps^\frac{3}{2}|\log\eps|^\ell.
\end{align}

Together \eqref{inv psi 1},\eqref{inv x} and \eqref{mu est} imply that \begin{align}
\norm{\mu-\frac{2}{R_j}n\cdot\nabla_y K(q_i,q_j)}_{L_0^2(\de B_j)}\,\lesssim\, \eps^\frac{3}{2}|\log\eps|^\ell.
\end{align}

Finally we can replace the $R_j$ in the denominator by $r$ as $\norm{n\cdot \nabla_y K(q_i,q_j)}_{L^2(\de B_j)}\lesssim \eps^\frac{1}{2}|q_i-q_j|^{-1}$ by \eqref{exp K} and \eqref{exp F}.

\end{proof}

\subsubsection{The derivative with respect to $q$}
To compute derivatives with respect to $q$, we only need to consider partial derivatives with respect to a single $R_i$ or $Z_i$, as everything is smooth by Lemma \ref{PsiDeri}. 
In the following we write $\psi_{i,R_j+s}$ instead of $\psi_i$ to emphasize with respect to which $R_j$ the function $\psi_i$ is defined and analogously write $\psi_{i,Z_j+s}$. For mixed derivatives we write $\psi_{i,R_j+s_1,R_m+s_2}$ if we want to indicate multiple positions.

We set \begin{align}
\delta_{j}^s(x)\,:=\,\begin{cases}
    \frac{\rho_j(R_j+s)}{\rho_j(R_j)}& \text{if } x\in \de B_j\\
    1              & \text{else,}
\end{cases}
\end{align}

where we again write $\rho_j(\cdot)$ for $\rho_j$ as a function of $R_j$.

\begin{proposition}\phantomsection\label{deri mu}
For all $i,j,l$ it holds that \begin{align}
\norm{\de_s\left(\delta_j^s\left(\frac{1}{r}\de_n\psi_{i,R_j+s}\right)\circ c\right)}_{L^2(\de B_l)}\,\lesssim\, \eps^\frac{1}{2}\left(\min_{a\neq b}|q_a-q_b|\right)^{-2},
\end{align}

where the diffeomorphism $c$ was introduced in the beginning of Subsection \ref{sec q deri} and corresponds to the change of $R_j$.

Similarly it holds that \begin{align}
\norm{\de_s\left(\left(\frac{1}{r}\de_n\psi_{i,Z_j+s}\right)\circ d\right)}_{L^2(\de B_l)}\,\lesssim\, \eps^\frac{1}{2}\left(\min_{a\neq b}|q_a-q_b|\right)^{-2},
\end{align}

where the diffeomorphism $d$ was introduced in the beginning of Subsection \ref{sec q deri} and corresponds to the change of $Z_j$.

The implicit constant in both estimates is locally uniform in $\tilde{q}$.
\end{proposition}

\begin{proposition}\phantomsection\label{2nd deri mu}
\begin{itemize}
\item[a)] For all $i,j,l$ we have \begin{align}
\norm{\de_s^2\left(\delta_j^s\left(\frac{1}{r}\de_n\psi_{i,R_j+s}\right)\circ c\right)}_{L^2(\de B_l)}\,\lesssim\, \eps^\frac{1}{2}|\log\eps|^\ell
\end{align}

and the same holds for the second derivative with respect to $Z_j$ and the mixed second derivative. The implicit constant is bounded locally uniformly in $\tilde{q}$.
\item[b)] For all $i,j,l,m$ with $j\neq m$ we have \begin{align}
\norm{\de_{s_1}\de_{s_2}\left(\delta_j^{s_1}\delta_m^{s_2}\left(\frac{1}{r}\de_n\psi_{i,R_j+s_1,R_m+s_2}\right)\circ c_{jm}(s_1,s_2)\right)}_{L^2(\de B_l)}\,\lesssim\, \eps^\frac{1}{2}|\log\eps|^\ell
\end{align}

where $c_{jm}$ is the composition of the map $c$ defined for $j$ and $m$ with arguments $s_1$ and $s_2$ respectively. The same estimate holds for the derivatives with respect to $Z$'s or mixed derivatives. The implicit constant is bounded locally uniformly in $\tilde{q}$.
\end{itemize}
\end{proposition}

We shall only prove the statements for the $R_j$-derivative and focus on the first derivative and occasionally comment on the slight changes needed for the $Z_j$-derivative, which is generally easier. The second derivatives can be handled with the same technique, but the involved calculations become a lot more tedious, so we omit them most of the calculations for them.

 We set \begin{align}
K_s^{R_j}(x,y)\eq\left(1+I_{\de B_j}(x)\left(\frac{R_j}{R_j+s}-1\right)\right)K(c(x,s),c(y,s)),
\end{align}

and similarly \begin{align}
K_s^{Z_j}\,:=\,K(d(x,s),d(y,s)).
\end{align}

We also write $K_s^{R_j}$ for the associated linear map. Note that for $f$ supported on $\de B_j$ it holds that \begin{align}
K_s^{R_j}f(y)\eq\frac{R_j\rho_j}{(R_j+s)\rho_j(R_j+s)}K(f\circ c^{-1})(c(y,s)),
\end{align}

where we again write $\rho_j(\cdot)$ to denote $\rho_j$ as a function of $R_j$ and where the additional prefactor comes from the change of the inner radius. For $f$ supported on any other $\de B_i$ it holds \begin{align}
K_s^{R_j}f(y)\eq K(f)(c(y,s)).
\end{align}

Similar for mixed second derivatives with respect to different indices, one would use the kernel \begin{align}
&K_{s_1,s_2}^{R_j,R_m}(x,y)\eq\left(1+I_{\de B_j}(x)\left(\frac{R_j}{R_j+s}-1\right)+I_{\de B_m}(x)\left(\frac{R_m}{R_m+s}-1\right)\right)\\
&\times K(c_{jm}(x,s_1,s_2),c_{jm}(y,s_1,s_2)).\nonumber
\end{align}

\begin{lemma}\phantomsection\label{deri K}
The linear operator $K_s^{R_j}$ is Fr\'echet differentiable in $s$ as a map from $L^2(\de B_j)$ to $\dot{H}^1(\de B_j)$ for all $i$ and we  have \begin{align}
\norm{\de_sK_s^{R_j}}_{L^2(\de B_j)\rightarrow \dot{H}^1(\de B_j)},\norm{\de_s^2K_s^{R_j}}_{L^2(\de B_j)\rightarrow \dot{H}^1(\de B_j)}\,\lesssim\, \eps|\log\eps|,
\end{align}

furthermore, the Fr\'echet derivative is given by integration against the pointwise derivative in $s$.
\end{lemma}

Note that the corresponding derivatives of $K_s^{Z_j}$ are trivially $0$ by the explicit form of $K$ in \eqref{exp K}. The statement for mixed second derivatives also trivially reduces to the derivative with respect to a single index, as only the change of $R_j$ matters.

\begin{proof}
It is easy to see that the kernel is pointwise smooth in $s$ for $x\neq y$ by using the Expansions \eqref{exp K} and \eqref{exp F} and the differentiability of $c$. We shall estimate the operator norm of the first pointwise derivative. A similar, but tedious calculation, which we omit here can be made to show that the second and third pointwise derivatives are bounded, which by the mean value theorem justifies that the first two pointwise derivatives agree with the Fr\'echet derivatives.\\

Let us estimate the first derivative w.r.t.\ $s$ of the different parts of the kernel: \begingroup\allowdisplaybreaks\begin{align}
&\de_s \log |c(x,s)-c(y,s)|\eq\de_s\log\left(\frac{\rho_j(R_j+s)}{\rho_j}\right)\eq-\frac{1}{2(R_j+s)}\label{log c}\\
&\left|\de_s\frac{R_j}{R_j+s}\sqrt{c(x,s)_Rc(y,s)_R}\right|\,\lesssim\, \eps\label{deri prefactor 0}\\
&\norm{\de_s\frac{R_j}{R_j+s}\sqrt{c(x,s)_Rc(y,s)_R}}_{C_{x,y}^1(\de B_j\times \de B_j)}\,\lesssim\, 1\label{deri prefactor}\\
&\de_s\log(R_j+s)\eq\frac{1}{R_j+s}\label{deri log R 0}\\
&\left|\de_s\log \frac{(R_j+s)^2}{c(x,s)_Rc(y,s)_R}\right|\,\lesssim\, \eps\label{deri log R 1}\\
&\norm{\de_s\log \frac{(R_j+s)^2}{c(x,s)_Rc(y,s)_R}}_{C_{x,y}^1(\de B_j\times \de B_j)}\,\lesssim\,1\label{deri log R}
\end{align}\endgroup

Furthermore we have \begin{align}
\left|\de_s\frac{(c(x,s)-c(y,s))^2}{c(x,s)_Rc(y,s)_R}\big|_{s=0}\right| \eq \left|\de_s|x-y|^2\frac{\rho_j(R_j+s)^2}{\rho_j^2c(x,s)_Rc(y,s)_R}\big|_{s=0}\right|\,\lesssim\, |x-y|^2\label{s deri bd}
\end{align}

and similarly it holds  \begin{align}
&\left|\de_y\de_s\frac{(c(x,s)-c(y,s))^2}{c(x,s)_Rc(y,s)_R}\big|_{s=0}\right|\,\lesssim\, |x-y|.\label{s deri bd 2}
\end{align}

\begingroup
\allowdisplaybreaks

Now we may use the Expansions \eqref{exp K} and \eqref{exp F} to write \begin{align}
&2\pi\de_y\de_sK_s^{R_j}(x,y)\big|_{s=0}\eq\notag\\
& \de_s\de_y\left(\frac{R_j}{R_j+s}\sqrt{c(x,s)_Rc(y,s)_R}\right)\left(\log(|x-y|)-\log(8)+2-\frac{1}{2}\log(x_Ry_R)+g(x,y)\right)\notag\\
&+\de_s\left(\frac{R_j}{R_j+s}\sqrt{c(x,s)_Rc(y,s)_R}\right)\de_y\left(\log(|x-y|)-\frac{1}{2}\log\left(\frac{x_Ry_R}{R_j^2}\right)+g(x,y)\right)\notag\\
&+\de_y\left(\sqrt{x_Ry_R}\right)\de_s\left(\log(|c(x,s)-c(y,s)|)-\frac{1}{2}\log\left(\frac{c(x,s)_Rc(y,s)_R}{(R_j+s)^2}\right)+\log( R_j+s)\right)\notag\\
&+\sqrt{x_Ry_R}\de_y\de_s\left(\log(|c(x,s)-c(y,s)|)-\frac{1}{2}\log\frac{c(x,s)_Rc(y,s)_R}{(R_j+s)^2}\right)\notag\\
&-\de_y\left(\sqrt{x_Ry_R}\right)\de_sh\left(\frac{(c(x,s)-c(y,s))^2}{c(x,s)_Rc(y,s)_R}\right)\notag\\
&-\sqrt{x_Ry_R}\de_s\de_yh\left(\frac{(c(x,s)-c(y,s))^2}{c(x,s)_Rc(y,s)_R}\right).\notag\\
&=\,I+II+III+IV+V+VI.
\end{align}

\endgroup

Here $I-VI$ stand for the obvious terms and we dropped some constants whose derivative vanishes.

The boundedness of the terms $I$ and $II$ follows from the estimates \eqref{deri prefactor 0} and \eqref{deri prefactor} above and Schur's Lemma and further from using Lemma \ref{perturbed log} for the derivative of the logarithm.

The boundedness of $III$ follows directly  from \eqref{log c} and \eqref{deri log R 1} by using Schur's Lemma.

The boundedness of $IV$ follows from \eqref{log c} and \eqref{deri log R}.

The boundedness of $V$ and $VI$ follows from \eqref{s deri bd} resp.\ \eqref{s deri bd 2} and the estimate \eqref{est g} on $g$.\\

\end{proof}

\begin{lemma}\phantomsection\label{deri K 2}
The kernel $K_s^{R_j}$ is Fr\'echet-differentiable in $s$ as a map from $L^2(\de B_i)$ to $\dot{H}^1(\de B_l)$ for $i\neq l$ and, locally uniformly in $\tilde{q}$, we have that \begin{align}
\norm{\de_s K_s^{R_j}}_{L_0^2(\de B_i)\rightarrow \dot{H}^1(\de B_l)}\,\lesssim\,\eps|q_i-q_l|^{-2}\quad\text{and}\quad\norm{\de_s^2 K_s^{R_j}}_{L_0^2(\de B_i)\rightarrow \dot{H}^1(\de B_l)}\,\lesssim\, \eps|q_i-q_l|^{-3}.
\end{align}

 Furthermore, the Fr\'echet derivative is given by integration against the pointwise derivative in $s$. The same estimates also hold for the Kernel $K_s^{Z_j}$ and the second derivatives of $K_{s_1,s_2}^{R_j,R_m}$.
\end{lemma}
\begin{proof}
We only consider the case $l=j$ and the first derivative, the other cases and the second derivative are very similar. In this case, the kernel is smooth in $(x,y,s)$ and hence the Fr\'echet- and pointwise derivative agree by e.g.\ the mean value theorem.

By \eqref{exp K} and \eqref{exp F}, we have that \begin{align}
&-2\pi\de_s\de_yK_s^{R_j}(x,y)\eq \de_s\de_y\sqrt{x_Rc(y,s)_R}F\left(\frac{|x-y|^2}{x_Ry_R}\right)+\de_s\sqrt{x_Rc(y,s)_R}F'\left(\frac{|x-y|^2}{x_Ry_R}\right)\de_y\frac{|x-y|^2}{x_Ry_R}\notag\\
&+\de_y\sqrt{x_Ry_R}F'\left(\frac{|x-c(y,s)|^2}{x_Rc(y,s)_R}\right)\de_s\frac{|x-c(y,s)|^2}{x_Rc(y,s)_R}+\\
&\sqrt{x_Ry_R}\biggl(F'\left(\frac{|x-c(y,s)|^2}{x_Rc(y,s)_R}\right)\de_y\de_s\frac{|x-c(y,s)|^2}{x_Rc(y,s)_R}\notag\\
&+F''\left(\frac{|x-c(y,s)|^2}{x_Rc(y,s)_R}\right)\de_s\frac{|x-c(y,s)|^2}{x_Rc(y,s)_R}\de_y\frac{|x-y|^2}{x_Ry_R}\biggr).\notag
\end{align}

It is easy to see that all relevant derivatives of the prefactor are $\lesssim 1$ and that $|\de_s c|\lesssim 1$. Hence the absolute value of this is \begin{align}
\lesssim\, \left|F\left(\frac{|x-y|^2}{x_Ry_R}\right)+F'\left(\frac{|x-y|^2}{x_Ry_R}\right)(1+|x-y|)+F''\left(\frac{|x-y|^2}{x_Ry_R}\right)|x-y|^2\right|.
\end{align}

From the asymptotic of $F$ (see \eqref{exp F}), we see that for $|x-y|$ small this is \begin{align}
\lesssim\, \frac{1}{|x-y|^2}.
\end{align}

Hence we see that \begin{align}
\norm{\de_sK_s^{R_j}}_{L^2(\de B_i)\rightarrow \dot{H}^1(\de B_j)}\,\lesssim\,\eps |q_i-q_l|^{-2}.
\end{align}

\end{proof}

\begin{proof}[Proof of Propsition \ref{deri mu}]
We know from Lemma \ref{PsiDeri} that $\frac{1}{r}\psi_{j,R_j+s}$ is differentiable, hence we may differentiate the equation \begin{align}
K_s^{R_j}\delta_j^s\frac{1}{r}\de_n\psi_{i,R_j+s}\circ c\eq const \text{ on each $\de B_m$}
\end{align}

with respect to $s$, as $K_s^{R_j}$ is differentiable by Lemma \ref{deri K}. This yields that \begin{align}
\de_sK_s^{R_j}\left(\frac{1}{r}\de_n\psi_{i}\right)+K\left(\de_s\left(\delta_j^s\left(\frac{1}{r}\de_n\psi_{i,R_j+s}\right)\circ c\right)\right)\eq const \text{ on each $\de B_m$}.
\end{align}

Note that $\de_s\delta_j^s\frac{1}{r}\de_n\psi_{i,R_j+s}\circ c$ is mean-free on each $\de B_m$, because the integral of $\delta_j^s\frac{1}{r}\de_n\psi_{i,R_j+s}\circ c$ over $\de B_m$ is either $1$ or $0$ for all $s$ by definition. Furthermore by Lemma \ref{deri K} and Lemma \ref{deri K 2}, we have \begin{align}
\norm{\de_sK_s^{R_j}\left(\frac{1}{r}\de_n\psi_{i}\right)}_{\dot{H}^1(\bigcup_m \de B_m)}\,\lesssim\,\eps^\frac{1}{2}\left(|\log\eps|+\min_{a\neq b}|q_a-q_b|^{-2}\right).
\end{align}

We can absorb the logarithm into the second summand by definition of the regimes.

By Corollary \ref{K inv mult}, we conclude.

\end{proof}

\begin{proof}[Proof of Proposition \ref{2nd deri mu}]
We only consider the second derivative with respect to $R_j$, all others work the same because one has the same or better estimates.

We have \begin{align}
\de_s^2K_s^{R_j}\left(\frac{1}{r}\de_n\psi_i\right)+2\de_sK_s^{R_j}\left(\de_s\delta_j^s\left(\frac{1}{r}\de_n\psi_{i,R_j+s}\right)\circ c\right)+K\left(\de_s^2\delta_j^s\left(\frac{1}{r}\de_n\psi_{i,R_j+s}\right)\circ c\right)\eq const 
\end{align}

on all $\de B_m$.

Note that by the Lemmata \ref{deri K}, \ref{deri K 2} and Proposition \ref{deri mu} it holds that \begin{align}
\norm{\de_s^2K_s^{R_j}\frac{1}{r}\de_n\psi_i}_{\dot{H}^1(\bigcup_m\de B_m)}+\norm{\de_sK_s^{R_j}\left(\de_s\delta_j^s\left(\frac{1}{r}\de_n\psi_{i,R_j+s}\right)\circ c\right)}_{\dot{H}^1(\bigcup_m\de B_m)}\,\lesssim \,\eps^\frac{1}{2}|\log\eps|^\ell.
\end{align}

Hence from Corollary \ref{K inv mult} we conclude the statement.\end{proof}

\subsubsection{Further estimates on the derivative of $K$}
To compute the force $G$, we will also need estimates in the $L^2\rightarrow L^2$-topology.

\begin{lemma}\phantomsection\label{deri K L2}
For all $j$ it holds that \begin{align}
\norm{\de_s K_s^{R_j}}_{L^2(\de B_j)\rightarrow L^2(\de B_j)}\,\lesssim\,\eps
\end{align} and \begin{align}
\norm{\de_s^2 K_s^{R_j}}_{L^2(\de B_j)\rightarrow L^2(\de B_j)}\,\lesssim\, \eps,
\end{align}

locally uniformly in $\tilde{q}$ and furthermore, the pointwise and Fr\'echet derivatives agree. 
\end{lemma}

Here we only need the estimate in the $R_j$-direction, because the derivative in the $Z_j$-direction is $0$. Also note that changes of the other $R_m$'s trivially give a derivative of zero here.

\begin{proof}
We only show this for the first derivative with respect to $R_j$, the second derivative is very similar.
 We may expand the kernel as \begin{align}
&2\pi\left(\de_sK_s^{R_j}(x,y)\big|_{s=0}\right)\\
&\eq \de_s\left(\frac{R_j}{R_j+s}\sqrt{c(x,s)_Rc(y,s)_R}\right)\left(\log(|x-y|)-\log(8)+2-\frac{1}{2}\log\left(x_Ry_R\right)+g(x,y)\right)\notag\\
&+\sqrt{x_Ry_R}\de_s\left(\log\left(|c(x,s)-c(y,s)|\right)-\frac{1}{2}\log(c(x,s)_Rc(y,s)_R)-h\left(\frac{|c(x,s)-c(y,s)|^2}{c(x,s)_Rc(y,s)_R}\right)\right).\notag
\end{align}

It is easy to see that \begin{align}
\left|\de_s\left(\frac{R_j}{R_j+s}\sqrt{c(x,s)_Rc(y,s)_R}\right)\right|\,\lesssim\, \eps
\end{align}

and by reusing \eqref{log c} and \eqref{deri log R 1} and \eqref{deri log R 0} one sees that \begin{align}\label{bd sqrt}
\left|\de_s\left(\log(|c(x,s)-c(y,s)|)-\frac{1}{2}\log(c(x,s)_Rc(y,s)_R)-h\left(\frac{|c(x,s)-c(y,s)|^2}{c(x,s)_Rc(y,s)_R}\right)\right)\right|\,\lesssim \,1.
\end{align}

The boundedness follows by Schur's Lemma. The Fr\'echet differentiability follows from the boundedness of the second derivative and the mean value theorem.

\end{proof}

\subsection{The force $G$}
We compute the asymptotic of $G$, defined in \ref{main def}. The crucial Lemma for the ``self-interaction'' terms is the following:

\begin{lemma}\phantomsection\label{Lem:EnergyDerivative}
We have \begin{align}
\int_{\mathcal{F}} \frac{1}{2}r\scalar{\frac{1}{r}\nabla^\perp\psi_i}{\frac{1}{r}\nabla^\perp\psi_j}\dx\eq-\frac{1}{2}C_{ij}
\end{align}  and for any $t_l^*$ associated to $B_l$ it holds that \begin{align}
\de_{q}\frac{1}{2}C_{ij}\cdot t_l^*\eq\frac{1}{2}\int_{\de B_l}\frac{1}{r}\de_n\phi_{l,t_l^*}\scalar{\nabla^\perp\psi_i}{\nabla^\perp\psi_j}\dx,\end{align}
where the $C$'s were defined in \ref{def psi}.
\end{lemma} 
\begin{proof}
 We have that \begin{align}
\frac{1}{2}\int_\mathcal{F} r\frac{1}{r^2}\scalar{\nabla^\perp\psi_i}{\nabla^\perp\psi_j}\dx\eq-\sum_l\frac{1}{2}\int_{\de B_l}\frac{1}{r}C_{il}\de_n\psi_j\dx\eq-\frac{1}{2}C_{ij}.
\end{align}

Here the partial integration is justified by the Lemmata \ref{ExistencePsi} and \ref{rep psi}.

Note that it holds \begin{align}\label{eq 6}
\de_q\psi_j\cdot t_l^*\eq\de_q C_{jm}\cdot t_l^*-\de_n\psi_j u(t_l^*) 
\end{align}

on $\de B_m$ as one can see by differentiating $C_{jm}=\psi_j(x_l)$ by $q$ for some point $x_l$ on $\de B_m$ moving with normal velocity $u(t_l^*)$.

Now by Reynolds (which can be used by the integrability statement in Lemma \ref{PsiDeri}) the derivative of this with respect to $q$ in direction $t_l^*\in T_{q_l}M$ equals \begin{align}
&\frac{1}{2}\de_q C_{ij}\cdot t_l^*\eq -\int_\mathcal{F} \frac{1}{r}\scalar{\nabla^\perp\psi_i}{\nabla^\perp\de_q\psi_j\cdot t_l^*}\dx+\frac{1}{2}\int_{\de B_l}\frac{1}{r}\scalar{\nabla^\perp\psi_i}{\nabla^\perp\psi_j}u(t_l^*)\dx\\
&\eq\sum_{m}\int_{\de B_m}\frac{1}{r}\de_q\psi_j\cdot t_l^*\de_n\psi_i\dx+\frac{1}{2}\int_{\de B_l}\frac{1}{r}\scalar{\nabla^\perp\psi_i}{\nabla^\perp\psi_j}u(t_l^*)\dx\\
&\eq\sum_m\int_{\de B_m}\frac{1}{r}(\de_q C_{jm}\cdot t_l^*-\de_n\psi_j u(t_l^*))\de_n\psi_i\dx+\frac{1}{2}\int_{\de B_l}\frac{1}{r}\scalar{\nabla^\perp\psi_i}{\nabla^\perp\psi_j}u(t_l^*)\dx\\
&\eq\de_qC_{ji}\cdot t_l^*-\int_{\de B_l}\frac{1}{r}\scalar{\nabla^\perp\psi_i}{\nabla^\perp\psi_j} u(t_l^*)\dx+\frac{1}{2}\int_{\de B_l}\frac{1}{r}\scalar{\nabla^\perp\psi_i}{\nabla^\perp\psi_j}u(t_l^*)\dx,
\end{align}

where we have made use of equation \eqref{eq 6} in the third line and of the facts that $\de_q C_{jm}$ is a constant function and that the matrix $C$ is symmetric by the first statement.
\end{proof}

\begin{proposition}\phantomsection\label{diagonal A}
For every tangent vector $t^*$, we have that \begin{align}
\left|G(q,e_i)\cdot t^*-\frac{1}{4\pi}\log(\eps\tilde{\rho})(t_i^*)\cdot e_R\right|\,\lesssim\, 1
\end{align}

and \begin{align}
\left|\de_qG(q,e_i)\right|\,\lesssim\, 1.
\end{align}

These estimate are locally uniform in $\tilde{q}$.
\end{proposition}
\begin{proof}
By Lemma \ref{Lem:EnergyDerivative} and the definition of $G(q,e_i)$ in \ref{main def}, it equals $\frac{1}{2}$ times the derivative of the energy \begin{align}
\int_{(\bigcup_l \de B_l)^2}K(x,y)\frac{1}{r^2}\de_n\psi_i(x)\de_n\psi_i(y)\dx\dy
\end{align}

with respect to $q$. We first consider the partial derivative in the direction $R_i$.

Note that $K\frac{1}{r}\de_n\psi_i$ is constant and that $\frac{1}{r}\de_n\psi_i$ is mean-free on all boundaries except $\de B_i$, hence the integral over all boundaries except $(\de B_i)^2$ is zero.

Using the diffeomorphism $c$, we can rewrite the energy with respect to $R_i+s$ as \begin{align}
\int_{(\de B_i)^2}\frac{R_i+s}{R_i}K_s^{R_i}(\delta_i^s)^2\left(\frac{1}{r}\de_n\psi_{i,R_i+s}\right)\circ c\left(\frac{1}{r}\de_n\psi_{i,R_i+s}\right)\circ c\dx\dy,
\end{align}

here the factor $\delta_i^s$ is the determinant due to the change of coordinates.

We can differentiate this under the integral as everything is smooth by Lemma \ref{PsiDeri}. We first show that the parts where a derivative falls on $\frac{1}{r}\de_n\psi_i$ are small. Indeed by Lemma \ref{K bd} and Proposition \ref{deri mu} we have \begin{align}
\norm{K\de_s\left(\delta_i^s\left(\frac{1}{r}\de_n\psi_{i,R_i+s}\right)\circ c\right)}_{L^2(\de B_i)}\,\lesssim\, \eps|\log\eps|^\ell \norm{\de_s\left(\delta_i^s\left(\frac{1}{r}\de_n\psi_{i,R_i+s}\right)\circ c\right)}_{L^2(\de B_i)}\,\lesssim\, \eps^{\frac{3}{2}}|\log\eps|^\ell.
\end{align}

Hence  \begin{align}
&\left|\int_{(\de B_i)^2}K(x,y)\de_s\left(\delta_i^s\left(\frac{1}{r}\de_n\psi_{i,R_i+s}\right)\circ c\right)\frac{1}{r}\de_n\psi_i(y)\dx\dy\right|\,\lesssim\, \eps^{\frac{3}{2}}|\log\eps|^\ell\norm{\frac{1}{r}\de_n\psi_i}_{L^2(\de B_i)}\\
&\lesssim\, \eps|\log\eps|^\ell.\notag
\end{align}

The same argument can also be made for all terms involving one or two derivatives of $\frac{1}{r}\de_n\psi_i$ or $K_s^{R_j}$ by the estimates in the Lemmata \ref{K bd}, \ref{deri K L2} and Propositions \ref{deri mu} and \ref{2nd deri mu} and also for derivatives with respect to other $R_j$'s or $Z_j$'s. As the second derivative of $\frac{R_i+s}{R_i}$ vanishes, this shows the estimate for the derivative, in the direction $R_j$.

Hence we are left with the main contribution where the derivative falls on $\frac{R_i+s}{R_i}$, which is \begin{align}
\int_{(\de B_i)^2}\frac{1}{R_i}K(x,y)\frac{1}{r^2}\de_n\psi_i(x)\de_n\psi_i(y)\dx\dy.
\end{align}

It can be rewritten as \begin{align}
&\strokedint_{(\de B_i)}\frac{1}{R_i}\overline{K}_{R_i}(x,y)\dx\dy+\\
&O\left(\eps^{-1}\norm{K-\overline{K}_{R_i}}_{L^2\rightarrow L^2}+\norm{K}_{L^2\rightarrow L^2}\norm{\frac{1}{r}\de_n\psi_i-\frac{1}{2\pi\eps\tilde{\rho_i}}}_{L^2}\norm{\frac{1}{r}\de_n\psi_i}_{L^2}\right)\notag.
\end{align}

The $O$-term is $\lesssim \eps|\log\eps|^\ell$ by Lemma \ref{K bd} and Proposition \ref{comp mult}. We computed in the Claim \eqref{action K flat} that the main integral equals  \begin{align}
\frac{1}{2\pi}\left(\log\left(\eps\tilde{\rho}_i\right)-\log(8)-2-\log R_i\right).
\end{align}

\end{proof}

Next, we consider the ``cross-terms'' in $G$, given by the interaction between $\psi_i$ and $\psi_j$.

\begin{proposition}\phantomsection\label{crossterms G}
Let $t^*=(t_1^*,\dots t_k^*)$ be a tangent vector, identified with a vector in $\R^{2k}$ as usual, then for $i\neq j$ we have that \begin{align}
\left|\int_{\bigcup_m\de B_m}\frac{1}{r}u(t^*)\de_n\psi_i\de_n\psi_j\dx-t_j^*\cdot\nabla_yK(q_i,q_j)-t_i^*\cdot\nabla_yK(q_j,q_i)\right|\,\lesssim\, \eps|\log\eps|^\ell,
\end{align}

where the normal velocity $u(t^*)$ was defined in \eqref{normal velo}. Furthermore it holds that \begin{align}
\left|\de_q G(q,\gamma)\right|\,\lesssim\, \max_{i,j}|q_i-q_j|^{-2}|\gamma|^2.
\end{align}

Both of these estimates are locally uniform in $\tilde{q}$.

\end{proposition}
\begin{proof}
We first prove the first statement for the contribution of $\de B_j$, which also covers the contribution of $\de B_i$ by symmetry.
We have\begin{align}\label{1st step}
&\int_{\de B_j}\frac{1}{r}\left(t_j^*\cdot n-\frac{\rho_j}{2R_j}t_j^*\cdot e_R\right)\de_n\psi_j\de_n\psi_i\dx\notag\\
&\eq \int_{\de B_j} \frac{1}{r}t_j^*\cdot n\de_n \psi_i\de_n\psi_j\dx+O\left(\eps\norm{\de_n\psi_i}_{L^2(\de B_j)}\norm{\de_n\psi_j}_{L^2(\de B_j)}\right).
\end{align}

By Propositions \ref{comp mult} and \ref{1st order exp} the error term here is $\lesssim \eps|\log\eps|^\ell$.

We can now use Proposition \ref{1st order exp} and \eqref{1st step} to obtain that \begin{align}
\int_{\de B_j}\frac{1}{r}u(t^*)\de_n\psi_i\de_n\psi_j\dx\eq 2\int_{\de B_j}(t^*\cdot n)n\cdot\nabla_yK(q_i,q_j)\frac{1}{r}\de_n\psi_j\dx+O\left(\eps|\log\eps|^\ell\right).
\end{align}

We further have \begin{align}
&2\int_{\de B_j}(t^*\cdot n)n\cdot\nabla_yK(q_i,q_j)\frac{1}{r}\de_n\psi_j\dx\eq2\strokedint_{\de B_j} (t^*\cdot n)n\cdot\nabla_yK(q_i,q_j)\dx\\
&+O\left(\eps^\frac{1}{2}|\nabla_yK(q_i,q_j)|\norm{\frac{1}{r}\de_n\psi_j-\frac{1}{2\pi\eps\tilde{\rho}_j}}_{L^2(\de B_j)}\right).\notag
\end{align}

By Proposition \ref{comp mult}, the error term is $\lesssim \eps|\log\eps|^\ell$. The main integral equals $t^*\cdot\nabla_yK(q_i,q_j)$.\\

For $m$ with $m\neq i,j$, one can directly see by Proposition \ref{1st order exp} that the integral is $\lesssim \eps|\log\eps|^\ell$.

It remains to estimate the derivative. Note that $G$ is a quadratic form in $\gamma$ and that we have already shown the statement for $\gamma=e_l$ in Proposition \ref{diagonal A}, and that the ``off-diagonal'' coefficients in $G$ are exactly the integrals we estimated in the first step, so we need to estimate their derivatives.

For notational simplicity we only consider the derivative of the integral on $\de B_j$ with respect to $R_l$, as the derivative with respect to $Z_l$ enjoys the same estimates, this is not restrictive. We begin with the derivative with respect to $R_l$ for $l\neq j$.

By Propositions \ref{deri mu} it holds that \begin{align}
\norm{\de_s\frac{1}{r}\de_n\psi_{i,R_l+s}}_{L^2(\de B_j)}\,\lesssim\, \eps^\frac{1}{2}|q_i-q_j|^{-2}\quad\text{and}\quad \norm{\de_s\frac{1}{r}\de_n\psi_{j,R_i+s}}_{L^2(\de B_j)}\,\lesssim\, \eps^\frac{1}{2}|\log\eps|^\ell,
\end{align}

which by the Cauchy-Schwarz inequality and Propositions \ref{comp mult} and \ref{1st order exp} implies the statement.\\

Finally, consider the derivative with respect to $R_j$.

We can rewrite the integral as \begin{align}
\int_{\de B_j(R_j)}\left(t^*\cdot n-\frac{\rho_j(R_j+s)}{2(R_j+s)}t^*\cdot e_R\right)\left(\delta_j^s\frac{1}{r}\de_n\psi_{j,R_j+s}\right)\circ c\left(\de_n\psi_{i,R_j+s}\right)\circ c\dx.
\end{align}

Here the factor $\delta_j^s$ is the Jacobian due to the coordinate change.

Using Proposition \ref{deri mu}, we see that \begin{align}
\norm{\de_s\left(\delta_j^s\left(\frac{1}{r}\de_n\psi_{j,R_j+s}\right)\circ c\right)}_{L^2(\de B_j)}\,\lesssim\, \eps^{\frac{1}{2}}|\log\eps|^\ell.
\end{align}

Furthermore by Proposition \ref{deri mu}  we have \begin{align}
\norm{\de_s\left(\left(\frac{1}{r}\de_n\psi_{i,R_j+s}\right)\circ c\right)}_{L^2(\de B_j)}\,\lesssim\, \eps^\frac{1}{2}|q_i-q_j|^{-2}
\end{align}

Hence by the Cauchy-Schwarz inequality we conclude.

\end{proof}

\subsection{The mixed term $A$}\label{subsec4.2}
We estimate the force $A$ (defined in \ref{main def}), which contains both the stream function and the potentials.

\begin{proposition}\phantomsection\label{conv A}
For all $s^*,t^*\in T_qM$ we have that \begin{align}(A(q,\gamma)t^*)\cdot s^*\rightarrow (t^*)^T\begin{pmatrix} 0 & R_1\gamma_1  &0 &\dots\\
-R_1\gamma_1 & 0 & \dots\\
\dots\\
& & \dots & 0 & R_k\gamma_k\\
& &\dots & -R_k\gamma_k & 0\end{pmatrix}s^*
\end{align}

with a rate of $O(\eps|\log\eps|^\ell)$ locally uniformly in $\tilde{q}$.

 Furthermore, it holds that \begin{align}
\left|\de_q A\right|\,\lesssim\,\eps|\log\eps|^\ell,
\end{align}

locally uniformly in $\tilde{q}$.
\end{proposition}
 In particular, $A$ is invertible for small enough $\eps$ with an inverse of order $\lesssim 1$ by the assumption that all $\gamma_i$'s are $\neq 0$.
\begin{proof} Recall that $A$ was defined as \begin{align}
(A(q,\gamma)t^*)\cdot s^*\eq\sum_l\int_{\de B_l}\Bigl(-\de_\tau\phi(s^*)\de_n\phi(t^*)+\de_\tau\phi(t^*)\de_n\phi(s^*)\Bigr)\de_n\sum_j\gamma_j\psi_j\dx,
\end{align}

where $\phi(s^*)$ and $\phi(t^*)$ are the summed up potentials.

Without loss of generality, we may assume $|t^*|=|s^*|=1$, $t^*\in T_{q_i}M$, $s^*\in T_{q_j}M$, and that only $\gamma_l$ is nonzero.
We first show the convergence for $i\neq j$. By definition we have \begin{align}
\norm{\de_n\phi_{i,t^*}}_{L^\infty(\de B_i)}\,\lesssim\, 1\quad\text{and}\quad \norm{\de_n\phi_{j,s^*}}_{L^\infty(\de B_j)}\,\lesssim\, 1
\end{align}

and on all other boundaries the normal derivatives vanish. By Corollary \ref{MultPotSpec}, we have \begin{align}
\norm{\de_\tau\phi_{i,t^*}}_{L^2(\de B_j)}\,\lesssim\, \eps^\frac{5}{2}|\log\eps|^\ell
\end{align}

and vice versa. Furthermore we have \begin{align}
\norm{\de_n\psi_l}_{L^2(\de B_i)},\norm{\de_n\psi_l}_{L^2(\de B_j)}\,\lesssim\, \eps^{-\frac{1}{2}}
\end{align}

for all $l$ by the Propositions \ref{comp mult} and \ref{1st order exp}. By the Cauchy-Schwarz inequality we conclude convergence to $0$. Similarly, we can directly estimate the derivative. By the Propositions \ref{prop deri} and \ref{deri mu}, we know that all derivatives of the boundary values enjoy estimates which are at worst an order $|\log\eps|^\ell$ worse, hence these derivatives are small by the Cauchy-Schwarz inequality and the product rule. 

Next, consider the case $i=j\neq l$. Here we again have \begin{align}
\norm{\de_\tau\phi_{i,t^*}}_{L^2(\de B_i)}\,\lesssim\, \eps^\frac{1}{2}
\end{align}

by Corollary \ref{MultPotSpec} and the same holds for the potential with respect to $s^*$. On the other hand we also have \begin{align}
\norm{\de_n\psi_l}_{L^2(\de B_i)}\,\lesssim\,\eps^\frac{1}{2}|\log\eps|^\ell
\end{align}

by Proposition \ref{1st order exp}.

By the Cauchy-Schwarz inequality, this implies that \begin{align}
\left|\int_{\de B_i}\de_n\psi_l\left(-\de_\tau\phi_{i,s^*}\de_n\phi_{i,t^*}+\de_\tau\phi_{i,t^*}\de_n\phi_{i,s^*}\right)\dx\right|\,\lesssim\,\eps|\log\eps|^\ell.
\end{align}

The smallness of the derivative of this term again follows from the fact that all derivatives have estimates which are at worst an order $|\log\eps|^\ell$ worse by the Propositions \ref{prop deri} and \ref{deri mu}.\\

It remains to consider the case $i=j=l$. In this case we have the same estimates as above for the tangential derivatives and furthermore by Corollary \ref{MultPotSpec} we have \begin{align}
\norm{\de_\tau(\phi_{i,t^*}-\chek{\phi}_{t^*})}_{L^2(\de B_i)}\,\lesssim\, \eps^\frac{3}{2}|\log\eps|^\ell\quad\text{and}\quad  \norm{\de_n(\phi_{i,t^*}-\chek{\phi}_{t^*})}_{L^2(\de B_i)}\,\lesssim\, \eps^\frac{3}{2}
\end{align}

 where the ``two-dimensional'' potential $\chek{\phi}$ was defined in \ref{def flat func} and the same holds for the potentials with respect to $s^*$.

Also by Proposition \ref{comp mult}, we have \begin{align}
\norm{\frac{1}{r}\de_n\psi_i-\frac{1}{2\pi\tilde{\rho}_i\eps}}_{L^2(\de B_i)}\,\lesssim\,\eps^\frac{1}{2}|\log\eps|^\ell.
\end{align}

Finally, we clearly have $\norm{r-R_i}_{L^2(\de B_i)}\lesssim \eps^\frac{3}{2}$. Hence by the Cauchy-Schwarz inequality we see that \begin{align}
&\int_{\de B_i}\de_n\psi_i\left(-\de_\tau\phi_{i,s^*}\de_n\phi_{i,t^*}+\de_\tau\phi_{i,t^*}\de_n\phi_{i,s^*}\right)\dx\notag\\ &\eq R_i\strokedint_{\de B_i}\left(-\de_\tau\chek{\phi}_{s^*}\de_n\chek{\phi}_{t^*}+\de_\tau\chek{\phi}_{t^*}\de_n\chek{\phi}_{s^*}\right)\dx+O\left(\eps|\log\eps|^\ell\right).
\end{align}

By the antisymmetry of these integrals with respect to $t^*$ and $s^*$ it suffices to consider the case $t^*=e_1$ and $s^*=e_2$. In this case, we can use the explicit form of $\chek{\phi}_{t^*}$ in \ref{def flat func} to see that \begin{align}
\strokedint_{\de B_i}\left(-\de_\tau\chek{\phi}_{s^*}\de_n\chek{\phi}_{t^*}+\de_\tau\chek{\phi}_{t^*}\de_n\chek{\phi}_{s^*}\right)\dx\eq \strokedint \tau\cdot e_2n\cdot e_1-\tau\cdot e_1n\cdot e_2\dx\eq 1.
\end{align}

The smallness of the derivative follows again from the fact that all the derivatives of the boundary values have estimates which are an order $\eps|\log\eps|^\ell$ better by Propositions \ref{prop deri} and \ref{deri mu}.

Finally, all these estimates are locally uniform in $\tilde{q}$ because all the used estimates for the boundary values are.\end{proof}

\begin{definition}
We let $J_\gamma^1$ and $J_\gamma^2$ be the velocities in \eqref{PointSystem2} and \eqref{PointSystem1}, i.e. \begin{align}
&(J_\gamma^1(\tilde{q}))_i\eq \frac{1}{2\pi}\sum_{j\neq i}\gamma_j\frac{(\tilde{q}_i-\tilde{q}_j)^\perp}{|\tilde{q}_i-\tilde{q}_j|^2}-\frac{\gamma_i}{4\pi R_0}e_Z\\
&(J_\gamma^2(\tilde{q}))_i\eq \frac{1}{2\pi}\sum_{j\neq i}\gamma_j\frac{(\tilde{q}_i-\tilde{q}_j)^\perp}{|\tilde{q}_i-\tilde{q}_j|^2}+\frac{\tilde{q}_{R_i}\gamma_i}{4\pi R_0^2}e_Z
\end{align}
\end{definition}

\begin{corollary}\phantomsection\label{conv J1}
In the Regime \eqref{def regime2} (=distances $\approx|\log\eps|$), we have that \begin{align}
\frac{A^{-1}G}{|\log\eps|}\rightarrow -J_\gamma^1(\tilde{q})
\end{align}

locally uniformly in $\tilde{q}$.
\end{corollary}
\begin{proof}
By the Propositions \ref{diagonal A} and \ref{conv A} have that \begin{align}
\frac{{A(q,\gamma)}^{-1}G(q,e_i)}{|\log\eps|}\rightarrow \frac{1}{4\pi R_0\gamma_i}e_{Z,i},
\end{align}

where $e_{Z,i}\in (\R^2)^k\simeq \R^{2k}$ denotes the vector which has an $e_Z$ in the $i$-th component and no other entries.

As $G$ is a quadratic form in $\gamma$ by definition, it remains to show the statement for the ``off-diagonal'' terms in $G$. By the Propositions \ref{crossterms G} and \ref{conv A} we have that \begin{align}
&\lim_{\eps\rightarrow 0}\frac{A(q,\gamma)^{-1}\left(G(q,e_i+e_j)-G(q,e_i)-G(q,e_j)\right)}{|\log\eps|}\\
&=\,- \lim_{\eps\rightarrow 0}\frac{1}{R_0|\log\eps|} \left(\frac{P_j}{\gamma_j}\nabla_y^\perp K(q_i,q_j)+\frac{P_i}{\gamma_i}\nabla_y^\perp K(q_j,q_i)\right)\\
&=\, \frac{1}{2\pi}\left(\frac{P_i(\tilde{q}_j-\tilde{q}_i)^\perp}{\gamma_i|\tilde{q}_i-\tilde{q}_j|^2}+\frac{P_j(\tilde{q}_i-\tilde{q}_j)^\perp}{\gamma_j|\tilde{q}_i-\tilde{q}_j|^2}\right).
\end{align}

Here $P_l:\R^2\rightarrow (\R^2)^k$ is the map to the $l$-th coordinate and in the last step we used the Asymptotics \eqref{exp K} and \eqref{exp F}.
\end{proof}

\begin{corollary}\phantomsection\label{conv J2}
In the Regime \eqref{def regime1} (=distances $\approx|\log\eps|^\frac{1}{2}$), we have that \begin{align}
\frac{A^{-1}G}{|\log\eps|^\frac{1}{2}}-\frac{|\log\eps|^\frac{1}{2}}{4\pi R_0}v_Z\rightarrow -J_\gamma^2(\tilde{q})
\end{align}

locally uniformly in $\tilde{q}$, where $v_Z\in (\R^2)^k$ is the vector $(e_Z,e_Z,\dots)$.
\end{corollary}
\begin{proof}
The calculation of the ``off-diagonal'' terms is the same as in the previous proof. For the diagonal terms, we have by the Propositions \ref{diagonal A} and \ref{conv A} that \begin{align}
\frac{A(q,\gamma)^{-1}G(q,e_i)}{|\log\eps|^\frac{1}{2}}-\frac{|\log\eps|^\frac{1}{2}}{4\gamma_i\pi R_i}\rightarrow 0.
\end{align}

The statement then follows from the definition of $\tilde{q}$ and the fact that e.g.\ by the mean value theorem we have \begin{align}
\frac{|\log\eps|^\frac{1}{2}}{4\pi R_i}-\frac{|\log\eps|^\frac{1}{2}}{4\pi R_0}+\frac{(R_i-R_0)|\log\eps|^\frac{1}{2}}{4\pi R_0^2}\rightarrow 0.
\end{align}

\end{proof}

\section{Passage to the limit}\label{Section5}

In this section, we write $\tilde{q}_\eps$ instead of $\tilde{q}$ to emphasize the $\eps$-dependence.

\begin{proof}[Proof of Thm.\ \ref{Limit1}]
We write the system \eqref{MainODE} in the rescaled time $s=t|\log\eps|^2$ and the rescaled position $\tilde{q}_\eps$, defined as in \eqref{def regime2}, it then reads as \begin{align}\label{rescaled ODE}
&|\log\eps|^3\left(E(q)\tilde{q}_{\eps}''+\frac{1}{2}\tilde{q}_{\eps}'(\nabla_{\tilde{q}_\eps}E(q)\cdot\tilde{q}_{\eps}')+\mathcal{M}(q)\tilde{q}_\eps''+|\log\eps|^{-1}\langle\Gamma(q),\tilde{q}_{\eps}',\tilde{q}_{\eps}'\rangle\right)\\
&\eq G(q,\gamma)+|\log\eps|(A(q,\gamma)\tilde{q}'),\notag
\end{align}

where the time derivatives are denoted with a $'$ and all derivatives taken with respect to the rescaled time and space.

We first show that the velocity in rescaled time and space is bounded, until either we approach the boundary or some component of $\tilde{q}$ goes to infinity.

We take $C$ as some large compact subset of $M_\eps$ containing $\tilde{q}_\eps(0)$. 
If $\tilde{q}_\eps\in C$, then this implies that each $q_i$ lies in a compact subset of $\mathbb{H}$. It follows from the definition of $M_\eps$ that it we view the $M_\eps$'s as subsets of $\R^{2k}$ that for small enough $\eps$ such a set $C$ is also a subset of $M_{\eps'}$ for $\eps'<\eps$. Hence $C$ can be chosen as the same set for all small enough $\eps$.

Recall further that the matrix $A$ is invertible by Proposition \ref{conv A} and that its inverse is $\lesssim 1$ as long as $\tilde{q}_\eps\in C$. We may hence rewrite the equation as \begin{equation}\begin{aligned}
&|\log\eps|^3\bigg(\left(E(q)+\mathcal{M}(q)\right)\left(\frac{\text{d}}{\text{d}s}\left(\tilde{q}_\eps'+\frac{A^{-1}G}{|\log\eps|}\right)\right)+\frac{1}{2}\left(\nabla_{\tilde{q}_\eps}E\cdot\tilde{q}_\eps'\right)\left(\tilde{q}_\eps'+\frac{A^{-1}G}{|\log\eps|}\right)\\
&\quad\quad+\frac{1}{2}\left(\nabla_{\tilde{q}_\eps}\mathcal{M}\cdot\tilde{q}_\eps'\right)\left(\tilde{q}_\eps'+\frac{A^{-1}G}{|\log\eps|}\right)\bigg)\\
&-|\log\eps|^3\bigg(\left(E\left(q\right)+\mathcal{M}\left(q\right)\right)\frac{\text{d}}{\text{d}s}\frac{A^{-1}G}{|\log\eps|}+\frac{1}{2}\left(\nabla_{\tilde{q}_\eps}E\cdot\tilde{q}_\eps'\right)\frac{A^{-1}G}{|\log\eps|}+|\log\eps|^{-1}\left\langle \Gamma\left(q\right),\tilde{q}_\eps',\frac{A^{-1}G}{|\log\eps|}\right\rangle\\
&\quad\quad-\left\langle |\log\eps|^{-1}\Gamma\left(q\right)-\frac{1}{2}\nabla_{\tilde{q}_\eps}\mathcal{M}\left(q\right),\tilde{q}_\eps',\tilde{q}_\eps'+\frac{A^{-1}G}{|\log\eps|}\right\rangle\bigg)\\
&=\,|\log\eps|A\left(\tilde{q}_\eps'+\frac{A^{-1}G}{|\log\eps|}\right)
\end{aligned}\end{equation}

where we used the notation $\langle N,a,b\rangle=Na\cdot b$ in the penultimate line.

By testing against $\tilde{q}_\eps'+\frac{A^{-1}G}{|\log\eps|}$ and dividing out the $|\log\eps|^3$ we obtain that from the antisymmetry of $A$ that \begingroup
\allowdisplaybreaks \begin{align}\label{mod energy}
&\frac{\text{d}}{\text{d}s}\frac{1}{2}\left(\left(\tilde{q}_\eps'+\frac{A^{-1}G}{|\log\eps|}\right)^T\left(E(q)+\mathcal{M}(q)\right)\left(\tilde{q}_\eps'+\frac{A^{-1}G}{|\log\eps|}\right)\right)\notag\\
&=\,\left(\tilde{q}_\eps'+\frac{A^{-1}G}{|\log\eps|}\right)^T\left(E\left(q\right)+\mathcal{M}\left(q\right)\right)\frac{\text{d}}{\text{d}s}\frac{A^{-1}G}{|\log\eps|}\notag\\
&+\frac{1}{2}\left(\tilde{q}_\eps'+\frac{A^{-1}G}{|\log\eps|}\right)^T\left(\nabla_{\tilde{q}_\eps}E\cdot\tilde{q}_\eps'\right)\frac{A^{-1}G}{|\log\eps|}\\
&+|\log\eps|^{-1}\left\langle \Gamma\left(q\right),\tilde{q}_\eps',\frac{A^{-1}G}{|\log\eps|}\right\rangle\left(\tilde{q}_\eps'+\frac{A^{-1}G}{|\log\eps|}\right)\notag\\
&-\left\langle|\log\eps|^{-1} \Gamma\left(q\right)-\frac{1}{2}\nabla_{\tilde{q}_\eps}\mathcal{M}\left(q\right),\tilde{q}_\eps',\tilde{q}_\eps'+\frac{A^{-1}G}{|\log\eps|}\right\rangle\left(\tilde{q}_\eps'+\frac{A^{-1}G}{|\log\eps|}\right)\notag\\
&=:\, I+II+III+IV\notag
\end{align}

\endgroup

where $I-IV$ stand for the terms in each line. Our goal is to show that each of these terms is $\lesssim \eps^2(1+|\tilde{q}_\eps'|^2+\eps|\log\eps|^\ell|\tilde{q}_\eps'|^3)$ as long as $\tilde{q}_\eps\in C$.

For $\tilde{q}_\eps\in C$ we have the following estimates:

By the Lemmata \ref{InteriorField} and \ref{conv Inertia}, we have \begin{align}\label{bd EM}
\eps^2\,\lesssim\, E+\mathcal{M}\,\lesssim\, \eps^2
\end{align} 

(in the sense that the smallest and highest eigenvalues have these bounds). Furthermore by Proposition \ref{conv A}  we have \begin{align}\label{bd A}
|A|,|\nabla_{\tilde{q}}A|, |A^{-1}|,|\nabla_{\tilde{q}}A^{-1}|\,\lesssim\, 1.
\end{align}

By the Propositions \ref{diagonal A} and \ref{crossterms G}, we have \begin{align}\label{bd G}
|G|, |\nabla_{\tilde{q}} G|\,\lesssim\, |\log\eps|,
\end{align}

finally by the Propositions \ref{InteriorField} b) and \ref{conv Gamma}, we have \begin{align}\label{bd deri M}
|\nabla_{\tilde{q}_\eps}E|,|\Gamma|\,\lesssim\, \eps^3|\log\eps|^\ell.
\end{align}

Hence we conclude for $\tilde{q}_\eps\in C$ that \begin{align}
&|I|\,\lesssim\,\eps^2(1+|\tilde{q}_\eps'|^2)\\
&|II|,|III|\,\lesssim\,\eps^3|\log\eps|^\ell(1+|\tilde{q}_\eps'|^3).
\end{align}

We have that $IV=0$. Indeed if we set $p:=\tilde{q}_\eps'+\frac{A^{-1}G}{|\log\eps|}$ then, by the definition of $\Gamma$ (\ref{main def}), it holds  \begin{align}
&\langle\Gamma(q),\dot{q},p\rangle p\eq\sum_{i,j,k}\frac{1}{2}(\de_j\mathcal{M}(q)_{ik}+\de_i\mathcal{M}(q)_{jk}-\de_k\mathcal{M}(q)_{ij})\dot{q}_ip_jp_k\\
&=\,\sum_{i,j,k}\frac{1}{2}\de_i\mathcal{M}(q)_{jk}\dot{q}_ip_jp_k\eq\frac{1}{2|\log\eps|}(\nabla_{\tilde{q}_\eps}\mathcal{M}\cdot\dot{q})p\cdot p.
\end{align}

Hence, plugging these estimates into \eqref{mod energy}, we obtain by \eqref{bd EM} that as long as $\tilde{q}_\eps\in C$, we have \begin{align}\label{1st Gronwall}
&\frac{\text{d}}{\text{d}s}\left(\left(\tilde{q}_\eps'+\frac{A^{-1}G}{|\log\eps|}\right)^T\left(E(q)+\mathcal{M}(q)\right)\left(\tilde{q}_\eps'+\frac{A^{-1}G}{|\log\eps|}\right)\right)\,\lesssim\\
&\eps^2+\left(\left(\tilde{q}_\eps'+\frac{A^{-1}G}{|\log\eps|}\right)^T\left(E(q)+\mathcal{M}(q)\right)\left(\tilde{q}_\eps'+\frac{A^{-1}G}{|\log\eps|}\right)\right)\notag\\
&+|\log\eps|^\ell\left(\left(\tilde{q}_\eps'+\frac{A^{-1}G}{|\log\eps|}\right)^T\left(E(q)+\mathcal{M}(q)\right)\left(\tilde{q}_\eps'+\frac{A^{-1}G}{|\log\eps|}\right)\right)^\frac{3}{2}\notag.
\end{align}

As long as we have \begin{align}
\left(\left(\tilde{q}_\eps'+\frac{A^{-1}G}{|\log\eps|}\right)^T\left(E(q)+\mathcal{M}(q)\right)\left(\tilde{q}_\eps'+\frac{A^{-1}G}{|\log\eps|}\right)\right)\leq \eps^\frac{1}{2}
\end{align}

the last term in \eqref{1st Gronwall} can be absorbed in the first two and by Gronwall we obtain that \begin{align}
&\left(\left(\tilde{q}_\eps'+\frac{A^{-1}G}{|\log\eps|}\right)^T\left(E(q)+\mathcal{M}(q)\right)\left(\tilde{q}_\eps'+\frac{A^{-1}G}{|\log\eps|}\right)\right)(s)\\
&\lesssim\, e^s\left(\eps^2+\left(\tilde{q}_\eps'+\frac{A^{-1}G}{|\log\eps|}\right)^T(E(q)+\mathcal{M}(q))\left(\tilde{q}_\eps'+\frac{A^{-1}G}{|\log\eps|}\right)(0)\right).
\end{align}

By \eqref{bd EM} and the assumption about the initial velocities, this implies that \begin{align}
\left|\tilde{q}_\eps'+\frac{A^{-1}G}{|\log\eps|}\right|\,\lesssim\, e^s
\end{align}

and hence \begin{align}
|\tilde{q}_\eps'|\,\lesssim\, e^s
\end{align}

until either $\tilde{q}_\eps$ leaves the set $C$ or up to a time of order $|\log\eps|$. By \eqref{bd A}, \eqref{bd G} and \eqref{rescaled ODE}, this implies that \begin{align}
|\log\eps|^2\left(E(q)\tilde{q}_{\eps}''+\frac{1}{2}\tilde{q}_{\eps}'(\nabla_{\tilde{q}_\eps}E(q)\cdot\tilde{q}_{\eps}')+\mathcal{M}(q)\tilde{q}_{\eps}''+|\log\eps|^{-1}\langle\Gamma(q),\tilde{q}_{\eps}',\tilde{q}_{\eps}'\rangle\right)\rightarrow 0 
\end{align}

in $W^{-1,\infty}$ up to a time of order $|\log\eps|$ or until $\tilde{q}_\eps$ leaves $C$. Hence we obtain that \begin{align}
A\tilde{q}_\eps'+\frac{G}{|\log\eps|}\overset{\ast}{\rightharpoonup} 0 \text{ in $L^\infty$}.
\end{align}

Because $A^{-1}$ and $\frac{G}{|\log\eps|}$ converge strongly, we see by Corollary \ref{conv J1} that \begin{align}
\tilde{q}_\eps'- J_\gamma^1(\tilde{q}_\eps)\overset{\ast}{\rightharpoonup}0 \text{ in $L^\infty$}
\end{align}

until $\tilde{q}_\eps$ leaves $C$, which takes at least $\gtrsim 1$ time, as $\tilde{q}_\eps'$ is bounded.

Hence we have that \begin{align}
\tilde{q}_\eps(v)-\tilde{q}_\eps(0)-\int_0^v J_\gamma^1(\tilde{q}_\eps(s))\ds\rightarrow 0
\end{align}

for all times $v\lesssim 1$ and $\tilde{q}_\eps$ converges locally uniformly to some $\tilde{q}_0$ by compactness, as long as long as $\tilde{q}_\eps$ lies in $C$. Therefore we see that  $\tilde{q}_0$ must be a solution of $\tilde{q}_0'=J_\gamma^1(\tilde{q}_0)$ because $J_\gamma^1$ is locally uniformly continuous.\\

Finally, we may remove the condition that $\tilde{q}_\eps$ lies in a compact set $C$, by taking $C$ so large that the solution of $q'=J_\gamma^1(q)$ lies in the interior of $C$ until some time $T$, which is possible for small enough $\eps$ whenever $q$ does not blow until time $T$. Then we have uniform convergence of $\tilde{q}_\eps$ as long as it lies in $C$. As the limit lies in the interior of $C$, the solution $\tilde{q}_\eps$ also lies in $C$ for small enough $\eps$ up to time $T$.

Hence we have convergence, as long as the limiting solution does not blow up.

\end{proof}

\begin{proof}[Proof of Thm.\ \ref{Limit2}]
The proof is quite similar to the previous one. In the rescaled time $s=|\log\eps|t$, and the rescaled spatial variable $\tilde{q}_\eps$, defined as in \eqref{def regime1} the System \eqref{MainODE} reads as 
\begin{align}\label{rescaled ODE 2}
&|\log\eps|^\frac{3}{2}\left(E(q)\tilde{q}_{\eps}''+\frac{1}{2}\tilde{q}_{\eps}'(\nabla_{\tilde{q}_\eps}E(q)\cdot\tilde{q}_{\eps}')+\mathcal{M}(q)\tilde{q}_\eps''+|\log\eps|^{-\frac{1}{2}}\langle\Gamma(q),\tilde{q}_{\eps}',\tilde{q}_{\eps}'\rangle\right)\\
&\eq G(q,\gamma)+|\log\eps|^\frac{1}{2}(A(q,\gamma)\tilde{q}_{\eps}').\notag
\end{align}

Similarly as in the previous proof we can rewrite the equation, tested against $\tilde{q}_\eps'+\frac{A^{-1}G}{|\log\eps|^\frac{1}{2}}$ as \begingroup
\allowdisplaybreaks
 \begin{align}
&\frac{\text{d}}{\text{d}s}\frac{1}{2}\left(\left(\tilde{q}_\eps'+\frac{A^{-1}G}{|\log\eps|^\frac{1}{2}}\right)^T\left(E(q)+\mathcal{M}(q)\right)\left(\tilde{q}_\eps'+\frac{A^{-1}G}{|\log\eps|^\frac{1}{2}}\right)\right)\notag\\
&=\,\left(\tilde{q}_\eps'+\frac{A^{-1}G}{|\log\eps|^\frac{1}{2}}\right)^T\left(E\left(q\right)+\mathcal{M}\left(q\right)\right)\frac{\text{d}}{\text{d}s}\left(\frac{A^{-1}G}{|\log\eps|^\frac{1}{2}}\right)\notag\\
&+\frac{1}{2}\left(\tilde{q}_\eps'+\frac{A^{-1}G}{|\log\eps|^\frac{1}{2}}\right)^T\left(\nabla_{\tilde{q}_\eps}E\cdot\tilde{q}_\eps'\right)\left(\frac{A^{-1}G}{|\log\eps|^\frac{1}{2}}\right)\\
&+|\log\eps|^{-\frac{1}{2}}\left\langle \Gamma\left(q\right),\tilde{q}_\eps',\frac{A^{-1}G}{|\log\eps|^\frac{1}{2}}\right\rangle\left(\tilde{q}_\eps'+\frac{A^{-1}G}{|\log\eps|^\frac{1}{2}}\right)\notag\\
&-\left\langle |\log\eps|^{-\frac{1}{2}}\Gamma\left(q\right)-\frac{1}{2}\nabla_{\tilde{q}_\eps}\mathcal{M}\left(q\right),\tilde{q}_\eps',\tilde{q}_\eps'+\frac{A^{-1}G}{|\log\eps|^\frac{1}{2}}\right\rangle\left(\tilde{q}_\eps'+\frac{A^{-1}G}{|\log\eps|^\frac{1}{2}}\right)\notag\\
&=:\, I+II+III+IV\notag.
\end{align}

\endgroup

Let $v_Z$ denote the vector $(e_Z,e_Z,\dots)\in \R^{2k}$. We would like to estimate the shifted velocity $\tilde{q}_\eps'+\frac{|\log\eps|^\frac{1}{2}}{4\pi R_0}v_Z$.

We again use a compact set $C\subset M_\eps$ containing $\tilde{q}_\eps(0)$. Then we let $\tilde{C}:=C+v_Z\R$. On this set we still have uniform estimates because the system is invariant in the $v_Z$ direction. If $\tilde{q}_\eps\in \tilde{C}$, then we clearly we still have the estimates \eqref{bd EM} and \eqref{bd deri M}. Furthermore by the Propositions \ref{diagonal A} and \ref{crossterms G} we then also have \begin{align}
|G|\,\lesssim\, |\log\eps|
\end{align}

and \begin{align}
|\nabla_{\tilde{q}_\eps}G|\,\lesssim\, |\log\eps|^\frac{1}{2}.
\end{align}

Furthermore by Proposition \ref{conv A} we have \begin{align}
|A|\,\lesssim\, 1,\quad |\nabla_{\tilde{q}_\eps}A|\,\lesssim\, |\log\eps|^{-\frac{1}{2}}
\end{align}

and from Corollary \ref{conv J2} one sees that \begin{align}
\left|\frac{A^{-1}G}{|\log\eps|^\frac{1}{2}}-\frac{|\log\eps|^\frac{1}{2}}{4\pi R_0}v_Z\right|\,\lesssim\, 1.
\end{align}

Hence, we directly see that \begin{align}
|II|,|III|\,\lesssim\, \eps^3|\log\eps|^\ell\left(1+\left|\tilde{q}_\eps'+\frac{|\log\eps|^\frac{1}{2}}{4\pi R_0}v_Z\right|^3\right).
\end{align}

The term $IV$ again drops out by the same calculation as in the previous proof. Note that we have \begin{align}
\frac{\text{d}}{\text{d}s}\left(\frac{A^{-1}G}{|\log\eps|^\frac{1}{2}}\right)\eq\frac{1}{|\log\eps|^\frac{1}{2}}\left(\left(\nabla_{\tilde{q}_\eps}A^{-1}\right)G+A^{-1}\nabla_{\tilde{q}_\eps}G\right)\left(\tilde{q}_\eps'+\frac{|\log\eps|^\frac{1}{2}}{4\pi R_0}v_Z\right)
\end{align}

because the derivative of $A^{-1}G$ in the $v_Z$ direction is $0$, as the system is invariant in that direction.

Hence we see that
 \begin{align}
&|I|\,\lesssim\,\frac{\eps^2}{|\log\eps|^\frac{1}{2}}\left(\left|\tilde{q}_\eps'+\frac{|\log\eps|^\frac{1}{2}}{4\pi R_0}v_Z\right|+\left|\frac{A^{-1}G}{|\log\eps|^\frac{1}{2}}-\frac{|\log\eps|^\frac{1}{2}}{4\pi R_0}v_Z\right|\right)\\
&\quad\times\left(|A^{-1}|^2|\nabla_{\tilde{q}_\eps}A||G|+|A^{-1}||\nabla_{\tilde{q}_\eps}G|\right)\left|\tilde{q}_\eps'+\frac{|\log\eps|^\frac{1}{2}}{4\pi R_0}v_Z\right|\\
&\lesssim\,\eps^2\left(1+\left|\tilde{q}_\eps'+\frac{|\log\eps|^\frac{1}{2}}{4\pi R_0}v_Z\right|^2\right).
\end{align}

From the assumption that \begin{align}
\left|\tilde{q}_\eps'(0)+\frac{|\log\eps|^\frac{1}{2}}{4\pi R_0}v_Z\right|\,\lesssim\, 1
\end{align}

we see by the same Gronwall argument as in the previous proof that \begin{align}
\left|\tilde{q}_\eps'(s)+\frac{|\log\eps|^\frac{1}{2}}{4\pi R_0}v_Z\right|\,\lesssim\, e^s
\end{align}

until $\tilde{q}_\eps$ leaves $\tilde{C}$ or until a time of order $|\log\eps|$.

From this, we conclude that $\tilde{q}_\eps''$ is bounded in $W^{-1,\infty}$ and by \eqref{bd EM} and \eqref{bd deri M} we see that \begin{align}
|\log\eps|^\frac{3}{2}\left(E(q)\tilde{q}_{\eps}''+\frac{1}{2}\tilde{q}_{\eps}'(\nabla_{\tilde{q}_\eps}E(q)\cdot\tilde{q}_{\eps}')+\mathcal{M}(q)\tilde{q}_\eps''+|\log\eps|^{-\frac{1}{2}}\langle\Gamma(q),\tilde{q}_{\eps}',\tilde{q}_{\eps}'\rangle\right)\rightarrow 0
\end{align}

in $W^{-1,\infty}$ until $\tilde{q}_\eps$ leaves $\tilde{C}$. Hence we see again that \begin{align}
\tilde{q}_\eps'+\frac{|\log\eps|^\frac{1}{2}}{4\pi R_0}v_Z+\left( \frac{A^{-1}G}{|\log\eps|^\frac{1}{2}}-\frac{|\log\eps|^\frac{1}{2}}{4\pi R_0}v_Z\right)\,\overset{\ast}{\rightharpoonup}\, 0 \text{ in $L^\infty$}.
\end{align}

This implies the statement by the same argument as in the previous proof and Corollary \ref{conv J2}.
\end{proof}

\textbf{Acknowledgment:} This work is funded by the Deutsche Forschungsgemeinschaft (DFG, German Research Foundation) under Germany’s Excellence Strategy EXC 2044 –390685587, Mathematics Münster: Dynamics–Geometry–Structure.

The author would like to thank Christian Seis for introducing him to the problem and both Christian Seis and Franck Sueur for some useful discussions and advice.

\bibliography{rings}
\bibliographystyle{abbrv}

\end{document}